\documentclass{article}%

\usepackage{amsmath}%
\usepackage{amsfonts}%
\usepackage{amssymb}%
\usepackage{graphicx}
\usepackage{fullpage}
\usepackage{palatino}
\usepackage[english]{babel}
\usepackage{tikz}
\usepackage{subcaption}
\usepackage[margin=2cm]{caption}
\usepackage{dsfont}
\usepackage[bookmarksnumbered=true]{hyperref}

\usepackage{ntheorem}

\newtheorem{theorem}{Theorem}[section]

\newtheorem{assumption}[theorem]{Assumption}

\newtheorem{corollary}[theorem]{Corollary}

\newtheorem{definition}[theorem]{Definition}

\newtheorem{lemma}[theorem]{Lemma}

\newtheorem{proposition}[theorem]{Proposition}
\newtheorem{remark}[theorem]{Remark}

\newenvironment{proof}[1][Proof]{\textbf{#1.} }{\ \rule{0.5em}{0.5em}}

\numberwithin{equation}{section}

\usepackage{amsmath}%
\usepackage{amsfonts}%
\usepackage{amssymb}%
\usepackage{graphicx}
\usepackage{booktabs}
\usepackage{bbm}
\usepackage[english]{babel}
\usepackage{tikz}
\usepackage{verbatim}
\usetikzlibrary{plotmarks}
\tikzset{
    scale plot marks/.is choice,
    scale plot marks/false/.code={
        \def\pgfuseplotmark##1{\pgftransformresetnontranslations\csname pgf@plot@mark@##1\endcsname}
    },
    scale plot marks/true/.style={},
    scale plot marks/.default=true
}
\usepackage[ruled, linesnumbered]{algorithm2e}

\usepackage{dsfont}
\usepackage[bookmarksnumbered=true]{hyperref}

\newcommand{\Prob}{\mathbb{P}}
\newcommand{\E}{\mathbb{E}}

\DeclareMathOperator*{\argmax}{arg\,max}
\DeclareMathOperator*{\argmin}{arg\,min}

\title{Complete resource pooling of a load balancing policy for a network of battery swapping stations}

\makeatletter
\let\@fnsymbol\@arabic
\makeatother

\author{Fiona Sloothaak\thanks{Eindhoven University of Technology, Eindhoven, Netherlands.} \and James Cruise\thanks{Riverlane, Cambridge, United Kingdom.} \and Seva Shneer\thanks{Heriot-Watt University, Edinburgh, United Kingdom.}~\textsuperscript{,}\thanks{Novosibirsk State University, Russia} \and Maria Vlasiou\footnotemark[1]~\textsuperscript{,}\thanks{University of Twente, Enschede, Netherlands.} \and Bert Zwart\footnotemark[1]~\textsuperscript{,}\thanks{Centrum Wiskunde \& Informatica, Amsterdam, Netherlands.}}

\date{\vspace{-1ex}}

\begin{document}
\maketitle
\begin{abstract}
		To reduce carbon emission in the transportation sector, there is currently a steady move taking place to an electrified transportation system. This brings about various issues for which a promising solution involves the construction and operation of a battery swapping infrastructure rather than in-vehicle charging of batteries. In this paper, we study a closed Markovian queueing network that allows for spare batteries under a dynamic arrival policy. We propose a provisioning rule for the capacity levels and show that these lead to near-optimal resource utilization, while guaranteeing good quality-of-service levels for Electric Vehicle (EV) users. Key in the derivations is to prove a state-space collapse result, which in turn implies that performance levels are as good as if there would have been a single station with an aggregated number of resources, thus achieving complete resource pooling.
\end{abstract}

	\maketitle
	
	%
	
	
\section{Introduction}
A key challenge in the deployment and take up of electric vehicles by society is the provision of a scalable charging infrastructure. A viable solution is the development of a battery swapping network. Currently, there has been work done on the operation and control of a single battery swapping stations (for example \cite{TanPerformance2018}), but there is a clear gap within the literature when extending this to the operation a wider network of stations. In this paper, we introduce a novel stochastic network model describing a network of battery swapping stations which clearly addresses this need and provides a foundation for future studies. In addition, we carry out a detailed analysis of this model and obtained a number of novel insights into the operation of a battery swapping network.
	
	A steady energy transition is taking place due to the de-carbonization of the economy, leading to many intrinsic challenges and research opportunities, of which an overview is given in~\cite{NAP2016} and~\cite{Bienstock2015}. There are numerous challenging problems caused by developments on the demand side. Examples include control problems in local, smart distribution grids, as well as managing increasing demand irregularities caused by e.g.~electric vehicles. Modeling the behavior of individual agents and their interaction naturally leads to stochastic models.

	Despite the apparent need of alternative energy sources in the transportation sector, the adoption of electrified vehicles has been slow initially due to various practical challenges, such as high purchase costs of an EV, battery life problems and long battery charging times~\cite{SunTanTsang2014}. A possible solution to address these issues is the construction and operation of a battery swapping infrastructure. The upfront costs of purchase of an EV can be significantly reduced when battery swapping station operators own and lease batteries to customers, the batteries can be charged more appropriately to prolong batteries' lifetime~\cite{TanPerformance2018}, and EV users can experience a fast exchange of batteries in contrast to long charging times. Beyond the consumer benefits, the centralized charging paradigm of battery swapping allows the deferment of huge network reinforcement works required to support charging at home by connecting the chargers to the medium voltage network.  Furthermore, the aggregation of a large number of batteries at charging stations can provide a comprehensive range of flexibility services to transmission and distribution network service operators.
	
	In this paper, we introduce a model for EVs utilizing battery swapping technology within the context of a fixed region. Within the region there are a number of charging/swapping stations and vehicles, in general, do not leave the region leading to the conservation of batteries. This leads us to introduce a class of closed Markovian queueing network model, which we use in a novel way to model the evolution of the battery population within a city.
	
	With the advancement of smartphones and online technologies, a range of service providers will utilize these advancements to provide occupancy level information to customers to improve delay performance. In a battery swapping system, such information can motivate EV users to visit the most appealing location in the direct vicinity. In this paper, we integrate a \textit{load-balancing} policy to incorporate this. An intrinsic problem is to establish suitable capacity levels that account for the inherent tradeoff between EV users' quality-of-service and operational costs. To the best of our knowledge, this is the first work that considers this question for a battery swapping system in a network framework under a dynamic arrival policy.

	Adequately balancing service performance and resource utilization is very much in the spirit of the \textit{Quality-and-Efficiency-Driven (QED) regime} known from asymptotic many-server queueing theory~\cite{HalfinWhitt1981}. Typically, this gives rise to a square-root slack provisioning policy for the capacity levels and has been successfully implemented in many applications such as call centers~\cite{Borst2004,Khudyakov2010,ZhangJvLZwart2012}, healthcare systems~\cite{Jennings2011,YomTov2014,JvLMatSloYT2016} and more. This policy leads to favorable performance for large systems: as the number of customers $r$ grows large, the waiting probability tends to a value strictly between zero and one, the waiting time vanishes with a rate $1/\sqrt{r}$, and near-optimal resource utilization of $1-O(1/\sqrt{r})$ is achieved. To inherit such properties for the battery swapping framework, we adopt a similar capacity level design policy for both the number of charging servers and the number of spare batteries relative to the expected offered load under the load-balancing arrival strategy.

	To add to the agreeable properties of delay performance in the QED regime, the arrival strategy ensures that the relative charging loads at the different stations do not grow apart too much since arriving EV users always move to the least loaded station. This phenomenon has been observed in a number of settings and is referred to as state space collapse, see~\cite{MR1663763,Williams1998} for an overview and~\cite{DaiTezcan2011} for work most closely related to this paper.  In fact, when capacity levels are chosen appropriately, this effect is so strong that complete resource pooling takes place: the system behaves as if there is only a single station with an aggregated number of resources. It ensures that it is unlikely that EV users are waiting for a battery at one station, while another is readily available at any other station, even among those stations that are far from his direct vicinity.
	
	The first main contribution of this paper is the introduction of a stochastic model for battery charging in a network setting. In recent years, there has been a growing amount of research on both the planning/design as well as the operation/scheduling in battery swapping systems, see~\cite{TanPerformance2018} for an overview. Most papers employ robust optimization techniques to find optimal solutions for certain objectives, while little of the works focus on the quality-of-service for EV users. The exception are a collection of papers written by a set of authors~\cite{SunTanTsang2014,TanSunTsang2014,SunTanTsang2018,TanPerformance2018,SunWhitt2018}, that use asymptotic analysis and Markov Decision Process techniques to propose suitable solutions. Whereas the focus in those papers is on issues arising in a single station, we propose a network setting to account for queue length correlations between stations.
	
	Our second main contribution involves the novelty of our load-balancing arrival mechanism. Load-balancing policies have attracted a lot of attention in recent years due to extremely relevant applications in large data centers, see~\cite{BoorMukherjee2018} for an overview. Typically, these systems comprise many single-server stations where a central dispatcher decides where to allocate incoming tasks. In contrast, our framework involves a network of (a fixed number of) multi-server stations for which we introduce a unique feature: an arriving EV user restricts itself to move only to one of the stations in his direct vicinity. By appropriately setting the capacity levels according to the QED provisioning rule, we show that this constraint turns redundant in the sense that the resource pooling effect can still be achieved.
	
	In this paper, we also make several theoretical contributions. Direct analysis of the steady-state distribution of the queue-length process is intractable under the load-balancing strategy in case of multiple stations. Instead, we resort to a fluid and diffusion limit approach. We derive the existence of the fluid limit and point out its unique invariant state. Using a diffusion-scaled queue length process, we zoom in on the fluctuations around the invariant state. We prove a state space collapse (SSC) result by showing that in the limit (as the number of EVs grows larger) the diffusion-scaled queue lengths tend to become arbitrarily close almost instantaneously and stay that way for any fixed interval. This property can be exploited to derive the limiting queue length behavior at every station, and show that it implies the complete resource pooling effect. The derivations of our results rely heavily on the framework developed by Dai and Tezcan~\cite{DaiTezcan2011}, that in turn can be seen as an extension of~\cite{Bramson1998}. We adapt their framework to incorporate a closed network setting under the novel load-balancing policy.

	The introduction of the novel framework within this paper acts as a foundation for a substantial research programme in the modeling of battery swapping networks. This will provide practitioners with a better understanding of how such networks should be designed and operated from both the perspective of quality of service requirements but also from an economic viewpoint.
	This can be carried out by enriching the model, here we highlight a few possible directions we consider important and challenging future steps.  Each of these will provide a detailed insight a specific aspect of such systems. Firstly, the inclusion of multiple customer types to model a range of car brands within the network using different battery systems. Secondly, there is a delay between the moment an EV user consults queue length information and the actual arrival due to transportation time. As is perceived in health care settings and bike-sharing systems, this can have a considerable effect on the queue length behavior. A third enhancement would be to incorporate a time-inhomogeneous demand rate to better simulate the expected diurnal variation. This will lead to a varying amount of slackness in the capacity within the QED regime. Finally, there is substantial underlying variability in the fluctuating energy prices, which sharply rise whenever the energy grid is more strained and vice versa. A battery swapping infrastructure will be sensitive to these prices changes and can provide an indispensable asset for supporting a stable grid in the future. Since it can relieve strain during peak moments by deferring the moment of charging or even deplete batteries providing energy to the grid. It is beyond the scope of this paper to provide efficient and adequate provisioning rules in these challenging settings, yet offers intriguing avenues to pursue in future research.
	The main insight provided in the present study is the effectiveness of simple load balancing policies, and while the model is parsimonious, this insight is useful in, at least, the planning stage of a swapping network.

	The remainder of this paper is organized as follows. In Section~\ref{sec:ModelDescription} we describe the battery swapping network and its corresponding load-balancing arrival mechanism. In Section~\ref{sec:SingleStation} we present the fluid and diffusion results in the special case of a single station, and generalize these results for the multiple stations setting in Section~\ref{sec:MultipleStations}. Our results imply approximations for certain performance measures, which we validate through several simulation experiments described in Section~\ref{sec:SimulationExperiments}.
	
	\section{Model Description}\label{sec:ModelDescription}
	In this paper, we consider a queueing network with $S$ battery swapping stations and $r$ EVs. Each EV has one battery (collection) providing the energy for the car to drive. Every station $i \in \{1,\ldots,S\}$ has three types of assets: $F_i$ \textit{charging points}, $B_i$ \textit{spare batteries} and $G_i$ \textit{swapping servers}. Whenever there is an EV arrival at a station, a swapping server takes out the almost depleted battery and exchanges it for a fully-charged one if available. The swapping time is relatively very short (with respect to charging times), and therefore we assume it to occur instantaneously. Batteries in need of charging are being recharged whenever a charging point is available, and we assume every recharge to take an exponential amount of time with rate $\mu$, independent of everything else. Whenever a battery is fully charged, it is placed in an EV immediately if one is waiting, and otherwise stocked for a future EV arrival. After receiving a fully-charged battery, the EV requires recharging after an exponential amount of time with rate $\lambda$. With probability $p_{ij}$ stations $i$ and $j$ are in the EV user's direct vicinity. We assume that EV users consult some online device, and are motivated to move to the station that is relatively least loaded (ties are broken evenly). We define which station is relatively least loaded more precisely later in this section. Figure~\ref{fig:MultipleStationIllustration} illustrates the closed queueing model under this load-balancing arrival mechanism.
	
	\begin{figure}
		\centering
		\includegraphics[width=10cm]{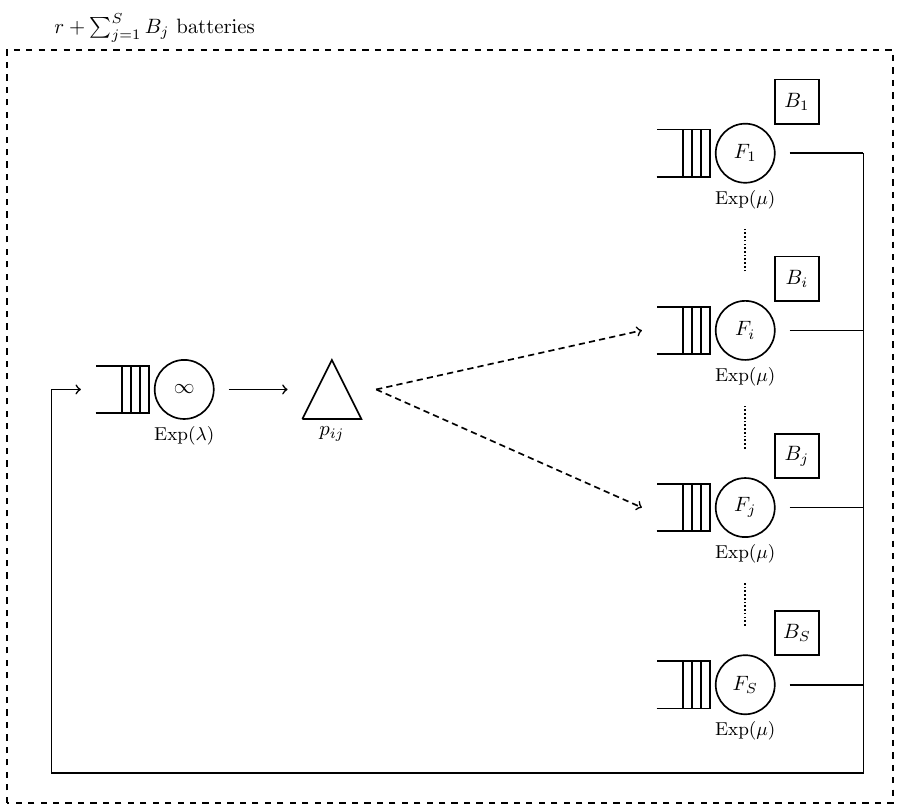}
		\caption{Illustration of closed queueing network with multiple stations.}
		\label{fig:MultipleStationIllustration}
	\end{figure}
	
	We point out that batteries are always exchanged, and therefore the number of batteries physically present at a station can never be below this station's number of spare batteries. In fact, this observation implies that the queueing model is closed, where the total number of batteries is given by
	\begin{align*}
	\textrm{Total \# batteries in system} = r + \sum_{j=1}^S B_i.
	\end{align*}
	Another observation concerns the role of the swapping servers. Whenever a battery is taken out of the EV, it cannot move from the swapping server until an exchange of batteries has taken place. Therefore, no more than $B_i+G_i$ batteries can be charged simultaneously at a station $i \in \{1,\ldots,S\}$. As a consequence, having more charging points creates no additional charging capacity, and can be bounded as
	\begin{align}
	F_i \leq B_i + G_i, \quad i=1,\ldots,S.
	\label{eq:ChargingPointsLessThanSpareBatteries}
	\end{align}
	In addition, we assume that the number of such expensive swapping technologies is small at every station, i.e.~$G_i < G$ for all $i=1,\ldots,S$, with $G< \infty$ being a small fixed number.
	
	The main quantity of interest in this paper is the number of batteries that are in need of charging, i.e.\ the aggregated number of batteries that are being charged at a charging point and the possible exchanged batteries that are waiting for an available charging point. We also refer this quantity as to as the queue length. Let $Q_i(t)$ denote the number of batteries in need of charging at station $i$ at time $t \geq 0$, and we write $Q(t)=(Q_1(t),\ldots,Q_S(t))$. Besides the queue length process, we focus on three performance measures in this paper: the waiting probability of an arbitrary EV, its expected waiting time and the resource utilization levels of the stations. As the role of swapping servers is non-existent in this framework, we consider the resources of the swapping stations to be the charging points and the spare batteries. We define the utilization level of the charging points to be the fraction of charging points that are busy with charging, and the utilization level of the spare batteries to be the fraction of batteries that are not fully-charged with respect to the total number of batteries at the station. In steady state, the latter corresponds to the fraction of time at a station that a battery is expected to wait for an arriving EV.
	
	To achieve favorable performance levels, we propose an associated QED-scaled capacity level for the resources at the stations. More specifically, we consider a sequence of systems indexed by the number of cars $r$, where we write a superscript $r$ for processes and quantities to stress the dependency on $r$. Under the policy where every arrival would choose randomly between the two stations in its direct vicinity, we observe that $p_i = \sum_{j=1}^S p_{ij}/2$ represents the effective arrival probability for every station~$i=1,\ldots,S$. Therefore, for a system with $r$ cars, we set the capacity levels of the number of charging points and the number of spare batteries as
	\begin{align}
	\left\{ \begin{array}{ll}
	B_i^r = p_i \left(\frac{\lambda r}{\mu} + \beta \sqrt{\frac{\lambda r}{\mu}} \right), & \beta \in \mathbb{R}, \\
	F_i^r = p_i \left(\frac{\lambda r}{\mu} + \gamma \sqrt{\frac{\lambda r}{\mu}} \right), & \gamma \leq \beta,
	\end{array}\right.
	\label{eq:QEDscalingNetworkModel}
	\end{align}
	for all $i =1,\ldots,S$. We remark that the bound for the number of charging points originates from~\eqref{eq:ChargingPointsLessThanSpareBatteries}. Since the number of swapping servers is fixed and small and the number of cars $r$ grows large, this condition reduces to the $\gamma \leq \beta$ requirement in~\eqref{eq:QEDscalingNetworkModel}.
	
	Since there are two types of resources at every station, i.e.~charging points and spare batteries, one can consider two types of utilization levels. However, in view of~\eqref{eq:QEDscalingNetworkModel}, we see that the capacity levels of both resources are of the magnitude~$p_i \lambda r/\mu + O(\sqrt{r})$, and hence the utilization levels of both resources are given by $Q(t)/(p_i \lambda r/\mu) (1+o(1))$. Using this observation, we define the relative occupancy level (load) of a station as $Q_i(t)/p_i$. We let our load-balancing policy prescribe that an EV in need of charging closest to station $i$ and $j$ at time $t\geq 0$ moves to station~$i$ iff
	\begin{align}
	\frac{Q_i(t)}{p_i} < \frac{Q_j(t)}{p_j},
	\label{eq:ArrivalMoveLoadBalancing}
	\end{align}
	where ties are broken evenly.  In our results, we show that this load-balancing policy ensures that the resource utilization levels at the different stations are approximately equal at all times. Consequently, this also ensures that the expected waiting times are approximately the same at every station at all times.

		\begin{remark}\label{rem:BetterScaling}\normalfont
	Our modeling prescribes that every EV user can choose between two stations in its direct vicinity. We point out that this is done for simplicity, as it helps to describe our scaling and load balancing-policy in a clear and concise manner. We point out that our model and results extends naturally to the cases where some EV arrivals may always move to one station, and some EV arrivals choose from multiple stations. With respect to the modeling, this extension can be included as follows. Let $\mathcal{M}$ be the set of arrival types, where every $m \in \mathcal{M}$ is a set of stations that is in the direct vicinity of the EV user. Let $s_m, m \in \mathcal{M}$ denote the probability that an EV arrival is of type~$m$. Then, the effective arrival rate at any station~$i \in \{1,...,S\}$ is given by
	\begin{align*}
		p_i = \sum_{m \in \mathcal M} \mathds{1}_{\{i \in m\}} \; s_m/|m|.
	\end{align*}
	In this extended setting, the scaling~\eqref{eq:QEDscalingNetworkModel} for the number of resources and the load-balancing policy~\eqref{eq:ArrivalMoveLoadBalancing} remains the same.
	\end{remark}

	\section{System behavior in case of a single swapping station}\label{sec:SingleStation}
	When there is only a single battery swapping station, all EVs simply move to this station with probability one. The system reduces to a closed network where batteries are in two possible locations: either positioned in a car in no need of charging, or at the station. An illustration of the closed queueing model is given in Figure~\ref{fig:SingleStationIllustration}. The square-root scaling rules reduces to
	\begin{align}
	\begin{split}
	B^r &= \frac{\lambda r}{\mu} + \beta \sqrt{\frac{\lambda r}{\mu}}, \quad \beta \in \mathbb{R}, \\
	F^r &= \frac{\lambda r}{\mu} + \gamma \sqrt{\frac{\lambda r}{\mu}}, \quad \gamma \leq \beta,
	\end{split}
	\label{eq:QEDscalingSingleBSSModel}
	\end{align}
	where we suppress the subscript $1$ for the station number in this case.
	
	\begin{figure}[htb]
		\centering
		\includegraphics[width=10cm]{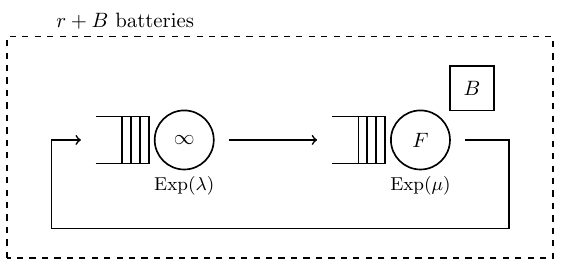}
		\caption{Illustration of the closed queueing network with a single stations.}
		\label{fig:SingleStationIllustration}
	\end{figure}
	
	The notable advantage of a single station is that all resources are assembled at one entity, and inherently, no resources are unavailable for being at different locations. There is also a considerable upside in terms of the analysis: since there is no routing policy anymore, the queue length process becomes a simple birth-death process for which the steady state distribution is easily derived. Yet, the steady-state distribution provides little qualitative insight in the queue length behavior, and in particular, the behavior of the process when it has not reached steady state yet. Therefore, we resort to fluid and diffusion limits, which in practice serve as good approximations for moderate to large-scale systems. This allows us to provide approximations for the performance measures of our interest, e.g.~the waiting probability and the expected waiting time.
	
	At first glance, the single-station variant of our model may seem similar to the classic repair man model. This model and its QED-scaling implications are thoroughly treated in~\cite{Jennings2008,Jennings2011}, which mainly focus on the healthcare setting. We point out that there is a crucial difference: our single-station model includes spare batteries, causing none of $r$ cars to be waiting at the station as long as there are sufficient fully-charged spares available. If $B=0$, our model reduces to the repair man model with $r$ machines and $F$ repair men. Generally, however, the birth rates are different.
	
	\subsection{Steady state distribution}
	As the queue length process is a birth-death process, it is straightforward to derive the steady-state distribution of the queue-length process by standard theory for Markov chains, irrespective of whether the QED~scaled provisioning rules~\eqref{eq:QEDscalingSingleBSSModel} hold. More specifically, the queue length $\{Q(t),t\geq 0\}$ is a birth-death process with state space $Q(t) \in \{0,1,\ldots,B^r+r\}$ for all $t \geq 0$, with birth rate $\lambda (r - \left(Q(t)-B^r\right)^+)$ and death rate $\mu \min\{Q(t),F^r\}$. Let
	\begin{align*}
	\pi_k^{(B^r,F^r,r)} = \Prob\left(Q(\infty) = k \right)
	\end{align*}
	denote the steady state distribution of the number of batteries in need of charging.
	
	\begin{lemma}
		Suppose $S=1$, where the single swapping station has $F$ charging points and $B$ spare batteries, i.e.\ we disregard the scaling in\eqref{eq:QEDscalingSingleBSSModel}. The steady state distribution is given by
		\begin{align}
		\pi_k^{(B,F,r)} = \left\{ \begin{array}{ll}
		\frac{(\lambda r / \mu)^k}{k!} \pi_0^{(B,F,r)} & \textrm{if } 0 \leq k \leq \min\{B,F\}, \\
		\frac{\left(\lambda r /\mu \right)^k }{F!F^{k-F}} \pi_0^{(B,F,r)} & \textrm{if } F \leq k \leq B, \\
		\frac{r^{B} r!}{(r+B-k)!} \frac{\left(\lambda/\mu \right)^k }{k!} \pi_0^{(B,F,r)} & \textrm{if } B \leq k \leq F,\\
		\frac{r^{B} r!}{(r+B-k)!} \frac{\left(\lambda/\mu \right)^k }{F!F^{k-F}} \pi_0^{(B,F,r)} & \textrm{if } \max\{B,F\}  \leq k \leq B+r,
		\end{array} \right.
		\label{eq:StDistributionLimitedChargingPoints}
		\end{align}
		where
		\begin{align}
		\pi_0^{(B,F,r)} = \left(\sum_{k=0}^{F} \frac{(\lambda r / \mu)^k}{k!} + \sum_{k=F+1}^{B-1} \frac{\left(\lambda r /\mu \right)^k }{F!F^{k-F}} + \sum_{k=B}^{B+r} \frac{r^{B} r!}{(r+B-k)!} \frac{\left(\lambda/\mu \right)^k }{F!F^{k-F}} \right)^{-1}
		\label{eq:StDistributionNormalizationConstantLimitedChargingPoints1}
		\end{align}
		if $F \leq B$, and
		\begin{align}
		\pi_0^{(B,F,r)} = \left(\sum_{k=0}^{B} \frac{(\lambda r / \mu)^k}{k!} + \sum_{k=B+1}^{F-1} \frac{r^{B} r!}{(r+B-k)!} \frac{\left(\lambda/\mu \right)^k }{k!} + \sum_{k=F}^{B+r} \frac{r^{B} r!}{(r+B-k)!} \frac{\left(\lambda/\mu \right)^k }{F!F^{k-F}}  \right)^{-1}
		\label{eq:StDistributionNormalizationConstantLimitedChargingPoints2}
		\end{align}
		if $B \leq F$.
		\label{lem:SteadyStateSingleBSS}
	\end{lemma}
	
	\begin{remark}\normalfont
		In view of~\eqref{eq:ChargingPointsLessThanSpareBatteries}, we exclude the case that $F \geq B$ in our analysis further on in this paper. Yet, in an application where e.g.~$G = \infty$ and hence $F \geq B$ possibly holds, we point out that the distribution can be derived similarly. That is, all EVs that arrive at the station find an available swapping server, and the swapping servers do not pose any restriction on the number of batteries that can be charged simultaneously. Only the number of charging points bounds the charging rate. One can also consider the QED provisioning rule in this case, which we treat in Appendix~\ref{app:UnlimitedGSingle}. Moreover, if both $G=\infty$ and $F = \infty$, Lemma~\ref{lem:SteadyStateSingleBSS} shows that
		\begin{align}
		\pi_k^{(B,r)} = \left\{ \begin{array}{ll}
		\frac{(\lambda r / \mu)^k}{k!} \pi_0^{(B,r)} & \textrm{if } 0 \leq k \leq B\\
		\frac{r!r^{B}}{(r+B)!} \binom{r+B}{k} \left(\frac{\lambda}{\mu} \right)^k \pi_0^{(B,r)} & \textrm{if } B \leq k \leq B+r,
		\end{array} \right.
		\label{eq:StDistribution}
		\end{align}
		where
		\begin{align}
		\pi_0^{(B,r)} = \left(\sum_{k=0}^{B-1} \frac{(\lambda r / \mu)^k}{k!} + \frac{r!r^{B}}{(r+B)!} \sum_{k=B}^{B+r}  \binom{r+B}{k} \left(\frac{\lambda}{\mu} \right)^k  \right)^{-1}
		\label{eq:StDistributionNormalizationConstant}
		\end{align}
		Also in this particular case one can pose a QED provisioning rule for the number of spare batteries alone, and derive the asymptotic properties. We treat this case in Appendix~\ref{app:UnlimitedCPSingle}.
	\end{remark}
	
	\subsection{Limiting queue length behavior}
	Due to the curse of dimensionality, it is very challenging to gain a qualitative insight in the (transient) behavior of processes in large-scale systems. Therefore, we resort to fluid and diffusion limits to provide good approximations for the behavior in the actual system when $r$ is large. Recall that $Q^r(t)$ corresponds to the queue length process (the number of batteries in need of charging) under the scaling rules~\eqref{eq:QEDscalingSingleBSSModel} with $r$ cars at time $t \geq 0$. We consider the fluid scaling
	\begin{align}
	\bar{Q}^r(t) = \frac{Q^r(t)}{r}, \quad r \geq 1, t \geq 0.
	\label{eq:FluidScaledProcessDefinition}
	\end{align}
	The fluid-scaled process converges to a deterministic, continuous monotone process with a single fixed steady state value.
	
	\begin{proposition}
		Suppose $S=1$ and scaling rules~\eqref{eq:QEDscalingSingleBSSModel} hold. If $\bar{Q}^r(0) \rightarrow \bar{Q}(0)$ as $r \rightarrow \infty$ with $\bar{Q}(0)$ a finite constant, then $\bar{Q}^r \rightarrow \bar{Q}$ in distribution as $r \rightarrow \infty$, where $\bar{Q}$ satisfies the ODE
		\begin{align*}
		\frac{d \bar{Q}(t) }{dt} = \left\{ \begin{array}{ll}
		\lambda -\mu \bar{Q}(t) & \textrm{if } \bar{Q}(t) < \lambda /\mu, \\
		\lambda^2/\mu - \lambda \bar{Q}(t) & \textrm{if } \bar{Q}(t) \geq \lambda /\mu
		\end{array}\right.
		\end{align*}
		and has the steady state value
		\begin{align*}
		\lim_{t \rightarrow \infty} \bar{Q}(t) = \frac{\lambda}{\mu}.
		\end{align*}
		\label{prop:SingleFluidLimitProcess}
	\end{proposition}
	
	Proposition~\ref{prop:SingleFluidLimitProcess} implies that the number of batteries in need of charging can be approximated by
	\begin{align*}
	Q^r(t) \approx  r  \bar{Q}(t),
	\end{align*}
	where $\bar{Q}(t) = \lim_{r \rightarrow \infty} \bar{Q}^r(t)$ is a solution of an ODE. It describes the approximate (possible) transient behavior before reaching steady state. The proof of Proposition~\ref{prop:SingleFluidLimitProcess} is given in Appendix~\ref{app:SingleMainFluidDiffusion}.
	
	We point out that whenever the queue length is near its steady state value, it remains close to its steady state value from that time onward. That is, if $Q^r(t_0) \approx  \lambda r/\mu$ for some $t_0 \geq 0$, then $Q^r(t) \approx  \lambda r /\mu$ for all $t \geq t_0$. From that point on, the fluid limit becomes a rather rough estimate for the number of batteries in need of charging that allows for further investigation on the fluctuations around this value.
	
	Therefore, we turn our focus to the diffusion scaling
	\begin{align}
	\hat{Q}^r(t) = \frac{Q^r(t)- \lambda r/\mu}{\sqrt{\lambda r / \mu}}, \quad r \geq 1, t\geq 0.
	\label{eq:DiffusionScaledProcessDefinition}
	\end{align}
	This scaling provides more sensitive approximations, as it captures fluctuations of order $\sqrt{r}$. The diffusion-scaled process will tend to a piecewise linear Ornstein-Uhlenbeck processes, with a steady state distribution that can be expressed analytically. The proof can be found in Appendix~\ref{app:SingleMainFluidDiffusion}.
	
	\begin{theorem}
		Suppose $S=1$ and the system operates under~\eqref{eq:QEDscalingSingleBSSModel}. If~$\hat{{Q}}^r(0) \rightarrow \hat{{Q}}(0)$ in distribution as $r \rightarrow \infty$, then $\hat{{Q}}^r \rightarrow \hat{{Q}}$ in distribution as $r \rightarrow \infty$. The process $\hat{{Q}}$ is a diffusion process with drift
		\begin{align*}
		m(x) = -\lambda(x-\beta)^+ -\mu \min\{x,\gamma\},
		\end{align*}
		and constant infinitesimal variance $2\mu$. The steady state density of $\hat{Q}(\infty) = \lim_{t \rightarrow \infty} \hat{Q}(t)$ is given by
		\begin{align}
		\hat{f}(x) = \left\{\begin{array}{ll}
		\alpha_1 \frac{\phi(x)}{\Phi(\gamma)} & \textrm{if } x < \gamma,\\
		\alpha_2 \left(\gamma e^{-\gamma(x-\gamma)} \right)\left(1-e^{-\gamma(\beta-\gamma)} \right)^{-1} & \textrm{if } \gamma \leq x < \beta, \\
		\alpha_3 \sqrt{\frac{\lambda}{\mu}} \phi\left( \frac{x-(\beta-\frac{\mu}{\lambda} \gamma)}{\sqrt{\mu / \lambda}}\right) \Phi\left( -\sqrt{\frac{\mu}{\lambda}}\gamma\right)^{-1} & \textrm{if } x \geq \beta,\\
		\end{array} \right.
		\label{eq:DensityQEDLimited1}
		\end{align}
		where $\alpha_i=r_i/(r_1+r_2+r_3)$, $i=1,2,3$ with
		\begin{align*}
		r_1 &=1, \\
		r_2 &= \left\{ \begin{array}{ll}
		{\phi(\gamma)}{\Phi(\gamma)}^{-1} \frac{1}{\gamma} \left(1-e^{-\gamma(\beta-\gamma)} \right) & \textrm{if } \gamma \neq 0, \\
		\sqrt{\frac{2}{\pi}} \beta & \textrm{if } \gamma = 0, \\
		\end{array}\right.\\
		r_3 &= \frac{\phi(\gamma)}{\Phi(\gamma)} e^{-\gamma(\beta-\gamma)}\sqrt{\frac{\mu}{\lambda}} \phi\left( \sqrt{\frac{\mu}{\lambda}}\gamma \right)^{-1} \Phi\left( -\sqrt{\frac{\mu}{\lambda}}\gamma\right).
		\end{align*}
		\label{thm:DiffusionLimitProcess}
	\end{theorem}
	
	Equation~\eqref{eq:DensityQEDLimited1} in Theorem~\ref{thm:DiffusionLimitProcess} is obtained by taking the limit of the scaled diffusion process (as $r \rightarrow \infty$), and finding its steady state distribution (as $t\rightarrow\infty$). However, in order to obtain a good approximation of the steady state distribution with a fixed number of cars $r$, it is arguably more reasonable to consider the steady state distribution of the scaled diffusion process (as $t\rightarrow \infty$) and next take the limit as $r \rightarrow \infty$. Fortunately, the following theorem shows that the order at which one takes the limit leads to the same result.
	
	\begin{theorem}
		If $S=1$ and~\eqref{eq:QEDscalingSingleBSSModel} holds, the steady state distribution of the diffusion scaled process~$\hat{Q}^r(\infty)$ converges in distribution to $\hat{Q}(\infty)$ as in Theorem~\ref{thm:DiffusionLimitProcess}.
		\label{thm:StStateToDiffusion}
	\end{theorem}
	
	The proof of Theorem~\ref{thm:StStateToDiffusion} is given in Appendix~\ref{app:SingleMainSteadyStateLimits}. As the order at which the limits are taken does not affect the result, we use the limiting process $\hat{Q}(\infty)$ to obtain approximations for the performance measures.
	
	\subsection{Performance measures}
	\label{sec:PerformanceSingleBSS}
	Typical performance measures for the QoS level for the EV users include the waiting probability and the expected waiting time. We view the efficiency-level for the station by the resources utilization. Typically, the QED regime in many-server systems causes the waiting probability to tend to a non-degenerate limit as $r \rightarrow \infty$, the waiting time to vanish, while the resource utilization tends to one. These features also appear in our system under the proposed QED scaling.
	
	Due to the PASTA (Poisson Arrivals See Time Averages) property in open queueing systems where the arrival process is a time-homogeneous Poisson process, the steady state value of any quantity is the same as at arrival instances. In particular, the waiting probability equals the steady state probability that the number of fully-charged batteries is zero, or equivalently, the number of batteries in need of charging is at least~$B$. Unfortunately, the arrival process in our closed setting is state-dependent. Yet, Theorem~\ref{thm:DiffusionLimitProcess} shows that the fluctuations in arrival rate is of order $O(\sqrt{r})$, i.e.\ the arrival rate is $\lambda r - O(\sqrt{r})$ (with high probability). These small changes will therefore become negligible as $r \rightarrow \infty$. In other words, this argument implies that the PASTA property remains valid asymptotically. This notion can be formalized similarly as is done in~\cite{Jennings2008}. Summarizing, if $W$ denotes the waiting time of an arriving EV user, then
	\begin{align*}
	\Prob(W > 0) = \lim_{r \rightarrow \infty} \Prob\left( Q^r(\infty) \geq B^r \right) = \Prob\left( \hat{Q}(\infty) \geq \beta \right),
	\end{align*}
	where $\hat{Q}(\infty)$ is as in Theorem~\ref{thm:DiffusionLimitProcess}.

	The key concept to derive the expected waiting time is Little's law, stating that the long-term average number of waiting cars, denoted by $Q^r_W$, equals the long-term throughput multiplied by the average waiting time. In other words,
	\begin{align*}
	\E(Q^r_W) = \theta \E(W),
	\end{align*}
	where the throughput $\theta$ can be viewed as the long-term average rate at which EVs arrive, and hence also leave the battery swapping station. We can express the throughput as
	\begin{align*}
	\theta = \lambda r - \lambda \E(Q^r_W),
	\end{align*}
	since the long-term average number of batteries not in need of charging is in fact the expected number of cars not waiting at the station in this closed system. Therefore, it follows that
	\begin{align}
	\E(W) = \frac{\E(Q^r_W)}{\lambda (r- \E(Q^r_W))}.
	\label{eq:LittleLawWaitingTime}
	\end{align}
	In turn, the expected number of waiting cars can be derived directly using Theorem~\ref{thm:DiffusionLimitProcess} and the observation $Q_W^r=(Q^r(\infty)-B^r)^+$,
	\begin{align*}
	\E(Q_W^r) &= \sum_{k=B^r+1}^r (k-B^r) \Prob\left(Q(\infty) = k\right) = \sqrt{\frac{\lambda r}{\mu}} \sum_{k=B^r+1}^r \frac{k-B^r}{\sqrt{\lambda r / \mu}} \Prob\left(\hat{Q}(\infty) = \frac{k-\lambda r / \mu}{\sqrt{\lambda r / \mu}} \right) \\
	&\sim \sqrt{\frac{\lambda r}{\mu}} \int_{\beta}^\infty  (x-\beta) \hat{f}(x) \, dx
	\end{align*}
	as $r\rightarrow\infty$. We point out that $\E(Q_W^r)$ is consequently of order $\Theta(\sqrt{r})$, and together with~\eqref{eq:LittleLawWaitingTime} this implies that $E(W)$ is of order $\Theta(1/\sqrt{r})$ and hence vanishes in the limit.
	
	The resources will be fully utilized under~\eqref{eq:QEDscalingSingleBSSModel} as $r \rightarrow \infty$. Theorem~\ref{thm:DiffusionLimitProcess} implies that at most $O(\sqrt{r})$ charging points are not utilized, and the number of fully-charged batteries is also of order $O(\sqrt{r})$. Therefore, as $r \rightarrow \infty$,
	\begin{align}
	\rho_{F^r} = 1- O(1/\sqrt{r}) , \quad \rho_{B^r} = 1- O(1/\sqrt{r}).
	\label{eq:UtilizationSingleBSS}
	\end{align}
	
	\begin{theorem}
		Suppose $S=1$, and the system is operating under~\eqref{eq:QEDscalingSingleBSSModel}. Then the following properties hold as $r \rightarrow \infty$. The waiting probability has a non-degenerate limit given by
		\begin{align*}
		\Prob&(W>0) \sim \Prob\left( \hat{Q}(\infty) \geq \beta \right) \\
		&= \left(1+ \sqrt{\frac{\lambda}{\mu}}\frac{\phi(\sqrt{\mu/\lambda}\gamma)}{\phi(\gamma)} e^{\gamma(\beta-\gamma)} \frac{\Phi(\gamma)}{\Phi(-\sqrt{\mu/\lambda}\gamma)} + \sqrt{\frac{\lambda}{\mu}} \frac{\phi(\sqrt{\mu/\lambda}\gamma)}{\gamma}\left( e^{\gamma(\beta-\gamma) } -1\right) \Phi\left(-\sqrt{\frac{\mu}{\lambda}}\gamma\right)^{-1} \right)^{-1}.
		\end{align*}
		The expected waiting time behaves as
		\begin{align*}
		\frac{\E(W)}{\sqrt{r}} \sim \frac{\alpha_3}{\sqrt{\lambda \mu}} \left(  \sqrt{\frac{\mu}{\lambda}} \phi\left(\frac{\mu}{\lambda}\gamma \right) \Phi\left(-\sqrt{\frac{\mu}{\lambda}} \gamma\right)^{-1} - \frac{\mu}{\lambda}\gamma \right),
		\end{align*}
		with $\alpha_i$ are as in Theorem~\ref{thm:DiffusionLimitProcess}. Finally, the resource utilizations behave as
		\begin{align*}
		\rho_{F^r} \rightarrow 1 , \quad \rho_{B^r} \rightarrow 1.
		\end{align*}
		\label{thm:PerformanceMeasuresSingleBSS}
	\end{theorem}
	
	The proof of Theorem~\ref{thm:PerformanceMeasuresSingleBSS} is given in Appendix~\ref{app:SingleMainPerformance}.
	
	\section{System behavior in case of multiple stations}
	\label{sec:MultipleStations}
	When the number of stations $S \geq 2$, the analysis of system behavior needs to account for the underlying routing mechanism of arriving EVs. Whenever an EV is in need of recharging, stations~$i$ and~$j$ are in its direct vicinity with probability $p_{ij}$, and it chooses to move the station~$i$ if~\eqref{eq:ArrivalMoveLoadBalancing} holds. For a resource pooling effect to occur, we require that there is a sufficient number of pairs $(i,j)$ for which $p_{ij}>0$. For example, if the network consists of four stations with $p_{12}=p_{34}=1/2$, there are no arrivals that can choose between one station in the set $\{1,2\}$ and another in the set $\{3,4\}$. Therefore, possible discrepancies in queue lengths are not levelled by the arrival mechanism between these two sets. Therefore, we assume that for every non-empty set $\mathcal{S}$ of stations, there is at least one pair $(i,j)$ with $i \in \mathcal{S}$ and $j \not\in \mathcal{S}$ for which $p_{ij}>0$. This statement is equivalent to the following assumption.
	
	\begin{assumption}
		Let $G = (V,E)$ be a graph, where $V=\{1,\ldots,S\}$ and $E=\{(i,j) : p_{ij}>0\}$. We assume that the graph $G$ is connected.
		\label{ass:ConnectedUnderlyingGraph}
	\end{assumption}

\begin{remark}\normalfont
For our results to follow through in the extended model as described in Remark~\ref{rem:BetterScaling}, Assumption~\ref{ass:ConnectedUnderlyingGraph} needs to be updated as follows. Let $G = (V,E)$ be a graph, where $V=\{1,\ldots,S\}$ and $E=\{(i,j) : i,j \in m, |m| \geq 2, m \in \mathcal{M}\}$. Then, we assume that the graph $G$ is connected. Note that if $m=2$ for every $m \in \mathcal{M}$, the setting as well as this assumption reduces to the original setting as described in this paper.
\end{remark}

	\subsection{System dynamics}
	There are many processes that are of interest in this system, and in particular, the queue length process at each station. In our analysis, we consider $\{\mathbb{X}^r(t), t \geq 0\}$ with
	\begin{align*}
	\mathbb{X}^r = \left(A^r, A_d^r, Q^r, Z^r, Y^r, T^r, D^r, L^r \right),
	\end{align*}
	where
	\begin{itemize}
		\item $A^r = \left(A_{ij}^r ; \{i,j\} \in E \right)$, where $A_{ij}^r(t)$ is the number of arrivals that are closest to stations $i$ and $j$ until time $t \geq 0$ in the $r$'th system;
		\item $A^r_d = \left(A_{ij,i}^r ; \{i,j\} \in E \right)$, where $A_{ij,i}^r(t)$ is the number of arrivals that are closest to stations $i$ and $j$ and are routed to station $i$ until time $t \geq 0$ in the $r$'th system;
		\item $Q^r = \left(Q_j^r ; 1 \leq j \leq S \right)$, where $Q_j^r(t)$ is the number of batteries in need of charging at time $t \geq 0$ in the $r$'th system;
		\item $Z^r = \left(Z_j^r ; 1 \leq j \leq S \right)$, where $Z_j^r(t)$ is the number of busy servers (charging points) at time $t \geq 0$ in the $r$'th system;
		\item $Y^r$, where $Y^r(t)$ is the aggregated time of all cars that are not waiting at some station until time $t \geq 0$ in the $r$'th system;
		\item $T^r = \left(T_j^r ; 1 \leq j \leq S \right)$, where $T_j^r(t)$ is the aggregated time of all servers at station~$j$ that were charging until time $t \geq 0$ in the $r$'th system;
		\item $D^r = \left(D_j^r ; 1 \leq j \leq S \right)$, where $D_j^r(t)$ is the number of service completions at station~$j$ until time $t \geq 0$ in the $r$'th system;
		\item $L^r$, where $L^r(t)$ is the number of batteries that are positioned in an EV not waiting at a station in the $r$'th system at time $t \geq 0$.
	\end{itemize}
	
	Clearly, there are strong relations between the individual processes in $\mathbb{X}^r$. For example, there is a routing policy that dictates where a car in need of a full battery drives to in order to swap its battery. This notion is captured by the arrival processes $A^r$ (the classification of the different arrival types) and $A^r_d$ (the routing decision). To generate the arrival and service completion processes, we introduce a set of independent Poisson processes. Let $\{\Lambda_{ij}(t),t\geq 0 \}$ for all $\{i,j\} \in E$ be independent Poisson processes with rate $p_{ij} \lambda$ and $\{S_j(t),t \geq 1\}$ for all $1 \leq j \leq S$ be independent Poisson processes with rate $\mu$. The system dynamics satisfy the following identities:
	
	\begin{align}
	A_{ij}^r(t) &= A_{ij,i}^r(t) + A_{ij,j}^r(t) \quad \forall \{i,j\} \in E, \label{eq:IdentityArrivals1}\\
	A_{ij}^r(t) &= \Lambda_{ij}\left( Y^r (t) \right), \quad \forall \{i,j\} \in E, \label{eq:IdentityArrivals2}\\
	Q_j^r(t) &= Q_j^r(0) + \sum_{i : \{i,j\} \in E } A_{ij,j}^r(t)-D_j^r(t), \quad \forall j=1,\ldots,S, \label{eq:IdentityQueueLength}\\
	D_j^r(t) &= S_j\left( T_j^r(t) \right), \quad \forall j=1,\ldots,S, \label{eq:IdentityServiceCompletions}\\
	Y^r(t) &= \int_0^t L^r(s) \, ds, \label{eq:IdentityAggregatedArrivalTime}\\
	T_j^r(t) &= \int_0^t Z_j^r(s) \, ds, \quad \forall j=1,\ldots,S, \label{eq:IdentityBusyTime}\\
	Z_j^r(t) &= \min\{Q_j^r(t),F_j^r\}, \quad \forall j=1,\ldots,S, \label{eq:IdentityBusyServers}\\
	L^r(t) &= r - \sum_{j=1}^S \left( Q_j^r(t) - B_j^r \right)^+, \label{eq:IdentityDrivingCars}\\
	\forall \{i,j\} \in E, A_{ij,i}^r(t) &\textrm{ can only increase when } Q_i^r(t)/p_i \leq Q_j^r(t)/p_j.
	\end{align}
	
	We refer to these equations as the system identities, and they prove to be central for deriving our results. The derivations use the framework set out in~\cite{DaiTezcan2011}, which in turn is based on~\cite{Bramson1998}. We adopt much of the notation and definitions in this paper, and before stating our main results, we repeat them for the purpose of self-containment. For each positive integer $d$, we denote by $\mathbb{D}^d[0,\infty]$ the $d$-dimensional Skorohod path space. For $x,y \in \mathbb{D}^d[0,\infty]$ and $T >0$, let
	\begin{align*}
	\lVert x(\cdot)-y(\cdot) \rVert_T  = \sup_{0 \leq t \leq T} |x(t)-y(t)|,
	\end{align*}
	where $|z| = \max_{i=1,\ldots,d} |z_i|$ for any $z=(z_1,\ldots,z_d) \in \mathbb{R}^d$. The space $\mathbb{D}^d[0,\infty]$ is endowed with the $J_1$ topology, and the weak convergence in this space is considered with respect to this topology. We say a sequence of functions $\{x_n\} \in \mathbb{D}^d[0,\infty]$ converges uniformly on compact sets (u.o.c) sets to $x \in \mathbb{D}^d[0,\infty]$ as $n \rightarrow \infty$ if for each $T \geq 0$,
	\begin{align*}
	\lVert x_n(\cdot) - x(\cdot) \rVert_T \rightarrow 0
	\end{align*}
	as $n \rightarrow \infty$. Moreover, we say that $t \geq 0$ is a regular point of a function $x$ if $x$ is differentiable at $t \geq 0$, and denote its derivative by $x'(\cdot)$. We assume that the random variables in $\mathbb{X}^r$ live on the same probability space $(\Omega, \mathcal{F}, \mathbb{P})$. Often, we consider sample paths of stochastic processes, and whenever we want to make the dependence on the sample path explicit, we write $X^r(\cdot,\omega)$ for the sample path associated with $\omega \in \Omega$ for a stochastic process $X^r$.

	\subsection{Fluid limit}
	To capture the rough system dynamics, we consider the fluid-scaled process
	\begin{align*}
	\bar{\mathbb{X}} = \lim_{r \rightarrow \infty} \bar{\mathbb{X}}^r, \quad \bar{\mathbb{X}}^r = \frac{\mathbb{X}^r}{r}.
	\end{align*}
	For each process $X^r$ in $\mathbb{X}^r$, we define similarly its fluid equivalent as $\bar{X}^r = X^r / r$ and its limiting process $\bar{X} = \lim_{r \rightarrow \infty} \bar{X}^r$. We adopt the definition of a fluid limit and its invariant state(s) from~\cite{DaiTezcan2011}. That is, we consider $\mathcal{A} \subset \Omega$ such that the FSLLN holds, i.e.
	\begin{align*}
	\frac{\Lambda_{ij}(r x)}{r} \rightarrow p_{ij} \lambda x, \quad \{i,j\} \in E \quad\textrm{and }\quad \frac{S_j(r x)}{r} \rightarrow \mu x, \hspace{0.25cm} j=1,\ldots,S,
	\end{align*}
	u.o.c. as $r \rightarrow \infty$. Due to the FSLLN, we observe that one can choose $\mathcal{A}$ large enough such that $\Prob(\mathcal{A})=1$.
	
	\begin{definition}\label{def:FluidLimitMultipleStations}
		We call $\bar{\mathbb{X}}$ a fluid limit of $\{\mathbb{X}^r \}$ if there exists an $\omega \in \mathcal{A}$ and (sub)sequence  $\{r_l\}$ with $r_l \rightarrow \infty$ as $l \rightarrow \infty$, such that $\bar{\mathbb{X}}^{r_l}(\cdot, \omega)$ converges u.o.c. to $\bar{\mathbb{X}}(\cdot, \omega)$. Moreover, let ${q} = (q_1,\ldots,q_S)$ be an invariant state of the fluid limits if for any fluid limit $\bar{\mathbb{X}}$, $\bar{Q}(0)= (\bar{Q}_1(0),\ldots,\bar{Q}_S(0))=(q_1,\ldots,q_S) = q$ implies that $\bar{Q}(t)=q$ for all $t\geq 0$.
	\end{definition}
	
	In Proposition~\ref{prop:SingleFluidLimitProcess}, we focus on the fluid-scaled queue length process only for $S=1$, and the sequence $r_l=l$. Instead of requiring $\bar{Q}^r(0) \rightarrow \bar{Q}(0)$ with $\bar{Q}(0)$ a finite constant, Definition~\ref{def:FluidLimitMultipleStations} allows for $\bar{Q}(0)$ to be random. Proposition~\ref{prop:SingleFluidLimitProcess} implies that in case that $S=1$, the fluid limits exist and are deterministic, (Lipschitz) continuous paths that depend only on the realization of $\bar{Q}(0)$. Moreover, there is a single unique invariant state given by $\lambda/\mu$. A similar result holds when $S \geq 2$.
	
	\begin{theorem}
		Let $\{\mathbb{X}^r\}$ be a sequence of systems. Then the fluid limits exist, where each component is Lipschitz continuous. Each fluid limit $\bar{\mathbb{X}}$ satisfies the following equations for all $t\geq 0$:
		\begin{align}
		\bar{A}_{ij}(t) &= \bar{A}_{ij,i}(t) + \bar{A}_{ij,j}(t), \quad \forall \{i,j\} \in E, \label{eq:IdentityArrivals1Fluid}\\
		\bar{A}_{ij}(t) &=  p_{ij} \lambda \bar{Y}(t), \quad \forall \{i,j\} \in E,  \label{eq:IdentityArrivals2Fluid}\\
		\bar{Q}_j(t) &= \bar{Q}_j(0) + \sum_{i : \{i,j\} \in E } \bar{A}_{ij,j}(t)-\bar{D}_j(t), \quad \forall j=1,\ldots,S, \label{eq:IdentityQueueLengthFluid} \\
		\bar{D}_j(t) &= \mu \bar{T}_j(t), \quad \forall j=1,\ldots,S,\label{eq:IdentityServiceCompletionsFluid}\\
		\bar{Y}(t) &= \int_0^t \bar{L}(s) \, ds, \quad \forall j=1,\ldots,S, \label{eq:IdentityAggregatedArrivalTimeFluid}\\
		\bar{T}_j(t) &= \int_0^t \bar{Z}_j(s) \, ds, \quad \forall j=1,\ldots,S, \label{eq:IdentityBusyTimeFluid}\\
		\bar{Z}_j(t) &= \min\{\bar{Q}_j(t),p_j \lambda/\mu\}, \quad \forall j=1,\ldots,S,\label{eq:IdentityBusyServersFluid} \\
		\bar{L}(t) &= 1 - \sum_{j=1}^S \left( \bar{Q}_j(t) - p_j \lambda/\mu \right)^+. \label{eq:IdentityDrivingCarsFluid}
		\end{align}
		Also, for every $\{i,j\} \in E$, if $t$ is a regular point of $\bar{\mathbb{X}}$, then
		\begin{align}
		\bar{A}'_{ij,i}(t) = \lambda p_{ij} \bar{L}(t) \quad \textrm{and}\quad  \bar{A}'_{ij,j}(t) = 0\quad \textrm{if   } \; \frac{\bar{Q}_j(t)}{p_j} > \frac{\bar{Q}_i(t)}{p_i}.
		\label{eq:IdentityArrivalDerivativeFluid}
		\end{align}
		Finally, there is a unique invariant state given by $q=(q_1,\ldots,q_S)$ with $q_i = p_i \lambda/\mu$ for $i=1,\ldots,S$.
		\label{thm:FluitLimitResult}
	\end{theorem}
	
	The (uniqueness of the) invariant state result for the fluid limit is central for the existence of a properly-defined diffusion process as it states that if $\bar{Q}(0)=(p_1 \lambda/\mu,\ldots,p_S \lambda/\mu)$, the fluid limits are time invariant. We present a proof of Theorem~\ref{thm:FluitLimitResult} in Appendix~\ref{app:FluidLimit}.
	
	\subsection{Diffusion limit}
	Due to the policy governing which station a car drives to in order to replace a battery, one observes the so-called load-balancing effect. By setting the number of resources as in~\eqref{eq:QEDscalingNetworkModel}, this load-balancing effect is so strong that in fact complete resource pooling occurs. In other words, the system behaves as if there is a single large swapping station where the number of resources equals the aggregated total of the individual stations. This appealing consequence ensures that there are no idle resources at one station, while at another there are possible long waiting lines of cars that are waiting for a battery exchange.
	
	The key concept to derive this effect is to show a state space collapse (SSC) result. That is, we consider the diffusion-scaled queue length process defined as
	\begin{align*}
	\hat{Q}^r_i(t) =\frac{Q^r_i(t)-p_i \lambda r /\mu}{p_i \sqrt{\lambda r/\mu}}, \quad i=1,\ldots,S.
	\end{align*}
	In our model, the SSC result states that (almost instantaneously) the diffusion-scaled queue length processes are arbitrarily close at all stations, and stay close during any fixed interval.
	
	\begin{theorem}
		Suppose
		\begin{align*}
		\hat{Q}^r(0) \overset{d}{\rightarrow} \hat{Q}(0),
		\end{align*}
		as $r \rightarrow \infty$, where $\hat{Q}(0)$ is a random vector. Then, for every $K^r=o(\sqrt{r})$ with $K^r \rightarrow \infty$ as $r \rightarrow \infty$, and for every $T >0$ and $ \epsilon >0$,
		\begin{align}
		\Prob\left( \sup_{K^r/\sqrt{r} \leq t \leq T} \lvert  \hat{Q}_i(t) - \hat{Q}_j(t) \rvert > \epsilon \right) \rightarrow 0
		\label{eq:StrongSSCInititiallyNonequal}
		\end{align}
		for every $i,j \in \{1,\ldots,S\}$ as $r \rightarrow \infty$. If, in addition for every $i,j \in \{1,\ldots,S\}$,
		\begin{align*}
		\lvert  \hat{Q}^r_i(0) - \hat{Q}^r_j(0) \rvert \overset{\Prob}{\rightarrow} 0,
		\end{align*}
		then
		\begin{align}
		\Prob\left( \lVert  \hat{Q}_i(\cdot) - \hat{Q}_j(\cdot) \rVert_T > \epsilon \right) \rightarrow 0
		\label{eq:StrongSSCInititiallyEqual}
		\end{align}
		for every $i,j \in \{1,\ldots,S\}$ as $r \rightarrow \infty$.
		\label{thm:StrongSSC}
	\end{theorem}
	
	The proof of Theorem~\ref{thm:StrongSSC} is given in Appendix~\ref{app:SSCproofs}. This result reveals that instead of considering the individual queue length processes, it suffices to track the total queue length process instead. More specifically, define the sequence of random processes $\{Q^r_\Sigma(t), t \geq 0\}$ with $r \in \mathbb{N}$, where $Q^r_\Sigma(t) = \sum_{j=1}^S Q_j^r(t)$, and
	\begin{align*}
	\hat{Q}_\Sigma^r (t) = \frac{\sum_{j=1}^S (Q^r_j(t)-p_j \lambda r/\mu)}{ \sqrt{\lambda r/\mu}} = \sum_{j=1}^S p_j \hat{Q}_j^r (t).
	\end{align*}
	As the state space collapse implies that $\hat{Q}_i^r (t) \approx \hat{Q}_j^r (t)$ for all $i,j \in \{1,\ldots,S\}$ (for $t \geq K^r/\sqrt{r}$), we can approximate the queue length at an individual queue by
	\begin{align*}
	Q_j^r (t) = p_j\left( \frac{\lambda r}{\mu} +  \hat{Q}_j^r (t) \sqrt{\frac{\lambda r}{\mu} }\right) \approx p_j\left( \frac{\lambda r}{\mu} +  \hat{Q}_\Sigma^r (t) \sqrt{\frac{\lambda r}{\mu} }\right)
	\end{align*}
	for all $j=1,\ldots,S$. The limiting process of the total queue length can be derived using the SSC result.
	
	\begin{theorem}
		Suppose~$\hat{Q}^r(0) \rightarrow \hat{Q}(0)$ in distribution as $r \rightarrow \infty$, and
		\begin{align*}
		\left\lvert \hat{Q}_i^r(0) - \hat{Q}_j^r(0) \right\rvert \overset{\Prob}{\rightarrow} 0
		\end{align*}
		for all $i,j \in \{1,\ldots,S\}$. Then, $\hat{Q}^r_\Sigma \rightarrow \hat{Q}_\Sigma$ in distribution as $r \rightarrow \infty$, where $\hat{Q}_\Sigma$ is a diffusion process with drift
		\begin{align*}
		m(x) = -\lambda(x-\beta)^+ -\mu \min\{x,\gamma\},
		\end{align*}
		and constant infinitesimal variance $2\mu$. The steady state density $\hat{Q}_\Sigma(\infty)$ is given by
		\begin{align}
		\hat{f}_{\Sigma}(x) = \left\{\begin{array}{ll}
		\alpha_1 \frac{\phi(x)}{\Phi(\gamma)} & \textrm{if } x < \gamma,\\
		\alpha_2 \left(\gamma e^{-\gamma(x-\gamma)} \right)\left(1-e^{-\gamma(\beta-\gamma)} \right)^{-1} & \textrm{if } \gamma \leq x < \beta, \\
		\alpha_3 \sqrt{\frac{\lambda}{\mu}} \phi\left( \frac{x-(\beta-\frac{\mu}{\lambda} \gamma)}{\sqrt{\mu / \lambda}}\right) \Phi\left( -\sqrt{\frac{\mu}{\lambda}}\gamma\right)^{-1} & \textrm{if } x \geq \beta,\\
		\end{array} \right.
		\label{eq:DensityQEDLimited1a}
		\end{align}
		where $\alpha_i=r_i/(r_1+r_2+r_3)$, $i=1,2,3$ with
		\begin{align*}
		r_1 &=1, \\
		r_2 &= \left\{ \begin{array}{ll}
		{\phi(\gamma)}{\Phi(\gamma)} \frac{1}{\gamma} \left(1-e^{-\gamma(\beta-\gamma)} \right) & \textrm{if } \gamma \neq 0, \\
		\sqrt{\frac{2}{\pi}} \beta & \textrm{if } \gamma = 0, \\
		\end{array}\right.\\
		r_3 &= \frac{\phi(\gamma)}{\Phi(\gamma)} e^{-\gamma(\beta-\gamma)}\sqrt{\frac{\mu}{\lambda}} \phi\left( \sqrt{\frac{\mu}{\lambda}}\gamma \right)^{-1} \Phi\left( -\sqrt{\frac{\mu}{\lambda}}\gamma\right).
		\end{align*}
		\label{thm:DiffusionLimitProcessA}
	\end{theorem}

\begin{proof}
	We observe that the steady state density is a direct consequence of the diffusion process~\cite{BrowneWhitt1995}. What remains to show is that $\hat{Q}_\Sigma^r$ converges to the described diffusion process as $r \rightarrow \infty$. Equivalently, we need to show that
	\begin{align}
	d \hat{Q}_\Sigma(t) = -\lambda \left( \hat{Q}_\Sigma(t) - \beta \right)^+ - \mu \min\left\{ \hat{Q}_\Sigma(t),\gamma \right\} + \sqrt{2 \mu} dW(t),
	\label{eq:AlternativeExpressionSummedQueueLength}
	\end{align}
	where $\{W(t),t\geq0\}$ is a standard Brownian motion. We note that due to the system identities,
	\begin{align*}
	\hat{Q}_\Sigma(t) = \hat{Q}_\Sigma(0) + \frac{\sum_{\{i,j\}\in E} A_{ij}^r(t) - \sum_{j=1}^S D^r_j(t) }{\sqrt{\lambda r/\mu}},
	\end{align*}
	where
	\begin{align*}
	\sum_{\{i,j\}\in E} A_{ij}^r(t) = \Lambda\left( \int_{0}^t L^r(s) \, ds\right),
	\end{align*}
	and
	\begin{align*}
	\sum_{j=1}^S D^r_j(t) =  \sum_{j=1}^S S_j \left( \int_{0}^t Z_j^r(s) \, ds\right).
	\end{align*}
	We observe that due to the FCLT and Theorem~\ref{thm:FluitLimitResult},
	\begin{align*}
	\frac{\Lambda\left( \int_{0}^t L^r(s) \, ds\right) -  \lambda \int_{0}^t L^r(s) \, ds}{\sqrt{\lambda r/\mu}} =  \frac{\Lambda\left( r \int_{0}^t \bar{L}^r(s) \, ds\right) - \lambda r \int_{0}^t \bar{L}^r(s) \, ds}{\sqrt{1/\mu}\sqrt{\lambda r}} \overset{d}{\rightarrow} \textrm{BM}_A(t),
	\end{align*}
	where $\{\textrm{BM}_A(t), t\geq 0\}$ is Brownian motion with mean zero and variance $\mu$. Similarly, due to the FCLT and Theorem~\ref{thm:FluitLimitResult},
	\begin{align*}
	\sum_{j=1}^S \frac{S_j\left( \int_{0}^t Z_j^r(s) \, ds\right) -  \mu \int_{0}^t Z_j^r(s) \, ds}{\sqrt{\lambda r/\mu}} = \sum_{j=1}^S \frac{S_j\left( r \int_{0}^t \bar{Z}^r(s) \, ds\right) - \mu r \int_{0}^t \bar{Z}^r(s) \, ds}{\sqrt{1/(p_j \mu)}\sqrt{p_j \lambda r}} \overset{d}{\rightarrow} \textrm{BM}_D(t),
	\end{align*}
	where $\{\textrm{BM}_D(t), t\geq 0\}$ is an (independent) Brownian motion with mean zero and variance $\mu$. The sum of these two processes results is equal to (in distribution) a Brownian with mean zero and variance $2\mu$, which contributes to the $\sqrt{2\mu} \, dW(t)$ term in~\eqref{eq:AlternativeExpressionSummedQueueLength}. Next, we observe that due to the system identities and the definition of the diffusion scaling,
	\begin{align*}
	\frac{ \lambda \int_{0}^t L^r(s) \, ds- \sum_{j=1}^S  \mu \int_{0}^t  Z_j^r(s) \, ds}{\sqrt{\lambda r/\mu}} = -\lambda \sum_{j=1}^S p_j \int_0^t \left( \hat{Q}_j^r(s)-\beta \right)^+ \, ds - \mu \sum_{j=1}^S p_j \int_0^t \min\left\{\hat{Q}_j^r(s),\gamma \right\} \, ds.
	\end{align*}
	Since
	\begin{align*}
	\min_{1 \leq j \leq S} \hat{Q}^r(t) \leq \hat{Q}^r_\Sigma(t) \leq \max_{1 \leq j \leq S} \hat{Q}^r(t)
	\end{align*}
	for all $t \in [0,T]$, and due to Theorem~\ref{thm:StrongSSC},
	\begin{align*}
	\left\lVert \hat{Q}_\Sigma^r(\cdot)- \hat{Q}_j^r(\cdot)  \right\rVert_T \leq \epsilon(r), \quad j=1,\ldots,S.
	\end{align*}
	where the sequence $\{\epsilon(r),r \in \mathbb{N}\}$ can be chosen such that $\epsilon(r) \rightarrow 0$ as $r \rightarrow \infty$. Then, as $r \rightarrow \infty$,
	\begin{align*}
	\frac{ \lambda \int_{0}^t L^r(s) \, ds- \sum_{j=1}^S  \mu \int_{0}^t  Z_j^r(s) \, ds}{\sqrt{\lambda r/\mu}} \overset{\Prob}{\rightarrow} -\lambda  \int_0^t \left( \hat{Q}_\Sigma(s)-\beta \right)^+ \, ds - \mu \int_0^t \min\left\{\hat{Q}_{\Sigma}(s),\gamma \right\} \, ds.
	\end{align*}
	This contributes to the first two terms in~\eqref{eq:AlternativeExpressionSummedQueueLength}. Applying the continuous mapping theorem concludes the result.
\hfill\end{proof}

	Another consequence of the state space collapse result is that the waiting probabilities and expected waiting times are equal at all stations, as well as the resource utilization levels. In fact, it exhibits the same behavior as if there would be a single station due to the complete resource pooling effect.
	
	\begin{corollary}
		Suppose the system is operating under~\eqref{eq:QEDscalingSingleBSSModel}. Then the following properties hold as $r \rightarrow \infty$ for all $i=1,\ldots,S$. The waiting probability has a non-degenerate limit given by
		\begin{multline*}
		\Prob(W^r_i>0) \to \Prob\left( \bar{Q}_\Sigma(\infty) \geq \beta \right) \\
		= \left(1+ \sqrt{\frac{\lambda}{\mu}}\frac{\phi(\sqrt{\mu/\lambda}\gamma)}{\phi(\gamma)} e^{\gamma(\beta-\gamma)} \frac{\Phi(\gamma)}{\Phi(-\sqrt{\mu/\lambda}\gamma)} + \sqrt{\frac{\lambda}{\mu}} \frac{\phi(\sqrt{\mu/\lambda}\gamma)}{\gamma}\left( e^{\gamma(\beta-\gamma) } -1\right) \Phi\left(-\sqrt{\frac{\mu}{\lambda}}\gamma\right)^{-1} \right)^{-1}.
		\end{multline*}
		The expected waiting time behaves as
		\begin{align*}
		\frac{\E(W^r_i)}{\sqrt{r}} \to \frac{\alpha_3}{\sqrt{\lambda \mu}} \left(  \sqrt{\frac{\mu}{\lambda}} \phi\left(\frac{\mu}{\lambda}\gamma \right) \Phi\left(-\sqrt{\frac{\mu}{\lambda}} \gamma\right)^{-1} - \frac{\mu}{\lambda}\gamma \right).
		\end{align*}
		with $\alpha_i$ are as in Theorem~\ref{thm:DiffusionLimitProcessA}. Finally, the resource utilizations behave as
		\begin{align*}
		\rho_{F^r_i} \rightarrow 1 , \quad \rho_{B^r_i} \rightarrow 1.
		\end{align*}
		\label{cor:PerformanceMeasures}
	\end{corollary}
	
	\section{Simulation experiments}\label{sec:SimulationExperiments}
	The results presented are given in an asymptotic regime where the charging times are exponentially distributed. In this section, we conduct simulation experiments to evaluate the quality of our approximations and the robustness of the state-space collapse result. We first focus on a large-scale system to illustrate the implications of our results. Next, we also zoom in on a moderate-sized system that reflect a more realistic setting for an EV battery swapping infrastructure.
	
	\subsection{Large-scaled system}\label{sec:LargeScaleSystemSimulations}
	Throughout the experiments in this section, we consider a network with $5$ stations where the arrival probabilities are given by $p_{12}=p_{24}=p_{34}=p_{35}=0.1$, $p_{23}=0.4$ and $p_{45}=0.2$, see Figure~\ref{fig:StationNetworkExperiments}. This results in an effective arrival probability $p=(p_1,\ldots,p_5)=(0.05,0.3,0.3,0.2,0.15)$ at the stations.
	
	\begin{figure}[htb]
		\centering
		\includegraphics[width=6cm]{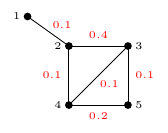}
		\caption{Illustration of the battery swapping network with arrival streams used in the simulation experiments.}
		\label{fig:StationNetworkExperiments}
	\end{figure}
	
	As a battery swapping infrastructure currently does not exist yet in real-life, there is no (significant) data that can be exploited to obtain useful parameter choices. Instead, we discuss an adequate provisioning strategy under the following assumptions. We assume that the battery swapping facility installed (relatively) fast charging points where recharging takes one hour on average ($\mu=1$), and that every EV user returns for recharging services after every $40$ hours on average ($\lambda=0.025$). In addition, we stress that our results are based on an asymptotic regime, and therefore require the system to be sufficiently large for the approximation to become meaningful. We allow for (at least) $r=50 000$ EV users in this infrastructure. The effective loads at the stations in this case are
	\begin{align*}
	\left(\frac{p_j \lambda r}{\mu}\right)_{j=1,\ldots,5}=(62.5,375,375,250,187.5),
	\end{align*}
	and we note that due to the QED provisioning rule in~\eqref{eq:QEDscalingNetworkModel}, the numbers of charging points and spare batteries are close to these values. Obviously, the number of resources are integer values, and in our simulation experiments we choose
	\begin{align*}
	F^r &= \left(\left\lceil \frac{\lambda r}{\mu} + \gamma \sqrt{\frac{\lambda r}{\mu}}\right\rceil\right)_{j=1,\ldots,5} , \qquad B^r = \left(\left\lceil \frac{\lambda r}{\mu} + \beta \sqrt{\frac{\lambda r}{\mu}}\right\rceil\right)_{j=1,\ldots,5}
	\end{align*}
	
	\subsubsection{State space collapse for exponential charging times}
	A first-order approximation for the queue length process is implied by the fluid result in Theorem~\ref{thm:FluitLimitResult}. We validate this approximation for the above-described setting, with initial queue length $Q_i^r(0)=150$ for all stations $i=1,\ldots,5$. That is, only station~1 is initially overloaded, while all other station are underloaded. The equations in Theorem~\ref{thm:FluitLimitResult} together with the Lipschitz continuity describe a unique fluid limit with the given initial queue length. This yields the approximations
	\begin{align*}
	Q_i^r(t) \approx r \bar{Q}_(t), \quad j=1,\ldots,S.
	\end{align*}
	In particular, in the case when the initial queue length is $Q_i^r(0)=150$ for all stations $i=1,\ldots,5$, this results in the approximations
	\begin{align*}
	Q_1^r(t) \approx \left\{ \begin{array}{ll}
	150-62.5t & \textrm{if } t \leq t_3, \\
	62.5 e^{1.4-t} & \textrm{if } t_3 \leq t \leq t_4, \\
	62.5 + \frac{195}{64} e^{1.4-t}-\frac{517}{16} e^{-t}& \textrm{otherwise},
	\end{array}\right.
	\end{align*}
	
	\begin{align*}
	Q_2^r(t) \approx Q_3^r(t) \approx \left\{ \begin{array}{ll}
	498.5+0.625t-348.5 e^{-t} & \textrm{if } t \leq t_2, \\
	\frac{75 t}{152}+\frac{14955}{38}-\frac{7755}{38} e^{-t} & \textrm{if } t_2 \leq t \leq t_3, \\
	\frac{7500}{19}-\frac{75}{152} e^{1.4-t}-\frac{7755}{38}e^{-t} & \textrm{if } t_3 \leq t \leq t_4, \\
	375 +\frac{585}{32} e^{1.4-t}-\frac{1551}{8} e^{-t} & \textrm{otherwise},
	\end{array}\right.
	\end{align*}
	
	\begin{align*}
	Q_4^r(t) \approx \left\{ \begin{array}{ll}
	249.25 + 0.3125 t  - 99.25 e^{-t} & \textrm{if } t \leq t_1, \\
	\frac{997}{7}+\frac{5}{28}t+29 e^{-t} & \textrm{if } t_1 \leq t \leq t_2, \\
	\frac{4985}{19} + \frac{25}{76} t -\frac{2585}{19} e^{-t} & \textrm{if } t_2 \leq t \leq t_3, \\
	\frac{5000}{19}-\frac{25}{76} e^{1.4-t}-\frac{2585}{19}  e^{-t} & \textrm{if } t_3 \leq t \leq t_4, \\
	250+\frac{195}{16} e^{1.4-t}-129.25 e^{-t} & \textrm{otherwise},
	\end{array}\right.
	\end{align*}
	
	\begin{align*}
	Q_5^r(t) \approx \left\{ \begin{array}{ll}
	150 e^{-t} & \textrm{if } t \leq t_1, \\
	\frac{15 t}{112}+\frac{2991}{28}+\frac{87}{4}e^{-t} & \textrm{if } t_1 \leq t \leq t_2, \\
	\frac{14955}{76}+\frac{75}{304}t-\frac{7755}{76}e^{-t}& \textrm{if } t_2 \leq t \leq t_3, \\
	\frac{3750}{19}-\frac{75}{304} e^{1.4-t}-\frac{7755}{76}e^{-t} & \textrm{if } t_3 \leq t \leq t_4, \\
	187.5+\frac{585}{64} e^{1.4-t}-\frac{1551}{16} e^{-t} & \textrm{otherwise},
	\end{array}\right.
	\end{align*}
	where $t_1\approx 0.1826$, $t_2 \approx 0.3189$, $t_3=1.4$ and $t_4\approx 1.4758$. The times $t_i, i=1,2,4$ correspond to the times where two stations (approximately) have the same relative queue lengths, and $t_3$ is the moment where the number of EVs in need of recharging is (approximately) equal to the number of stations/spare batteries.
	
	\begin{figure}[htb]
		\centering
		\includegraphics[width=12cm]{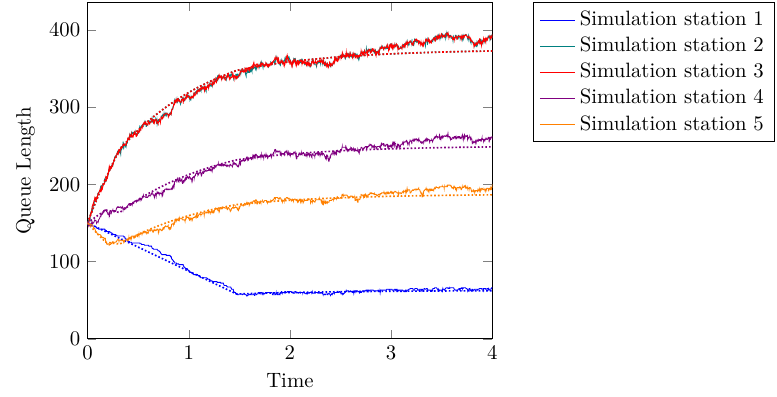}
		\caption{Sample path of the queue lengths when $Q^r(0)=(150,150,150,150,150)$.}
		\label{fig:QueueLengthExponential}
	\end{figure}
	
	A sample path comparison with its fluid approximation (dotted lines) is graphically illustrated in Figure~\ref{fig:QueueLengthExponential}. We observe that the fluid limit approximations capture the typical values of the actual queue length process quite accurately. We observe apparent fluctuations around its approximation, and we note that these become relatively small as $r$ grows large.
	
	\begin{figure}[htb]
		\centering
		\includegraphics[width=12cm]{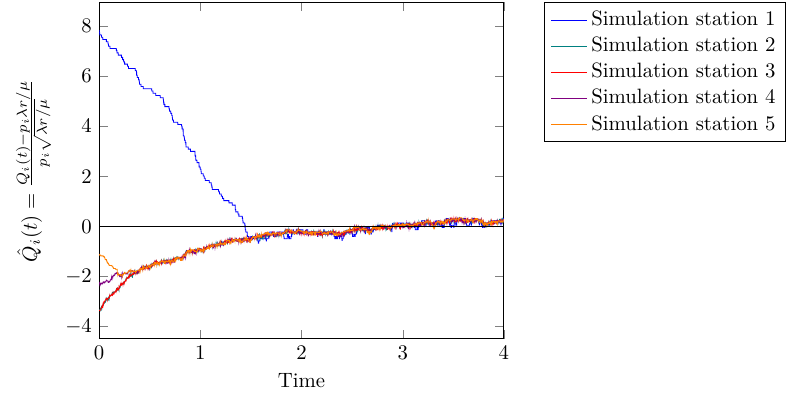}
		\caption{Sample path of the diffusion-scaled queue lengths when $Q^r(0)=(150,150,150,150,150)$.}
		\label{fig:QueueLengthExponentialDiffusion}
	\end{figure}
	
	To observe the state-space collapse, we plot the same sample path by its diffusion scaling, see Figure~\ref{fig:QueueLengthExponentialDiffusion}. Indeed, around $t_4 \approx 1.4758$ the diffusion-scaled queue lengths appear to become close and remain nearly equal to one another after this time. In addition, as time moves, on the diffusion-scaled queue lengths fluctuate around zero.

	\subsubsection{Performance measures}
	Our results imply approximations for performance measures such as the waiting probability and waiting time, see Corollary~\ref{cor:PerformanceMeasures}. In particular, the state-space collapse result implies that the performance at all stations is approximately the same, and can be approximated by the closed-form expressions as given in Corollary~\ref{cor:PerformanceMeasures}.
	
	\begin{figure}[htb]
		\centering
		\includegraphics[width=12cm]{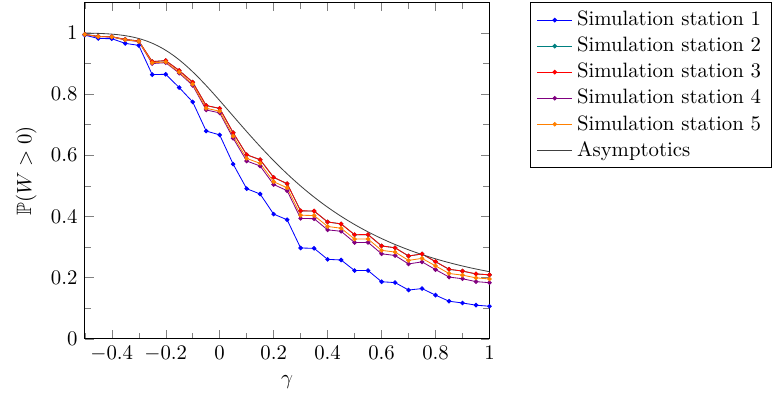}
		\caption{Waiting probabilities with respect to its asymptotic expression when $\beta=1$.}
		\label{fig:WaitingProb}
	\end{figure}
	
	In Figure~\ref{fig:WaitingProb}, we plotted the waiting probabilities of all stations in the case of $2 500 000$ EV arrivals averaged over $20$ samples for the large-scaled system. We point out that the stair-type effect appearing in the waiting probabilities is due to the ceiling of the number of resources at the stations. Moreover, as $r$ is finite and we use the ceiling function, the waiting times are not all exactly equal, which is most apparent for station~1. This is also reflected in Figure~\ref{fig:QueueLengthExponentialDiffusion}, where a closer view suggests that the diffusion-scaled queue length at station 1 being smaller than the queue length at another station (often station~2 or station~3). As $r$ grows large, the waiting probabilities do grow closer and move near to their asymptotic expressions. Still, the waiting probabilities are typically below their asymptotic expressions. This implies that the provisioning rules~\eqref{eq:QEDscalingNetworkModel} guarantee that a desired waiting probability is achieved.
	
	\subsection{Universality result for charging time distribution}
	In order to be able to rigorously prove the state-space and consequential results, we assumed exponential charging times in our framework. Yet, extensive simulation experiments suggest that these results hold for any charging time distribution with finite mean and variance. In Figure~\ref{fig:FluidDeterministic} and~\ref{fig:FluidUniform}, we consider the system setting as described in Section~\ref{sec:LargeScaleSystemSimulations}. It appears that similar behavior occurs on fluid scale in case of deterministic and uniformly distributed charging times as for the exponential case.
	
	\begin{figure}[!htb]
		\centering
		\includegraphics[width=12cm]{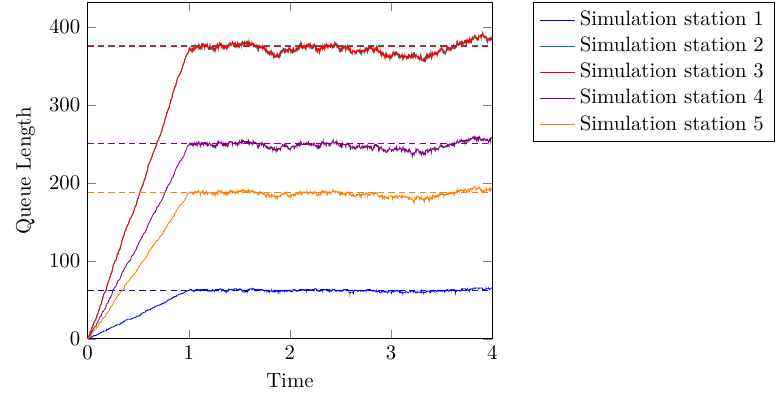}
		\caption{Sample paths of the queue lengths for charging times equal to one when $Q^r(0)=(0,0,0,0,0)$.}
		\label{fig:FluidDeterministic}
	\end{figure}
	
	\begin{figure}[!htb]
		\centering
		\includegraphics[width=12cm]{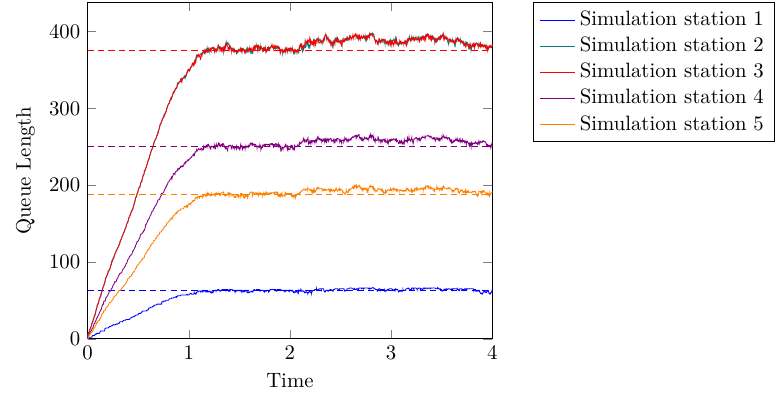}
		\caption{Sample paths of the queue lengths for charging distribution uniform $U(0.75,1.25)$ when $Q^r(0)=(0,0,0,0,0)$.}
		\label{fig:FluidUniform}
	\end{figure}
	
	When the queue lengths are initially zero, the system behaves close to its invariant state for $t \geq 1$. Similarly to the setting with exponential charging times, the maximum difference between the diffusion-scaled queue length behaves quite erratically, see Figure~\ref{fig:GTDifference}. Still, the differences are very small, and grow smaller as $r$ grows larger suggesting that state-space collapse also holds in this setting. That is, the system behaves similarly to the situation when there would be a single station with an aggregated number of charging points and spare batteries and a charging time distribution as at the individual stations. Consequently, performance measures such as  waiting probability and expected waiting time are approximately equal to their equivalents in a single-station system.
	
	\begin{figure}[!htb]
		\centering
		\includegraphics[width=11cm]{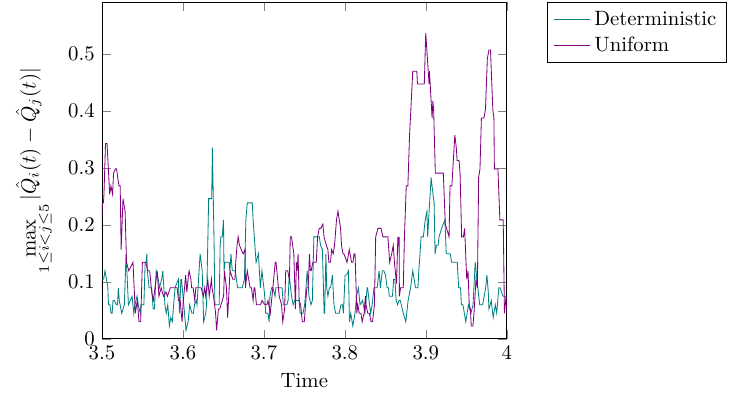}
		\caption{Maximum distance between queue lengths for non-exponential charging times.}
		\label{fig:GTDifference}
	\end{figure}

	\subsection{The role of system size}
	In the previous sections, we commented at a view points that the differences between the diffusion-scaled queue lengths are small and fluctuate erratically among each other when one would zoom in on this domain. Obviously, the differences between the diffusion-scaled queue lengths are not arbitrarily small since $r$ is finite. Even if the diffusion-scaled queue lengths at all stations are the same, and an arriving EV moves to station~1, this causes a discrepancy of $1/(p_1 \sqrt{\lambda r/\mu})\approx 0.5657$ in the described setting in Section~\ref{sec:LargeScaleSystemSimulations}. Theorem~\ref{thm:StrongSSC} implies that the distance between the queue lengths become smaller as the number of EV users $r$ grows large. To illustrate this notion, we consider the maximum difference between the queue lengths over a finite interval $T=1$ in Figure~\ref{fig:MaxDiffExponential}, which is monotonically decreasing in $r$. In addition, we observe that the average maximum distance, i.e.\ $1/T \int_{0}^T \max_{1 \leq i < j \leq S}\{|Q_i(t)-Q_j(t)|\, \} dt$, is also monotonically decreasing in $r$, and is not excessively smaller than the maximum distance of the interval.
	
	\begin{figure}[htb]
		\centering
		\includegraphics[width=12cm]{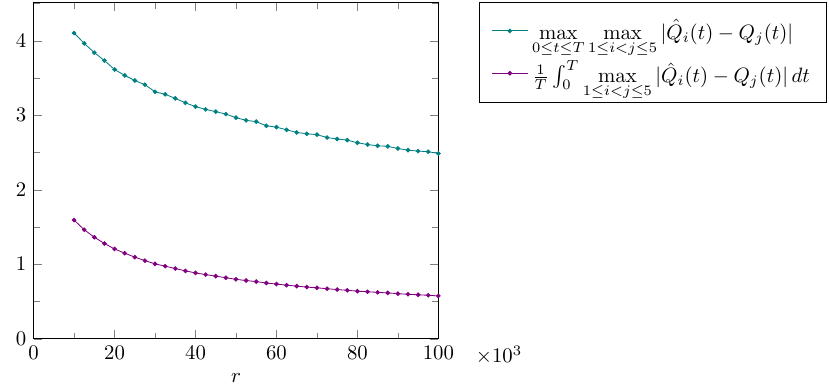}
		\caption{Maximum queue length measures for $T=1$ averaged over 10 000 samples.}
		\label{fig:MaxDiffExponential}
	\end{figure}
	
	Summarizing, as the system size increases, the accuracy of the approximations improve. However, one can imagine that a real-life battery swapping infrastructure is not of the scale as discussed in the previous sections. Therefore, we consider how our results hold up in a more realistic for an EV battery swapping infrastructure.
	
	\subsection{Moderate-sized system}
	We consider the following setting. We have the same network structure as given in Figure~\ref{fig:StationNetworkExperiments}, but with other arrival probabilities. More specifically, we assume $p_{12}=p_{23}=p_{24}=p_{25}=0.2$ and $p_{13}=p_{34}=0.1$, giving an effective arrival probability of $(p_1,\ldots,p_5)=(0.1,0.25,0.3,0.2,0.15)$. For the other parameters, we assume that recharging takes four hours on average ($\mu=0.25$) and every EV user returns for recharging services after every 50 hours on average ($\lambda=0.02$). We assume that there our infrastructure consists of a thousand electric vehicles ($r=1000$). The effective load at the stations are
	\begin{align*}
	\left(\frac{p_j \lambda r}{\mu}\right)_{j=1,\ldots,5}=(8,20,24,16,12),
	\end{align*}
	
	Note that $\lambda r / \mu = 80$ and $\sqrt{\lambda r / \mu} = 4\sqrt{5} \approx 8.94$. Relatively, their sizes are much closer to one another than when $r$ becomes larger, and this will also have its impact on the behavior.
	
	Due to the smaller system size, we can expect that the fluctuations of around its fluid limit is relatively larger. Indeed, if one plots a sample path for this system with initial queue length $Q(0)=(0,0,0,0,0)$, we observe that the fluctuations to its corresponding fluid limit approximation are significant, see Figure~\ref{fig:QLModerateSizeLargeBetaAndGamma} and~\ref{fig:QLModerateSizeLargeBetaGamma0}. Moreover, in Figure~\ref{fig:QLModerateSizeLargeBetaGamma0}, the queue lengths even seem to lie above the fluid limit results. We point out that this is a consequence of the high waiting probabilities for these parameter settings. Moreover, since the fluctuations are of a significant size (relatively), this leads to sample paths that appear quite far off (above) its fluid limit approximation at first glance.
	
	\begin{figure}[htb!]
		\centering
		\includegraphics[width=10.5cm]{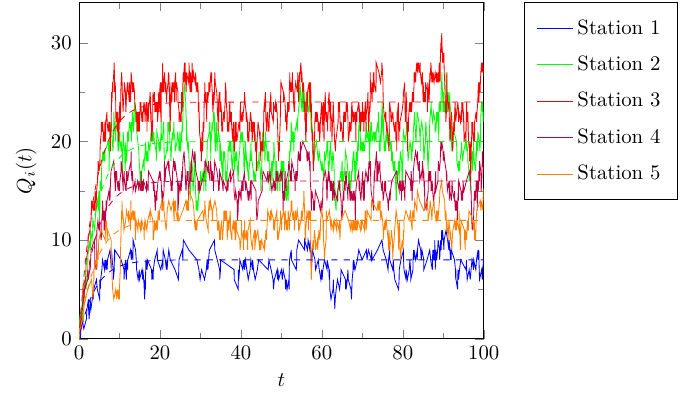}
		\caption{Queue length behavior for moderate-sized system with $\beta=\gamma \sqrt{5}$.}
		\label{fig:QLModerateSizeLargeBetaAndGamma}
	\end{figure}
	
	\begin{figure}[htb!]
		\centering
		\includegraphics[width=10.5cm]{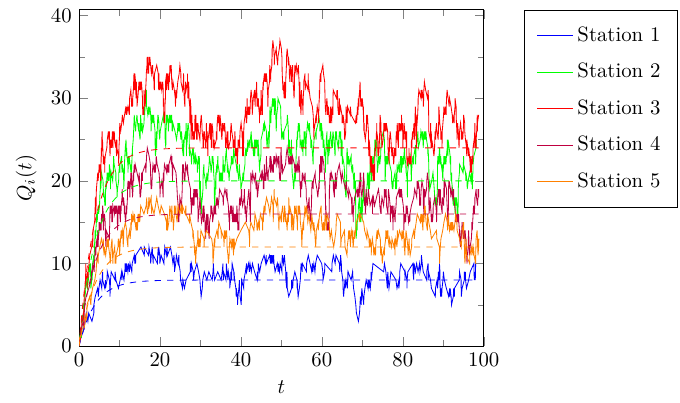}
		\caption{Queue length behavior for moderate-sized system with $\beta= \sqrt{5}$ and $\gamma=0$.}
		\label{fig:QLModerateSizeLargeBetaGamma0}
	\end{figure}
	
	To see whether a load-balancing effect still takes place for a moderate-sized system, we should consider the diffusion scaled queue lengths, see Figure~\ref{fig:QLModerateSizeDiffusion}. Numerous experiments, including the setting as in Figure~\ref{fig:QLModerateSizeDiffusion} suggest that even for these settings, the effect of state-space collapse is very much visible. In other words, there is still a strong load-balancing effect present that leads to the occupation level at the different stations stay close to one another. In turn, this lead that performance, e.g. waiting probability and waiting times, are comparable at the different stations at all times.
	
	\begin{figure}[htb!]
		\centering
		\includegraphics[width=10.5cm]{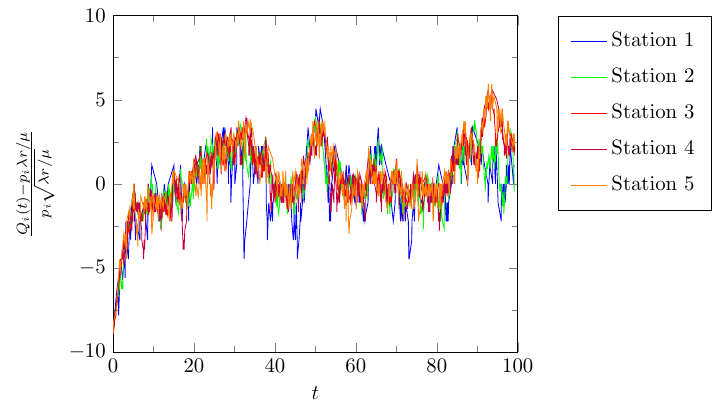}
		\caption{Queue length behavior for moderate-sized system with $\beta= \sqrt{5}$ and $\gamma=0$.}
		\label{fig:QLModerateSizeDiffusion}
	\end{figure}
	
	\subsection*{Acknowledgment}
	This research is supported by the Netherlands Organisation for Scientific Research through the programmes Gravitation NETWORKS grant 024.002.003, MEERVOUD grant 632.003.002, and VICI grant 639.033.413. The work of Seva Shneer is supported by the Mathematical Center in Akademgorodok under agreement No. 075-15-2019-1675 with the Ministry of Science and Higher Education of the Russian Federation.

	

	\clearpage
	
	\bibliographystyle{plain}
	\bibliography{bibEV}

	\newpage

\appendix	
		\section{Fluid and diffusion limit proofs for single station system} \label{app:SingleMain}
		In this appendix, we provide the proofs for the results in case that $S=1$.
		
		\subsection{Fluid and diffusion limits} \label{app:SingleMainFluidDiffusion}
		First, we derive the fluid limit.
		
\begin{proof}[Proof of Proposition~\ref{prop:SingleFluidLimitProcess}]
		We observe that the $Q^r(\cdot)$ is a birth-death process with state space $\mathcal{Q}^r=\{0,1,\ldots,B^r+r\}$, arrival (birth) rate $\lambda_r(j)=\lambda \left(r-(j-B^r)^+\right)$ and service (death) rate $\mu_r(j) = \mu \min\{j,F^r\}$ for all $j \in \mathcal{Q}^r$. The fluid-scaled process $\bar{Q}^r$ therefore has state space $\{0,1/r,\ldots,(B^r+r)/r\}$, and drift and infinitesimal variance functions
		\begin{align*}
		m_r(x) = \frac{\lambda_r( \lfloor r x \rfloor)}{r} - \frac{\mu_r( \lfloor r x \rfloor)}{r} = \lambda - \frac{\lambda(\lfloor r x \rfloor -B^r)^+}{r}- \frac{\mu \min\{\lfloor r x \rfloor,F^r\}}{r}
		\end{align*}
		and
		\begin{align*}
		\sigma^2_r(x) = \frac{\lambda_r( \lfloor r x \rfloor)}{r^2} + \frac{\mu_r( \lfloor r x \rfloor)}{r^2} = \lambda - \frac{\lambda(\lfloor r x \rfloor -B^r)^+}{r} + \frac{\mu \min\{\lfloor r x \rfloor,F^r\}}{r}.
		\end{align*}
		Taking the limit yields under scaling rule~\eqref{eq:QEDscalingSingleBSSModel}
		\begin{align*}
		m(x) := \lim_{r \rightarrow \infty} m_r(x) = \left\{ \begin{array}{ll}
		\lambda -\mu x & \textrm{if } x < \lambda /\mu, \\
		-\lambda x + \lambda^2/\mu & \textrm{if } x \geq \lambda /\mu,
		\end{array}\right.
		\end{align*}
		and
		\begin{align*}
		\sigma^2(x):= \lim_{r \rightarrow \infty} \sigma^2_r(x) = 0.
		\end{align*}
		That is, we obtain a degenerate limiting diffusion, and hence the limiting initial point will yield a deterministic path for the limiting process $\bar{Q}$. The mean directly provides the ODE that the limiting process $Q$ satisfies. The limiting value (as $t \rightarrow \infty)$ of the limiting process is also correct due to the following reasoning. For all $t \geq 0$ for which $\bar{Q}(t) < \frac{\lambda}{\mu}$, we observe $d \bar{Q}(t)/dt > 0$. Hence, if $\bar{Q}(0) < \lambda /\mu$, then the limiting process follows by the deterministic path
		\begin{align*}
		\bar{Q}(t) = \frac{\lambda}{\mu}+\left(\bar{Q}(0)-\frac{\lambda}{\mu}\right)e^{-\mu t}.
		\end{align*}
		Similarly, for all $t \geq 0$, if $\bar{Q}(t) > \lambda /\mu$, then $d \bar{Q}(t)/dt < 0$. Hence if $\bar{Q}(0) > \lambda /\mu$ , then the limiting process follows the deterministic path
		\begin{align*}
		\bar{Q}(t) = \frac{\lambda}{\mu}+\left(\bar{Q}(0)-\frac{\lambda}{\mu}\right)e^{-\lambda t}.
		\end{align*}
		Finally, if $\bar{Q}(0)=\lambda/\mu$, then $d \bar{Q}(t) /dt =0$ for all $t \geq 0$, which yields $\bar{Q}(t)=\lambda/\mu$ for all $t\geq 0$.
\hfill\end{proof}
		
		The second result involves the diffusion limit.
		
\begin{proof}[Proof of Theorem~\ref{thm:DiffusionLimitProcess}]
		The infinitesimal mean of the centered scaled process is given by
		\begin{align*}
		m_r(x) &= \frac{\lambda_r\left(\lfloor \lambda r / \mu + x \sqrt{\lambda r / \mu } \rfloor \right) - \mu_r\left(\lfloor \lambda r / \mu + x \sqrt{\lambda r / \mu } \rfloor \right)}{\sqrt{\lambda r/ \mu}} \rightarrow  -\lambda(x-\beta)^+ -\mu \min\{x,\gamma\}, \\
		\end{align*}
		and the infinitesimal variance is given by
		\begin{align*}
		\sigma^2_r(x) &= \frac{\lambda_r\left(\lfloor \lambda r / \mu + x \sqrt{\lambda r / \mu } \rfloor \right) + \mu_r\left(\lfloor \lambda r / \mu + x \sqrt{\lambda r / \mu } \rfloor \right)}{\lambda r/ \mu} \rightarrow \frac{2\lambda r}{\lambda r/ \mu} = 2\mu.
		\end{align*}
		Using Equations~(28) and~(33) in~\cite{BrowneWhitt1995}, this implies that the limiting process is a piecewise-linear diffusion process with steady state density given by~\eqref{eq:DensityQEDLimited1}, where $\alpha_i$, $i=1,2,3$ is the probability that the steady state process is in that interval. Equations~(5) and (6) from~\cite{BrowneWhitt1995} yield the steady state probabilities.
\hfill\end{proof}
		
		\subsection{Steady state limits} \label{app:SingleMainSteadyStateLimits}
		Theorem~\ref{thm:StStateToDiffusion} states that it is allowed to interchange the limits. To prove this result, we need to show that the steady state distribution (as $t\rightarrow \infty$) converges to the same distribution as $r \rightarrow \infty$ as the diffusion limiting result. First we derive the limiting steady state distribution conditioned to be in one of the intervals $[0,F^r]$, $[F^r,B^r]$, and $[B^r,B^r+r]$.
		
		\begin{lemma}
			Suppose $S=1$. Under scaling rules~\eqref{eq:QEDscalingSingleBSSModel}, the following limiting results hold as $r \rightarrow \infty$:
			\begin{align}
			\begin{array}{ll}
			\Prob\left(Q^r < \frac{\lambda r}{\mu} + x \sqrt{\frac{\lambda r}{\mu}} \big| Q^r < F^r \right) \rightarrow \frac{\Phi(x)}{\Phi(\gamma)}, & x<\gamma, \\
			\Prob\left(Q^r < \frac{\lambda r}{\mu} + x \sqrt{\frac{\lambda r}{\mu}} \big| F^r \leq Q^r < B^r \right) \rightarrow  \left(1- e^{-\gamma(x-\gamma)} \right)  \left(1- e^{-\gamma(\beta-\gamma)} \right)^{-1}  , & \gamma \leq x < \beta, \\
			\Prob\left(Q^r \geq \frac{\lambda r}{\mu} + x \sqrt{\frac{\lambda r}{\mu}} \big| Q^r \geq B^r \right) \rightarrow \Phi\left((\beta-x)\sqrt{\frac{\lambda}{\mu}}-\gamma\sqrt{\frac{\mu}{\lambda}} \right) \Phi\left(-\gamma\sqrt{\frac{\mu}{\lambda}}\right)^{-1}, & x \geq \beta.
			\end{array}
			\end{align}
			\label{lem:StStateConditionedQEDDistribution}
		\end{lemma}
		
\begin{proof}
		The first expression follows from the central limit theorem and the properties of the Poisson distribution,
		\begin{align*}
		\Prob \left(Q^r < \frac{\lambda r}{\mu} + x \sqrt{\frac{\lambda r}{\mu}} \bigg| Q^r < F^r \right) &= \frac{\Prob\left(Q^r < \frac{\lambda r}{\mu} + x \sqrt{\frac{\lambda r}{\mu}} \right)}{\Prob\left( Q^r < F^r \right)} = \frac{\sum_{k=0}^{\lambda r/\mu + x\sqrt{\lambda r/\mu}-1} \pi_k }{\sum_{k=0}^{F^r-1} \pi_k}\\
		&= \frac{\Prob\left( \textrm{Pois}(\lambda r/ \mu) \leq  \lambda r/\mu + x\sqrt{\lambda r/\mu}-1 \right)}{\Prob\left( \textrm{Pois}(\lambda r/ \mu) \leq  \lambda r/\mu + \gamma\sqrt{\lambda r/\mu}-1 \right)} \rightarrow \frac{\Phi(x)}{\Phi(\gamma)}.
		\end{align*}
		The second expression can be obtained using geometric series,
		\begin{align*}
		\Prob&\left(Q^r < \frac{\lambda r}{\mu} + x \sqrt{\frac{\lambda r}{\mu}} \bigg| F^r \leq Q^r < B^r \right) = \frac{\Prob\left(F^r\leq Q^r < \frac{\lambda r}{\mu} + x \sqrt{\frac{\lambda r}{\mu}} \right)}{\Prob\left(F^r \leq Q^r < B^r \right)} = \frac{\sum_{k=F^r}^{\lambda r/\mu + x\sqrt{\lambda r/\mu}-1} \pi_k }{\sum_{k=F^r}^{B^r-1} \pi_k} \\
		&= \frac{\sum_{k=F^r}^{\lambda r/\mu + x\sqrt{\lambda r/\mu}-1} \left( \frac{\lambda r}{\mu F^r}\right)^k }{\sum_{k=F^r}^{B^r-1} \left( \frac{\lambda r}{\mu F^r}\right)^k} = \frac{1-\left( \frac{\lambda r}{\mu F^r}\right)^{\lambda r/\mu + x\sqrt{\lambda r/\mu}} -1 + \left( \frac{\lambda r}{\mu F^r}\right)^{F^r}  }{1-\left( \frac{\lambda r}{\mu F^r}\right)^{B^r} - 1 + \left( \frac{\lambda r}{\mu F^r}\right)^{F^r}}\\
		& = \frac{\left(  1+\frac{\gamma}{\sqrt{\lambda r /\mu}}\right)^{ -x\sqrt{\lambda r/\mu}} -\left(  1+\frac{\gamma}{\sqrt{\lambda r /\mu}}\right)^{ -\gamma \sqrt{\lambda r/\mu}} }{\left( 1+\frac{\gamma}{\sqrt{\lambda r /\mu}}\right)^{ -\beta\sqrt{\lambda r/\mu}} - \left(  1+\frac{\gamma}{\sqrt{\lambda r /\mu}}\right)^{ -\gamma \sqrt{\lambda r/\mu}}} \rightarrow  \frac{e^{-\gamma x} -e^{-\gamma^2}}{e^{-\gamma \beta}-e^{-\gamma^2}} = \frac{1- e^{-\gamma (x-\gamma)} }{1-e^{-\gamma (\beta-\gamma)}}.
		\end{align*}
		The final expression again uses the central limit theorem and the properties of the Poisson distribution. That is,
		\begin{align*}
		\Prob&\left(Q^r \geq \frac{\lambda r}{\mu} + x \sqrt{\frac{\lambda r}{\mu}} \bigg| Q^r \geq B^r \right) = \frac{\Prob\left(Q^r \geq \frac{\lambda r}{\mu} + x \sqrt{\frac{\lambda r}{\mu}}  \right)}{\Prob\left(Q^r \geq B^r \right)} = \frac{\sum_{k=\lambda r/\mu + x\sqrt{\lambda r/\mu}}^{B^r+r} \frac{1}{(r+B^r-k)!} \left(\frac{\lambda}{\mu F^r}\right)^k }{\sum_{k=B^r}^{B^r+r}\frac{1}{(r+B^r-k)!} \left(\frac{\lambda}{\mu F^r}\right)^k} \\
		&= \frac{\sum_{k=0}^{B^r+r-\lambda r/\mu - x\sqrt{\lambda r/\mu}} \frac{1}{k!} \left(\frac{\lambda}{\mu F^r}\right)^{B^r+r-k} }{\sum_{k=0}^{r}\frac{1}{k!} \left(\frac{\lambda}{\mu F^r}\right)^{B^r+r-k} }= \frac{\sum_{k=0}^{r + (\beta- x)\sqrt{\lambda r/\mu}} \frac{1}{k!} \left(\frac{\mu F^r}{\lambda}\right)^{k} e^{-\mu F^r/\lambda}}{\sum_{k=0}^{r}\frac{1}{k!} \left(\frac{\mu F^r}{\lambda}\right)^{k} e^{-\mu F^r/\lambda}}.
		\end{align*}
		Since for every $z \in \mathbb{R}$,
		\begin{align*}
		\Prob\left(\textrm{Pois}\left(\mu F^r/\lambda \right) \leq r +z \sqrt{\lambda r /\mu} \right) &= \Prob\left(\frac{\textrm{Pois}\left(\mu F^r/\lambda \right) - r - \gamma \sqrt{\frac{\mu r}{\lambda}}}{\sqrt{r}} \leq z \sqrt{\frac{\lambda}{\mu}} -\gamma\sqrt{\frac{\mu}{\lambda}} \right) \\
		&\rightarrow \Phi\left( z \sqrt{\frac{\lambda}{\mu}} -\gamma\sqrt{\frac{\mu}{\lambda}}\right)
		\end{align*}
		as $r \rightarrow \infty$, we obtain
		\begin{align*}
		\Prob&\left(Q^r \geq \frac{\lambda r}{\mu} + x \sqrt{\frac{\lambda r}{\mu}} \bigg| Q^r \geq B^r \right) \rightarrow \Phi\left((\beta-x) \sqrt{\frac{\lambda}{\mu}} -\gamma\sqrt{\frac{\mu}{\lambda}}\right) \Phi\left( -\gamma\sqrt{\frac{\mu}{\lambda}}\right)^{-1}.
		\end{align*}
\hfill\end{proof}
		
		To obtain the limiting distribution of the steady state distribution as $r \rightarrow \infty$, it remains to derive the asymptotic behavior of the probabilities that the queue length process is in the intervals $[0,F^r]$, $[F^r,B^r]$, and $[B^r,B^r+r]$. To do so, we use the following three lemmas.
		
		\begin{lemma}
			Under scaling rules~\eqref{eq:QEDscalingSingleBSSModel}, as $r \rightarrow \infty$,
			\begin{align*}
			\sum_{k=0}^{F^r-1} \frac{\left( \frac{\lambda r}{\mu} \right)^k}{k!} e^{-\lambda r / \mu} \rightarrow \Phi(\gamma).
			\end{align*}
			\label{lem:LimitedA1}
		\end{lemma}
		
\begin{proof}
		This follows from the central limit and the properties of the Poisson distribution,
		\begin{align*}
		\sum_{k=0}^{F^r-1} \frac{\left( \frac{\lambda r}{\mu} \right)^k}{k!} e^{-\lambda r / \mu} = \Prob\left( \textrm{Pois}(\lambda r/ \mu) <  \lambda r/\mu + \gamma\sqrt{\lambda r/\mu} \right) \rightarrow \Phi(\gamma).
		\end{align*}
\hfill\end{proof}

		\begin{lemma}
			Under scaling rules~\eqref{eq:QEDscalingSingleBSSModel} with $\gamma \neq 0$, as $r \rightarrow \infty$,
			\begin{align*}
			\sum_{k=F^r}^{B^r-1} \left(\frac{\lambda r}{\mu F^r} \right)^k \frac{{F^r}^{F^r}}{F^r!} e^{-\frac{\lambda r}{\mu}} \rightarrow \frac{1}{\sqrt{2\pi}} \frac{1}{\gamma}\left( e^{-\gamma^2/2} - e^{-\beta \gamma+\gamma^2/2} \right) = \frac{\phi(\gamma)}{\gamma} \left(1-e^{-\gamma(\beta-\gamma)}\right).
			\end{align*}
			If $\gamma=0$, then
			\begin{align*}
			\sum_{k=F^r}^{B^r-1} \left(\frac{\lambda r}{\mu F^r} \right)^k \frac{{F^r}^{F^r}}{F^r!} e^{-\frac{\lambda r}{\mu}} \rightarrow \frac{\beta}{\sqrt{2\pi}}.
			\end{align*}
			\label{lem:LimitedA2}
		\end{lemma}
		
\begin{proof}
		If $\gamma \neq 0$, we observe that
		\begin{align*}
		\sum_{k=F^r}^{B^r-1} \left(\frac{\lambda r}{\mu F^r} \right)^k = \frac{\left(\frac{\lambda r}{\mu F^r} \right)^{F^r+1} - \left(\frac{\lambda r}{\mu F^r} \right)^{B^r} }{1-\left(\frac{\lambda r}{\mu F^r} \right)} = \left(\frac{\lambda r}{\mu F^r} \right)^{\frac{\lambda r}{\mu}} \frac{\left(\frac{\lambda r}{\mu F^r} \right)^{\gamma \sqrt{\lambda r /\mu}+1} - \left(\frac{\lambda r}{\mu F^r} \right)^{\beta \sqrt{\lambda r /\mu}} }{1-\left(\frac{\lambda r}{\mu F^r} \right)}
		\end{align*}
		due to geometric series. Moreover,
		\begin{align*}
		\left(\frac{\lambda r}{\mu F^r} \right)^{\gamma \sqrt{\lambda r /\mu}+1} - \left(\frac{\lambda r}{\mu F^r} \right)^{\beta \sqrt{\lambda r /\mu}} &= \left(1+\frac{\gamma}{\sqrt{\lambda r/\mu}} \right)^{-\gamma \sqrt{\lambda r /\mu}+1} - \left(1+\frac{\gamma}{\sqrt{\lambda r/\mu}}\right)^{-\beta \sqrt{\lambda r /\mu}}\\
		&\rightarrow  e^{-\gamma^2} - e^{-\beta \gamma},
		\end{align*}
		and due to geometric series,
		\begin{align*}
		1-\left(\frac{\lambda r}{\mu F^r} \right) = 1-\left(1+\frac{\gamma}{\sqrt{\lambda r/\mu}} \right)^{-1} = -\frac{\gamma}{\sqrt{\lambda r/ \mu}} + O(r^{-1}).
		\end{align*}
		In addition,
		\begin{align*}
		\left(\frac{\lambda r}{\mu F^r} \right)^{\frac{\lambda r}{\mu}} &= \left(1+\frac{\gamma}{\sqrt{\lambda r/\mu}} \right)^{\frac{\lambda r}{\mu}} = \textrm{exp}\left\{ -\frac{\lambda r}{\mu} \left(\frac{\gamma}{\sqrt{\lambda r/\mu}} -\frac{1}{2} \frac{\gamma^2}{\lambda r/\mu} + O(r^{-3/2}) \right)  \right\} \\
		&= \textrm{exp}\left\{ -\sqrt{\frac{\lambda r}{\mu}} \gamma + \frac{\gamma^2}{2} + O(r^{-1/2})  \right\}.
		\end{align*}
		Finally, using Stirling's approximation
		\begin{align*}
		\frac{{F^r}^{F^r}}{F^r!} e^{-\frac{\lambda r}{\mu}} \sim \frac{1}{\sqrt{2 \pi F^r}}e^{-\frac{\lambda r}{\mu}+F^r} \sim \frac{1}{\sqrt{2 \pi}\sqrt{ \lambda r /\mu}}e^{\gamma \sqrt{\frac{\lambda r}{\mu}}}.
		\end{align*}
		Combining the expressions yields
		\begin{align*}
		\sum_{k=F^r}^{B^r-1} \left(\frac{\lambda r}{\mu F^r} \right)^k \frac{{F^r}^{F^r}}{F^r!} e^{-\frac{\lambda r}{\mu}}  \rightarrow \frac{1}{\sqrt{2\pi}} \frac{1}{\gamma} e^{\gamma^2/2} \left( e^{-\gamma^2} - e^{-\beta \gamma} \right) =  \frac{1}{\sqrt{2\pi}} \frac{1}{\gamma}\left( e^{-\gamma^2/2} - e^{-\beta \gamma+\gamma^2/2} \right).
		\end{align*}
		If $\gamma = 0$, using Stirling's approximation, we directly obtain
		\begin{align*}
		\sum_{k=F^r}^{B^r-1} \left(\frac{\lambda r}{\mu F^r} \right)^k \frac{{F^r}^{F^r}}{F^r!} e^{-\frac{\lambda r}{\mu}} =  (B^r-F^r-1) \frac{{F^r}^{F^r}}{F^r!} e^{-\frac{\lambda r}{\mu}} \rightarrow \frac{\beta}{\sqrt{2\pi}}.
		\end{align*}
\hfill\end{proof}

		\begin{lemma}
			Under scaling rules~\eqref{eq:QEDscalingSingleBSSModel}, as $r \rightarrow \infty$,
			\begin{align*}
			\frac{r^{B^r} r! {F^r}^{F^r} }{F^r!} e^{-\frac{\lambda r}{\mu}} \sum_{k=B^r}^{B^r+r} \frac{1}{(r+B^r-k)!}\left(\frac{\lambda}{\mu F^r} \right)^k &\rightarrow \sqrt{\frac{\mu}{\lambda}} e^{\frac{\gamma^2}{2}-\beta \gamma +\frac{\mu}{\lambda} \frac{\gamma^2}{2}} \Phi\left( - \sqrt{\frac{\mu}{\lambda} \gamma} \right) \\
			&=\phi(\gamma) e^{-\gamma(\beta-\gamma)}\sqrt{\frac{\mu}{\lambda}} \phi\left( \sqrt{\frac{\mu}{\lambda}}\gamma \right)^{-1} \Phi\left( -\sqrt{\frac{\mu}{\lambda}}\gamma\right).
			\end{align*}
			\label{lem:LimitedA3}
		\end{lemma}
		
\begin{proof}
		We observe that
		\begin{align*}
		\left(\frac{\lambda}{\mu F^r} \right)^{-(B^r+r)} &e^{-\frac{\mu F^r}{\lambda}} \sum_{k=B^r}^{B^r+r} \frac{1}{(r+B^r-k)!}\left(\frac{\lambda}{\mu F^r} \right)^k =  \sum_{k=0}^{r} \frac{1}{k!} \left(\frac{\mu F^r}{\lambda} \right)^k  e^{-\frac{\mu F^r}{\lambda}} = \Prob \left( \textrm{Pois}\left(\frac{\mu F^r}{\lambda} \right) \leq r \right) \\
		&= \Prob \left( \frac{\textrm{Pois}\left(\mu F^r/\lambda \right)-\mu F^r / \lambda}{\sqrt{\mu F^r / \lambda}} \leq \frac{r-\mu F^r / \lambda}{\sqrt{\mu F^r / \lambda}} \right) \rightarrow \Phi \left(-\gamma \sqrt{\frac{\mu}{\lambda}} \right)
		\end{align*}
		Moreover, applying Stirling's approximation twice,
		\begin{align*}
		\frac{r^{B^r} r! {F^r}^{F^r} }{F^r!} &e^{-\frac{\lambda r}{\mu}}   \left(\frac{\lambda}{\mu F^r} \right)^{B^r+r} e^{\frac{\mu F^r}{\lambda}} \sim \sqrt{\frac{r}{F^r}}  \left(\frac{\lambda r}{\mu F^r} \right)^{B^r+r} e^{-r+F^r-\frac{\lambda}{\mu}r +\frac{\mu}{\lambda} F^r} \\
		&\sim \sqrt{\frac{\mu}{\lambda}} \left(1+\frac{\gamma}{\sqrt{\lambda r/\mu}}\right)^{B^r+r} e^{\gamma \sqrt{\frac{\lambda r}{\mu}}+ \sqrt{\frac{\mu r}{\lambda}}} \\
		&\sim \sqrt{\frac{\mu}{\lambda}}e^{\gamma \sqrt{\frac{\lambda r}{\mu}}+ \sqrt{\frac{\mu r}{\lambda}}}  \textrm{exp}\left\{-\left(r+\frac{\lambda r}{\mu}+\beta \sqrt{\frac{\lambda r}{\mu}} \right)\left( \frac{\gamma}{\sqrt{\lambda r/\mu}} - \frac{1}{2} \frac{\gamma^2}{\lambda r/\mu} + O(r^{-3/2}) \right)  \right\}  \\
		&\rightarrow \sqrt{\frac{\mu}{\lambda}} \textrm{exp}\left\{\frac{\gamma^2}{2}-\beta \gamma +\frac{\mu}{\lambda} \frac{\gamma^2}{2}\right\}.
		\end{align*}
		Multiplying the two expression concludes the result.
\hfill\end{proof}
		
		\begin{lemma}
			Under scaling rules~\eqref{eq:QEDscalingSingleBSSModel}, as $r \rightarrow \infty$:
			\begin{align*}
			\Prob\left( Q^r < F^r \right) &\rightarrow  \frac{r_1}{r_1+r_2+r_3} \\
			\Prob\left( F^r \leq Q^r < B^r \right) &\rightarrow \frac{r_2}{r_1+r_2+r_3}\\
			\Prob\left( Q^r \geq B^r \right) &\rightarrow \frac{r_3}{r_1+r_2+r_3},
			\end{align*}
			where
			\begin{align*}
			r_1 &=1, \\
			r_2 &= \left\{ \begin{array}{ll}
			{\phi(\gamma)}{\Phi(\gamma)} \frac{1}{\gamma} \left(1-e^{-\gamma(\beta-\gamma)} \right) & \textrm{if } \gamma \neq 0, \\
			\sqrt{\frac{2}{\pi}} \beta & \textrm{if } \gamma = 0, \\
			\end{array}\right.\\
			r_3 &= \frac{\phi(\gamma)}{\Phi(\gamma)} e^{-\gamma(\beta-\gamma)}\sqrt{\frac{\mu}{\lambda}} \phi\left( \sqrt{\frac{\mu}{\lambda}}\gamma \right)^{-1} \Phi\left( -\sqrt{\frac{\mu}{\lambda}}\gamma\right).
			\end{align*}
			\label{lem:StStateIntervalDistributions}
		\end{lemma}
		
\begin{proof}
		It follows from Lemmas~\ref{lem:LimitedA1}-\ref{lem:LimitedA3} that
		\begin{align*}
		\pi_0^{(B^r,F^r,r)} e^{\lambda r /\mu} \rightarrow \Phi(\gamma)(r_1+r_2+r_3),
		\end{align*}
		and applying Lemmas~\ref{lem:LimitedA1}-\ref{lem:LimitedA3} again yields
		\begin{align*}
		\Prob\left( Q^r < F^r \right) &\rightarrow  \frac{\Phi(\gamma)}{\Phi(\gamma)(r_1+r_2+r_3)}, \\
		\Prob\left( F^r \leq Q^r < B^r \right) &\rightarrow \frac{\Phi(\gamma)r_2}{\Phi(\gamma)(r_1+r_2+r_3)},\\
		\Prob\left( Q^r \geq B^r \right) &\rightarrow \frac{\Phi(\gamma)r_3}{\Phi(\gamma)(r_1+r_2+r_3)}.
		\end{align*}
\hfill\end{proof}
		
		Consequently, it follows that Theorem~\ref{thm:StStateToDiffusion} holds.
		
\begin{proof}[Proof of Theorem~\ref{thm:StStateToDiffusion}]
		This result follows directly from Lemmas~\ref{lem:StStateConditionedQEDDistribution} and~\ref{lem:StStateIntervalDistributions}.
\hfill\end{proof}
		
		\subsection{Performance measures} \label{app:SingleMainPerformance}
		Using the results of Theorem~\ref{thm:DiffusionLimitProcess}, one can derive the asymptotic behavior of the performance measures, as described in Section~\ref{sec:PerformanceSingleBSS}.
		
\begin{proof}[Proof of Theorem~\ref{thm:PerformanceMeasuresSingleBSS}]
		Since the PASTA property holds asymptotically, we find that the probability an arriving car has to wait for a fully-charged battery is asymptotically equivalent to the probability that the system is in a state where more than $B^r$ batteries are charging at the same time, which gives the first asymptotic relation. Then,
		\begin{align*}
		\Prob\left( \bar{Q}(\infty) \geq \beta \right) = r_3/(r_1+r_2+r_3)
		\end{align*}
		where $r_i, i=1,2,3$ are as in Theorem~\ref{thm:DiffusionLimitProcess}, from which the result follows directly.
		
		For the expected waiting time, we consider the expected number of waiting cars at the charging station. We observe that
		\begin{align*}
		\frac{\E(Q^W)}{\sqrt{\lambda r/\mu}} &\rightarrow \int_{\beta}^\infty (x-\beta) \hat{f}(x) \, dx = \alpha_3 \Phi\left(-\sqrt{\frac{\mu}{\lambda}}\gamma \right)^{-1} \int_{0}^\infty \sqrt{\frac{\lambda}{\mu}} y \phi\left( \frac{y+\mu/\lambda \gamma}{\sqrt{\mu/\lambda}}\right) \, dy\\
		&= \alpha_3 \left(  \sqrt{\frac{\mu}{\lambda}} \phi\left(\sqrt{\frac{\mu}{\lambda}}\gamma \right) \Phi\left(-\sqrt{\frac{\mu}{\lambda}} \gamma\right)^{-1} - \frac{\mu}{\lambda}\gamma \right)  ,
		\end{align*}
		where $\alpha_3$ is as in Theorem~\ref{thm:DiffusionLimitProcess}. We remark that this expression is always positive: trivially for $\gamma \leq 0$, and also for $\gamma >0$ since for every $x>0$
		\begin{align*}
		\Phi(-x) \leq \frac{\phi(x)}{x},
		\end{align*}
		which can be shown by partial integration. Due to Little's law, we note that
		\begin{align*}
		E(W) = \frac{\E(Q^W)}{\lambda \left( r - \E(Q^W)\right)},
		\end{align*}
		which yields the result for the expected waiting time.
		
		Finally, for the utilization levels, the scaling of $\hat{Q}(\infty)$ implies that
		\begin{align*}
		\E(\textrm{\# Idle Charging Points}) &= \Theta(\sqrt{r}), \\
		\E(\textrm{\# Fully-Charged Batteries}) = \Theta(\sqrt{r}), \hspace{1cm}& \E(Q^r) = \Theta(r).
		\end{align*}
		Therefore, the utilization levels satisfy~\eqref{eq:UtilizationSingleBSS} and the result follows.
\hfill\end{proof}
		
		\section{Unlimited number of charging points for single station system} \label{app:UnlimitedCPSingle}
		\label{app:UnlimitedCPSingleBSS}
		A related setting involves the case where there are an unlimited number of charging points ($F=\infty$) and also an unlimited number of swapping servers $G=\infty$. Comparing these results to the performance of a related system where $G < \infty$ and the charging points scaled as in~\eqref{eq:QEDscalingSingleBSSModel} provides good insight into the significance of the effect of these resources. The number of spare batteries remains as before, i.e.
		\begin{align}
		B^r &= \frac{\lambda r}{\mu} + \beta \sqrt{\frac{\lambda r}{\mu}}.
		\label{eq:QEDUnlimitedSingleBSS}
		\end{align}
		
		\subsection{Fluid and diffusion limits}
		Although our main focus is performance measures in steady state, it is instructive to see the transient behavior for this large-scale system. To that end, we explore the fluid and diffusion limits.
		
		Consider the fluid-scaled process $\{\bar{Q}^r(t), t \geq 0\}$, where $\bar{Q}^r(t) = Q^r(t)/r$ for all $r \geq 1$. The following theorem holds under the QED scaling.
		
		\begin{theorem}
			If $S=1$ and $\bar{Q}^r(0) \rightarrow \bar{Q}(0)$ in distribution with $\bar{Q}(0)$ a finite constant, then $\hat{Q}^r \rightarrow \bar{Q}$ in distribution as $r \rightarrow \infty$, where $\bar{Q}$ satisfies the ODE
			\begin{align*}
			\frac{d \bar{Q}(t)}{dt} = \lambda- \mu \bar{Q}(t)  - \lambda \left(\bar{Q}(t)- \frac{\lambda}{\mu}\right)^+,
			\end{align*}
			and has the unique steady state value
			\begin{align*}
			\bar{Q}(\infty) = \frac{\lambda}{\mu}.
			\end{align*}
		\end{theorem}
		
\begin{proof}
		We follow the framework of~\cite{BrowneWhitt1995} for birth-death processes. We observe that the $Q^r$ is a birth-death process with state space $\mathcal{Q}^r=\{0,1,\ldots,B^r+r\}$, and arrival rates $\lambda_r(j)=\lambda \left(r-(j-B^r)^+\right)$ and service rate $\mu_r(j) = j \mu$ for all $j \in \mathcal{Q}^r$. The fluid-scaled process $\bar{Q}^r$ therefore has state space $\{0,1/r,\ldots,(B^r+r)/r\}$, and drift and diffusion functions
		\begin{align*}
		m_r(x) = \frac{\lambda_r( \lfloor r x \rfloor)}{r} - \frac{\mu_r( \lfloor r x \rfloor)}{r} = \lambda - \frac{\lambda(\lfloor r x \rfloor -B^r)^+}{r}- \frac{\mu \lfloor r x \rfloor}{r}
		\end{align*}
		and
		\begin{align*}
		\sigma^2_r(x) = \frac{\lambda_r( \lfloor r x \rfloor)}{r^2} + \frac{\mu_r( \lfloor r x \rfloor)}{r^2} = \lambda - \frac{\lambda(\lfloor r x \rfloor -B^r)^+}{r} + \frac{\mu \lfloor r x \rfloor}{r}.
		\end{align*}
		Taking the limit yields
		\begin{align*}
		m(x) := \lim_{r \rightarrow \infty} m_r(x) = \lambda- \mu x  - \lambda \left(x-\frac{\lambda}{\mu}\right)^+
		\end{align*}
		and
		\begin{align*}
		\sigma^2(x):= \lim_{r \rightarrow \infty} \sigma^2_r(x) = 0.
		\end{align*}
		That is, we obtain a degenerate limiting diffusion (since the limiting infinitesimal variance is zero), and hence the limiting initial point will yield a deterministic path for the limiting process $\bar{Q}$. The mean directly provides the ODE that the limiting process $\bar{Q}$ satisfies. What remains to show is the limiting value (as $t \rightarrow \infty)$ .
		
		For all $t \geq 0$ for which $\bar{Q}(t) < \frac{\lambda}{\mu}$, we have
		\begin{align*}
		\frac{d \bar{Q}(t) }{dt} = \lambda - \mu \bar{Q}(t) > 0,
		\end{align*}
		and for all $t \geq 0$ for which $\bar{Q}(t) > \frac{\lambda}{\mu}$, we obtain
		\begin{align*}
		\frac{d \bar{Q}(t) }{dt} = \frac{\lambda(\lambda + \mu)}{\mu} - (\lambda+\mu) \bar{Q}(t) < 0.
		\end{align*}
		Moreover, if $\bar{Q}(t)=\lambda/\mu$ for some $t \geq 0$, then $d \bar{Q}(t) /dt =0$. This implies that the initial limiting value $\bar{Q}(0)$ determines the unique deterministic path for the limiting process. That is, if $\bar{Q}(0)< \lambda/\mu$, then the limiting process satisfies
		\begin{align*}
		\bar{Q}(t)= \frac{\lambda}{\mu}+\left(\bar{Q}(0)-\frac{\lambda}{\mu}\right)e^{-\mu t},
		\end{align*}
		a strictly increasing function that tends to $\lambda/\mu$ as $t \rightarrow \infty$. Similarly, if $\bar{Q}(0) > \lambda/\mu$, then the limiting process satisfies
		\begin{align*}
		\bar{Q}(t)= \frac{\lambda}{\mu}+\left(\bar{Q}(0)-\frac{\lambda}{\mu}\right)e^{-(\lambda+\mu) t},
		\end{align*}
		a strictly decreasing function that tends to $\lambda/\mu$ as $t \rightarrow \infty$. Finally, if $\bar{Q}(0) = \lambda/\mu$, then $d \bar{Q}(t) /dt =0$ and hence $\bar{Q}(t)=\lambda/\mu$ for all $t\geq 0$.
\hfill\end{proof}

		We now zoom in on the fluctuations around the fluid limit by taking the diffusion limit. We consider the centered scaled process
		\begin{align}
		\hat{Q}^r(t):= \frac{Q^r(t)- \lambda r/\mu}{\sqrt{\lambda r/\mu}}, \quad t\geq 0.
		\end{align}
		Under the QED scaling rule~\eqref{eq:QEDUnlimitedSingleBSS}, we find the following diffusion limiting process.
		
		\begin{theorem}
			Suppose the system operates under the QED~scaling rule~\eqref{eq:QEDUnlimitedSingleBSS}. If~$\hat{Q}^r(0) \rightarrow \hat{Q}(0)$ in distribution as $r \rightarrow \infty$, then $\hat{Q}^r \rightarrow \hat{Q}$ in distribution as $r \rightarrow \infty$, where $\hat{Q}$ can be described as a two pieced-together Ornstein-Uhlenbeck (OU) process. More precisely, $\hat{Q}$ is a diffusion process with mean
			\begin{align*}
			m(x) = -\lambda(x-\beta)^+ -\mu x,
			\end{align*}
			and constant infinitesimal variance $2\mu$. The steady state density $\hat{Q}(\infty)$ is given by
			\begin{align}
			\hat{f}(x) = \left\{\begin{array}{lr}
			(1-\alpha) \frac{\phi(x)}{\Phi(\beta)} & \textrm{if } x < \beta,\\
			\alpha \sqrt{\frac{\lambda+\mu}{\mu}} \phi\left( \frac{x-\beta \lambda/(\lambda+\mu)}{\sqrt{\mu/(\lambda+\mu)}} \right)\Phi\left(-\beta \sqrt{\frac{\mu}{\lambda+\mu}} \right)^{-1} & \textrm{if } x \geq \beta,\\
			\end{array} \right.
			\label{eq:DensityQED}
			\end{align}
			where
			\begin{align*}
			\alpha= \left( 1+ \sqrt{\frac{\lambda+\mu}{\mu}} e^{\frac{\beta^2}{2} \frac{\lambda}{\lambda+\mu}}\frac{\Phi(\beta)}{\Phi\left(-\beta \sqrt{\frac{\mu}{\lambda+\mu}}  \right)}\right)^{-1}.
			\end{align*}
			\label{thm:DiffusionDenstityQED}
		\end{theorem}
		
\begin{proof}
		The proof is similar to the techniques used in the proof of the fluid limits. The infinitesimal mean of the centered scaled process is given by
		\begin{align*}
		m_r(x) &= \frac{\lambda_r\left(\lfloor \lambda r / \mu + x \sqrt{\lambda r / \mu } \rfloor \right) - \mu_r\left(\lfloor \lambda r / \mu + x \sqrt{\lambda r / \mu } \rfloor \right)}{\sqrt{\lambda r/ \mu}} \\
		&= \frac{\lambda r - \lambda \left( \lfloor \lambda r / \mu + x \sqrt{\lambda r / \mu } \rfloor - B^r \right)^+ - \mu \left(\lfloor \lambda r / \mu + x \sqrt{\lambda r / \mu } \rfloor \right)}{\sqrt{\lambda r/ \mu}},
		\end{align*}
		which converges to
		\begin{align*}
		m(x) = \lim_{r \rightarrow \infty} m_r(x) = -\lambda(x-\beta)^+-\mu x
		\end{align*}
		as $r \rightarrow \infty$. The infinitesimal variance tends to
		\begin{align*}
		\sigma^2_r(x) &= \frac{\lambda_r\left(\lfloor \lambda r / \mu + x \sqrt{\lambda r / \mu } \rfloor \right) + \mu_r\left(\lfloor \lambda r / \mu + x \sqrt{\lambda r / \mu } \rfloor \right)}{\lambda r/ \mu} \rightarrow \frac{2\lambda r}{\lambda r/ \mu} = 2\mu,
		\end{align*}
		as $r \rightarrow \infty$. Using Equation~(28) from~\cite{BrowneWhitt1995}, this implies that the limiting process is a piecewise-linear diffusion process with steady state density given as in~\eqref{eq:DensityQED}, where $\alpha$ is the probability that the steady state process is above $\beta$. Equations~(5) and (6) from~\cite{BrowneWhitt1995} yield
		\begin{align*}
		\alpha = \left(1+ \frac{\hat{f}(\beta^+)/\alpha}{\hat{f}(\beta^-)/(1-\alpha)} \right)^{-1} = \left( 1+ \sqrt{\frac{\lambda+\mu}{\mu}} e^{\frac{\beta^2}{2} \frac{\lambda}{\lambda+\mu}}\frac{\Phi(\beta)}{\Phi\left(-\beta \sqrt{\frac{\mu}{\lambda+\mu}}  \right)}\right)^{-1}.
		\end{align*}
\hfill\end{proof}

		\subsection{Steady state limits}
		We point out that the fluid and diffusion limits are useful for studying the system in transience. The steady state limiting distribution is then obtained by setting $t \rightarrow \infty$ with respect to the limiting diffusion process. In other words, the steady state density function~\eqref{eq:DensityQED} is obtained by first taking the limit of the centered scaled process as $n \rightarrow \infty$, after which its steady state distribution is obtained by setting $t \rightarrow \infty$. The following result shows that interchanging the limits does not change the result.
		
		\begin{theorem}
			As $r \rightarrow \infty$,
			\begin{align*}
			\hat{Q}^r(\infty) \rightarrow \hat{{Q}}(\infty)
			\end{align*}
			in distribution, where $\hat{{Q}}(\infty)$ is the steady state distribution as in Theorem~\ref{thm:DiffusionDenstityQED}.
			\label{thm:UnlimitedSteadyStateTendsToDiffusionLimitProcess}
		\end{theorem}
		
		Before moving to the proof of this result, we need to understand the limiting behavior of the steady state probability that more than $B^r$ batteries are in need of recharging. To do so, we first determine the asymptotic behavior of the normalization constant in~\eqref{eq:StDistributionNormalizationConstant}.
		
		\begin{lemma}
			Under scaling rule~\eqref{eq:QEDUnlimitedSingleBSS}, as $r \rightarrow \infty$,
			\begin{align*}
			\pi_0^{(B^r,r)} e^{\frac{\lambda}{\mu}r} \rightarrow \left( \Phi(\beta) + \sqrt{\frac{\mu}{\lambda+\mu}}e^{-\frac{\beta^2}{2}\frac{\lambda}{\lambda+\mu}} \Phi\left(-\beta \sqrt{\frac{\mu}{\lambda+\mu}} \right) \right)^{-1}.
			\end{align*}
			\label{lem:NormalizationConstantAsymptoticsSingleServer}
		\end{lemma}
		
\begin{proof}
		Note that
		\begin{align*}
		\left(\pi_0^{(B^r,r)}\right)^{-1} e^{-\frac{\lambda}{\mu}r} = \sum_{k=0}^{B^r} \frac{(\lambda r/\mu)^k}{k!} e^{-\frac{\lambda r}{\mu}} + \sum_{k=B^r+1}^{B^r+r} \frac{r! r^{B^r}}{(r+B^r)!} \binom{r+B^r}{k} \left(\frac{\lambda}{\mu}\right)^k  e^{-\frac{\lambda r}{\mu}}.
		\end{align*}
		The convergence of the first term follows directly from the CLT for a Poisson distributed random variable,
		\begin{align*}
		\sum_{k=0}^{B^r} \frac{(\lambda r/\mu)^k}{k!} e^{-\frac{\lambda r}{\mu}} = \Prob\left( \textrm{Pois} \left(\lambda r / \mu \right) \leq B^r \right) = \Prob\left( \frac{\textrm{Pois} \left(\lambda r / \mu \right)- \lambda r / \mu }{\sqrt{\lambda r / \mu }} \leq \beta \right) \rightarrow \Phi(\beta).
		\end{align*}
		We extract from the second term
		\begin{align*}
		\sum_{k=B^r+1}^{B^r+r} &\binom{r+B^r}{k} \left(\frac{\lambda}{\mu}\right)^k \left( \frac{\mu}{\lambda+ \mu}\right)^{B^r+r} = \Prob\left( \textrm{Bin}(B^r+r,\lambda/(\lambda+\mu)) \geq B^r+1 \right) \\
		&= \Prob\left( \frac{ \textrm{Bin}(B^r+r,\lambda/(\lambda+\mu)) - (B^r+r) \frac{\lambda}{\lambda+\mu}}{\sqrt{(B^r+r) \frac{\lambda \mu }{(\lambda+\mu)^2}}} \geq \frac{\mu B^r- \lambda r +\lambda +\mu	}{\sqrt{\lambda \mu } \sqrt{B^r+r}}\right) \rightarrow \Phi\left(-\beta \sqrt{\frac{\mu}{\lambda + \mu}} \right),
		\end{align*}
		where we used the CLT for a sum of independent Bernoulli distributed random variables. Finally, what is left of the second term is given by
		\begin{align*}
		\left( \frac{\mu}{\lambda +\mu}\right)&^{-(B^r+r)}  \frac{r! r^{B^r}}{(r+B^r)!} e^{-\frac{\lambda r}{\mu}} \sim \sqrt{\frac{\mu}{\lambda+\mu}} \left( 1+ \frac{\beta \frac{\sqrt{\lambda \mu}}{\lambda+\mu}}{\sqrt{r}}  \right)^{\frac{\lambda+\mu}{\mu} r + \beta \sqrt{\frac{\lambda r}{\mu}}} e^{\beta \sqrt{\frac{\lambda r}{\mu}}} \\
		&= \sqrt{\frac{\mu}{\lambda+\mu}} \textrm{Exp} \left\{ \left( \frac{\lambda+\mu}{\mu} r + \beta \sqrt{\frac{\lambda r}{\mu}}\right) \left( \beta \frac{\sqrt{\lambda \mu}}{\lambda+\mu} \frac{1}{\sqrt{r}} - \frac{\beta^2}{2} \frac{\lambda \mu}{(\lambda+\mu)^2} \right) + \beta \sqrt{\frac{\lambda r}{\mu}} \frac{1}{r} + O(r^{-3/2}) \right\}\\
		&= \sqrt{\frac{\mu}{\lambda+\mu}} \textrm{Exp} \left\{ -\frac{\beta^2}{2} \frac{\lambda}{\lambda+\mu} + O(r^{-1/2})\right\},
		\end{align*}
		where we used Stirling's approximation twice and a Taylor expansion for the logarithm term. Combining the three expressions yields the results.
\hfill\end{proof}

		Using Lemma~\ref{lem:NormalizationConstantAsymptoticsSingleServer}, we can derive the asymptotic behavior of having more than $B^r$ batteries in need of recharging in steady state.
		
		\begin{proposition}
			Under scaling rule~\eqref{eq:QEDUnlimitedSingleBSS}, the probability that an arriving car has to wait for a battery behaves as
			\begin{align*}
			\Prob(Q^r(\infty) \geq B^r) \rightarrow \left( 1+ \sqrt{\frac{\lambda+\mu}{\mu}} e^{\frac{\beta^2}{2} \frac{\lambda}{\lambda+\mu}}\frac{\Phi(\beta)}{\Phi\left(-\beta \sqrt{\frac{\mu}{\lambda+\mu}}  \right)}\right)^{-1}
			\end{align*}
			as $r \rightarrow \infty$.
			\label{prop:WaitProbSingleStationUnlimited}
		\end{proposition}
		
\begin{proof}
		Note that
		\begin{align*}
		\Prob(Q^r(\infty) \geq B^r) = \sum_{k=B^r}^{B^r+r} \pi_k^{(B^r,r)} = \pi_0^{(B^r,r)} e^{\frac{\lambda}{\mu}r} \sum_{k=B^r+1}^{B^r+r} \frac{r! r^{B^r}}{(r+B^r)!} \binom{r+B^r}{k} \left(\frac{\lambda}{\mu}\right)^k  e^{-\frac{\lambda r}{\mu}} .
		\end{align*}
		
		It follows from the proof of Lemma~\ref{lem:NormalizationConstantAsymptoticsSingleServer} that, as $r \rightarrow \infty$,
		\begin{align*}
		\sum_{k=B^r+1}^{B^r+r} \frac{r! r^{B^r}}{(r+B^r)!} \binom{r+B^r}{k} \left(\frac{\lambda}{\mu}\right)^k  e^{-\frac{\lambda r}{\mu}} \rightarrow \Phi\left(-\beta \sqrt{\frac{\mu}{\lambda + \mu}} \right) \sqrt{\frac{\mu}{\lambda+\mu}} \textrm{Exp} \left\{ -\frac{\beta^2}{2} \frac{\lambda}{\lambda+\mu} \right\}.
		\end{align*}
		Rewriting the terms concludes the result.
\hfill\end{proof}

		Next, we prove Theorem~\ref{thm:UnlimitedSteadyStateTendsToDiffusionLimitProcess}.
		
\begin{proof}[Proof of Theorem~\ref{thm:UnlimitedSteadyStateTendsToDiffusionLimitProcess}]
		By definition,
		\begin{align*}
		\hat{{Q}}(\infty) = \lim_{t \rightarrow \infty} \lim_{r \rightarrow \infty} \hat{Q}^r(t),
		\end{align*}
		and we want to show that changing the order of the limits yields the same distribution. The steady state distribution for a system with $r$ cars is given by~\eqref{eq:StDistribution}. Due to Proposition~\ref{prop:WaitProbSingleStationUnlimited}, it suffices to show that for every $x < \beta$
		\begin{align*}
		\lim_{r \rightarrow \infty} \Prob \left( \hat{Q}^r(\infty) \leq x \big| \hat{Q}^r(\infty) < \beta \right) \rightarrow \frac{\Phi(x)}{\Phi(\beta)},
		\end{align*}
		and for every $x \geq \beta$,
		\begin{align*}
		\lim_{r \rightarrow \infty} \Prob \left( \hat{Q}^r(\infty) > x \big| \hat{Q}^r(\infty) \geq \beta \right) \rightarrow \Phi\left(- \frac{x-\beta \lambda/(\lambda+\mu)}{\sqrt{\mu/(\lambda+\mu)}} \right)\Phi\left(-\beta \sqrt{\frac{\mu}{\lambda+\mu}} \right)^{-1}.
		\end{align*}
		
		\noindent
		Using the steady state distribution given in~\eqref{eq:StDistribution}, we observe that for every $\beta$,
		\begin{align*}
		\Prob \left( \hat{Q}^r(\infty) \leq x \big| \hat{Q}^r(\infty) < \beta \right) &= \frac{\Prob \left( \hat{Q}^r(\infty) \leq x \right)}{\Prob \left(  \hat{Q^r}(\infty) < \beta \right)} \\
		&= \Prob\left( \frac{\textrm{Pois}( \lambda r / \mu) -\lambda r / \mu}{\sqrt{\lambda r / \mu}} \leq x \right) \Prob\left( \frac{\textrm{Pois}( \lambda r / \mu) -\lambda r / \mu}{\sqrt{\lambda r / \mu}} < x \right)^{-1} \rightarrow \frac{\Phi(x)}{\Phi(\beta)}
		\end{align*}
		as $r \rightarrow \infty$. For every $x \geq \beta$,
		\begin{align*}
		\Prob \left( \hat{Q}^r(\infty) > x \big| \hat{Q}^r(\infty) \geq \beta \right) = \frac{\Prob \left( \hat{Q}^r(\infty) > x \right) }{\Prob \left(  \hat{Q}^r(\infty) \geq \beta \right) } = \frac{\Prob\left( \textrm{Bin}\left(B^r+r,\frac{\lambda}{\lambda+\mu}\right) > \frac{\lambda r}{\mu} + x \sqrt{\frac{\lambda r}{\mu}}\right)}{\Prob\left( \textrm{Bin}\left(B^r+r,\frac{\lambda}{\lambda+\mu}\right) > \frac{\lambda r}{\mu} + \beta \sqrt{\frac{\lambda r}{\mu}}\right)}.
		\end{align*}
		Similarly to the proof of Proposition~\ref{prop:WaitProbSingleStationUnlimited}, The CLT implies
		\begin{align*}
		\Prob\left( \textrm{Bin}\left(B^r+r,\frac{\lambda}{\lambda+\mu}\right) > \frac{\lambda r}{\mu} + x \sqrt{\frac{\lambda r}{\mu}}\right) \rightarrow \Phi\left(- \frac{x-\beta \lambda/(\lambda+\mu)}{\sqrt{\mu/(\lambda+\mu)}} \right)
		\end{align*}
		and
		\begin{align*}
		\Prob\left( \textrm{Bin}\left(B^r+r,\frac{\lambda}{\lambda+\mu}\right) > \frac{\lambda r}{\mu} + \beta \sqrt{\frac{\lambda r}{\mu}}\right) \rightarrow \Phi\left(-\beta \sqrt{\frac{\mu}{\lambda+\mu}} \right)
		\end{align*}
		as $r \rightarrow \infty$.
\hfill\end{proof}
		
		\subsection{Performance measures}
		We consider the waiting probability, the waiting time and the utilization level of the spare batteries. We note that $\alpha$ in~\eqref{eq:DensityQED} corresponds to the probability that $\{\hat{Q}(\infty) \geq \beta\}$. Although the PASTA property does not hold in this closed system, it does hold asymptotically, providing an appropriate approximation for the waiting probability.
		
		\begin{corollary}
			Under scaling rule~\eqref{eq:QEDUnlimitedSingleBSS}, as $r\rightarrow \infty$,
			\begin{align*}
			\Prob(W > 0) \rightarrow \Prob(\hat{{Q}}(\infty) \geq \beta) = \left( 1+ \sqrt{\frac{\lambda+\mu}{\mu}} e^{\frac{\beta^2}{2} \frac{\lambda}{\lambda+\mu}}\frac{\Phi(\beta)}{\Phi\left(-\beta \sqrt{\frac{\mu}{\lambda+\mu}}  \right)}\right)^{-1}.
			\end{align*}
		\end{corollary}
		
		For the waiting time, we point out that in this case the identity
		\begin{align*}
		E(W) = \frac{\E(Q^W)}{\lambda \left( r - \E(Q^W)\right)}
		\end{align*}
		still holds due to Little's law. Using this relation, we can obtain the expected waiting time.
		
		\begin{proposition}
			Under scaling rule~\eqref{eq:QEDUnlimitedSingleBSS}, as $r \rightarrow \infty$,
			\begin{align*}
			\frac{ \E(Q^W)}{\sqrt{r}} \longrightarrow \frac{\frac{\sqrt{\lambda \mu}}{\lambda+\mu} e^{ -\frac{\beta^2}{2} \frac{\lambda}{\lambda+\mu}}}{ \Phi(\beta) + \sqrt{\frac{\mu}{\lambda+\mu}}e^{-\frac{\beta^2}{2}\frac{\lambda}{\lambda+\mu}} \Phi\left(-\beta \sqrt{\frac{\mu}{\lambda+\mu}} \right) } \left( \frac{1}{\sqrt{2\pi}}  e^{-\frac{\beta^2}{2} \frac{\mu}{\lambda+\mu}}  - \beta \sqrt{\frac{\mu}{\lambda+\mu}}\Phi\left(-\beta \sqrt{\frac{\mu}{\lambda+\mu}}  \right) \right).
			\end{align*}
			Consequently, the waiting time of a car behaves as
			\begin{align*}
			E(W) \sqrt{r} &= \frac{\E(Q^W) \sqrt{r}}{\lambda r - \lambda \E(Q^W)} \\
			&\longrightarrow  \frac{\sqrt{\frac{\mu}{\lambda}}\frac{1}{\lambda+\mu} e^{ -\frac{\beta^2}{2} \frac{\lambda}{\lambda+\mu}}}{ \Phi(\beta) + \sqrt{\frac{\mu}{\lambda+\mu}}e^{-\frac{\beta^2}{2}\frac{\lambda}{\lambda+\mu}} \Phi\left(-\beta \sqrt{\frac{\mu}{\lambda+\mu}} \right) } \left( \frac{1}{\sqrt{2\pi}}  e^{-\frac{\beta^2}{2} \frac{\mu}{\lambda+\mu}}  - \beta \sqrt{\frac{\mu}{\lambda+\mu}}\Phi\left(-\beta \sqrt{\frac{\mu}{\lambda+\mu}}  \right) \right)
			\end{align*}
			as $r \rightarrow \infty$.
			\label{prop:WaitingTimeSingleStation}
		\end{proposition}
		
\begin{proof}
		We note that as $r\rightarrow \infty$,
		\begin{align*}
		\frac{\E(Q^W)}{\sqrt{\lambda r/ \mu}} \rightarrow \int_{\beta}^\infty (x-\beta) \hat{f}(x) \, dx= \int_{0}^\infty y \hat{f}(y+\beta) \, dy,
		\end{align*}
		Elementary calculus yields
		\begin{align*}
		\int_{\beta}^\infty y \phi\left( \frac{y-\beta \mu/(\lambda+\mu)}{\sqrt{\mu/(\lambda+\mu)}} \right) \, dx = \frac{\mu}{\lambda+ \mu} \phi\left( \beta \sqrt{\frac{\mu}{\lambda+\mu}}\right)-\beta \left(\frac{\mu}{\lambda+\mu}\right)^{3/2} \Phi\left( -\beta \sqrt{\frac{\mu}{\lambda+\mu}}\right).
		\end{align*}
		Then as $r \rightarrow \infty$, we obtain
		\begin{align*}
		\frac{\E(Q^W)}{\sqrt{\lambda r/ \mu}} \rightarrow \alpha \sqrt{\frac{\lambda+\mu}{\mu}} \left( \frac{\mu}{\lambda+ \mu} \phi\left( \beta \sqrt{\frac{\mu}{\lambda+\mu}}\right)-\beta \left(\frac{\mu}{\lambda+\mu}\right)^{3/2} \Phi\left( -\beta \sqrt{\frac{\mu}{\lambda+\mu}}\right) \right) \Phi\left(-\beta \sqrt{\frac{\mu}{\lambda+\mu}} \right)^{-1},
		\end{align*}
		where
		\begin{align*}
		\alpha = \frac{\Phi\left(-\beta \sqrt{\frac{\mu}{\lambda+\mu}}  \right)}{\Phi\left(-\beta \sqrt{\frac{\mu}{\lambda+\mu}}  \right) + \sqrt{\frac{\lambda+\mu}{\mu}} e^{\frac{\beta^2}{2}\frac{\lambda}{\lambda+\mu}}  \Phi(\beta)}.
		\end{align*}
		Simply rewriting the terms yields the result.
\hfill\end{proof}
		
		\noindent
		Finally, the utilization level follows directly from the scaling of the diffusion scaled process. More specifically, the scaling implies that $\E(\textrm{\# Fully-Charged Batteries})= \Theta(\sqrt{r})$ and $\E(Q^r)=\Theta(r)$.
		
		\begin{corollary}
			Under scaling rule~\eqref{eq:QEDUnlimitedSingleBSS},
			\begin{align*}
			\rho_{B^r} \rightarrow 1 \quad \text{as} \quad r\rightarrow \infty.
			\end{align*}
		\end{corollary}
		
		\section{Unlimited number of swapping servers for single station system} \label{app:UnlimitedGSingle}
		In this appendix, we consider a setting with a single station where the swapping servers pose no condition, i.e.~$G=\infty$ or equivalently $G \geq \max\{1,F^r-B^r\}$. That is, every battery of an EV that is at the station can be taken to a charging point (if available). We point out that for all $F^r \leq B^r$, this reduces to the main setting in this paper. Therefore, in this section, we consider only
		\begin{align}
		B^r &= \frac{\lambda r}{\mu} + \beta \sqrt{\frac{\lambda r}{\mu}}, \quad \beta \in \mathbb{R}, \\
		F^r &= \frac{\lambda r}{\mu} + \gamma \sqrt{\frac{\lambda r}{\mu}}, \quad \gamma > \beta.
		\label{eq:QEDscalingNetworkModelGUnlimited}
		\end{align}
		
		\subsection{Diffusion limits}
		The steady state distribution is given in Lemma~\ref{lem:SteadyStateSingleBSS}. Moreover, the fluid limit result of Proposition~\ref{prop:SingleFluidLimitProcess} and its proof remain valid in this case, since the fluctuations of order $\Theta(\sqrt{r})$ are too small to be of any impact on the fluid scale. There will however be a notable effect on the diffusion scale.
		
		\begin{theorem}
			Suppose the system operates under the QED~scaling rule~\eqref{eq:QEDscalingSingleBSSModel}. If~$\hat{Q}^r(0) \rightarrow \hat{Q}(0)$ in distribution as $r \rightarrow \infty$, then $\hat{Q}^r \rightarrow \hat{Q}$ in distribution as $r \rightarrow \infty$. The process $\{\hat{Q}(t),t \geq 0\}$ is a diffusion process with mean
			\begin{align*}
			m(x) = -\lambda(x-\beta)^+ -\mu \min\{x,\gamma\},
			\end{align*}
			and constant infinitesimal variance $2\mu$. The steady state density $\hat{Q}(\infty)$ is given by
			\begin{align}
			\hat{f}(x) = \left\{\begin{array}{ll}
			\alpha_1 \frac{\phi(x)}{\Phi(\beta)} & \textrm{if } x < \beta,\\
			\alpha_2 \sqrt{\frac{\lambda+\mu}{\mu}}\phi\left( \frac{x-\frac{\lambda}{\lambda+ \mu}\beta}{\sqrt{\mu/(\lambda+\mu)}}\right) \left( \Phi\left( \frac{\gamma-\frac{\lambda}{\lambda+ \mu}\beta}{\sqrt{\mu / (\lambda+\mu)}}\right)-  \Phi\left( \sqrt{\frac{\mu}{\lambda+\mu}\beta}\right)   \right)^{-1} & \textrm{if } \beta \leq x < \gamma, \\
			\alpha_3 \sqrt{\frac{\lambda}{\mu}} \phi\left( \frac{x-(\beta-\frac{\mu}{\lambda} \gamma)}{\sqrt{\mu / \lambda}}\right) \Phi\left(- \frac{\gamma-(\beta-\frac{\mu}{\lambda} \gamma)}{\sqrt{\mu / \lambda}}\right)^{-1} & \textrm{if } x \geq \gamma,\\
			\end{array} \right.
			\label{eq:DensityQEDLimitedGUnlimited}
			\end{align}
			where $\alpha_i=r_i/(r_1+r_2+r_3)$, $i=1,2,3$ with
			\begin{align*}
			r_1 &=1, \\
			r_2 &= \frac{\phi(\beta)}{\Phi(\beta)} \sqrt{\frac{\mu}{\lambda+\mu}}\phi\left( \sqrt{\frac{\mu}{\lambda+\mu}}\beta\right)^{-1} \left( \Phi\left( \frac{\gamma-\frac{\lambda}{\lambda+ \mu}\beta}{\sqrt{\mu / (\lambda+\mu)}}\right)-  \Phi\left( \sqrt{\frac{\mu}{\lambda+\mu}} \beta \right)   \right), \\
			r_3 &= \frac{\phi(\beta)}{\Phi(\beta)}  \phi\left( \frac{\gamma-\frac{\lambda}{\lambda+ \mu}\beta}{\sqrt{\mu/(\lambda+\mu)}} \right) \phi\left( \sqrt{\frac{\mu}{\lambda+\mu}}\beta\right)^{-1} \sqrt{\frac{\mu}{\lambda}} \phi\left( \frac{\gamma-(\beta-\frac{\mu}{\lambda} \gamma)}{\sqrt{\mu / \lambda}}\right)^{-1} \Phi\left(- \frac{\gamma-(\beta-\frac{\mu}{\lambda} \gamma)}{\sqrt{\mu / \lambda}}\right) .
			\end{align*}
			\label{thm:DiffusionDenstityQEDLimitedGUnlimited}
		\end{theorem}
		
\begin{proof}
		The infinitesimal mean of the centered scaled process is given by
		\begin{align*}
		m_r(x) &= \frac{\lambda_r\left(\lfloor \lambda r / \mu + x \sqrt{\lambda r / \mu } \rfloor \right) - \mu_r\left(\lfloor \lambda r / \mu + x \sqrt{\lambda r / \mu } \rfloor \right)}{\sqrt{\lambda r/ \mu}} \rightarrow  -\lambda(x-\beta)^+ -\mu \min\{x,\gamma\}, \\
		\end{align*}
		and the infinitesimal variance is given by
		\begin{align*}
		\sigma^2_r(x) &= \frac{\lambda_r\left(\lfloor \lambda r / \mu + x \sqrt{\lambda r / \mu } \rfloor \right) + \mu_r\left(\lfloor \lambda r / \mu + x \sqrt{\lambda r / \mu } \rfloor \right)}{\lambda r/ \mu} \sim \frac{2\lambda r}{\lambda r/ \mu} = 2\mu.
		\end{align*}
		Using Equation~(28) and~(33) from~\cite{BrowneWhitt1995}, this implies that the limiting process is a piecewise-linear diffusion process with steady state density given as in~\eqref{eq:DensityQEDLimitedGUnlimited}, where $\alpha_i$, $i=1,2,3$ is the probability that the steady state process is in that interval. These steady state probabilities follow directly from Equations~(5) and (6) in~\cite{BrowneWhitt1995}.
\hfill\end{proof}
		
		\subsection{Steady state limits}
		The main goal of this section is to show that the limiting steady state distribution is the same process as the diffusion limiting process in steady state. Again, we observe that the form of the formula is different in the intervals $[0,F]$, $[B^r,F]$, and $[F,B^r+r]$. First we derive the limiting steady state distribution conditioned to be in one of these intervals.
		
		\begin{lemma}
			Under scaling rules~\eqref{eq:QEDscalingSingleBSSModel}, as $r \rightarrow \infty$,
			\begin{align*}
			\begin{array}{ll}
			\Prob\left(Q^r < \frac{\lambda r}{\mu} + x \sqrt{\frac{\lambda r}{\mu}} \big| Q^r < B^r \right) \rightarrow \frac{\Phi(x)}{\Phi(\beta)} & x<\beta, \\
			\Prob\left(Q^r < \frac{\lambda r}{\mu} + x \sqrt{\frac{\lambda r}{\mu}} \big| B^r \leq Q^r < F^r \right) \rightarrow \Phi\left( \frac{x-\frac{\lambda}{\lambda+ \mu}\beta}{\sqrt{\mu/(\lambda+\mu)}}\right) \left( \Phi\left( \frac{\gamma-\frac{\lambda}{\lambda+ \mu}\beta}{\sqrt{\mu / (\lambda+\mu)}}\right)-  \Phi\left( \sqrt{\frac{\mu}{\lambda+\mu}} \beta \right)   \right)^{-1}   & \beta \leq x < \gamma, \\
			\Prob\left(Q^r \geq \frac{\lambda r}{\mu} + x \sqrt{\frac{\lambda r}{\mu}} \big| Q^r \geq F^r \right) \rightarrow \Phi\left( - \frac{x-(\beta-\frac{\mu}{\lambda} \gamma)}{\sqrt{\mu / \lambda}}\right) \Phi\left(- \frac{\gamma-(\beta-\frac{\mu}{\lambda} \gamma)}{\sqrt{\mu / \lambda}}\right)^{-1} & x \geq \beta.
			\end{array}
			\end{align*}
		\end{lemma}
		
\begin{proof}
		The first expression follows, similarly to the case of $F^r<B^r$, from the CLT and the properties of the Poisson distribution:
		\begin{align*}
		\Prob \left(Q^r < \frac{\lambda r}{\mu} + x \sqrt{\frac{\lambda r}{\mu}} \bigg| Q^r < B^r \right) &= \frac{\Prob\left(Q^r < \frac{\lambda r}{\mu} + x \sqrt{\frac{\lambda r}{\mu}} \right)}{\Prob\left( Q^r < B^r \right)} \\
		&= \frac{\Prob\left( \textrm{Pois}(\lambda r/ \mu) \leq  \lambda r/\mu + x\sqrt{\lambda r/\mu}-1 \right)}{\Prob\left( \textrm{Pois}(\lambda r/ \mu) \leq  \lambda r/\mu + \beta\sqrt{\lambda r/\mu}-1 \right)} \rightarrow \frac{\Phi(x)}{\Phi(\beta)}.
		\end{align*}
		The second expression follows as in the proof in Proposition~14, using the binomial distribution and the CLT:
		\begin{align*}
		\Prob\left(Q^r < \frac{\lambda r}{\mu} + x \sqrt{\frac{\lambda r}{\mu}} \big| B^r \leq Q^r < F^r \right) \rightarrow \Phi\left( \frac{x-\frac{\lambda}{\lambda+ \mu}\beta}{\sqrt{\mu/(\lambda+\mu)}}\right) \left( \Phi\left( \frac{\gamma-\frac{\lambda}{\lambda+ \mu}\beta}{\sqrt{\mu/(\lambda+\mu)}}\right)  - \Phi\left( \frac{\beta-\frac{\lambda}{\lambda+ \mu}\beta}{\sqrt{\mu/(\lambda+\mu)}}\right)  \right)^{-1}.
		\end{align*}
		The final expression again uses the CLT and the properties of the Poisson distribution:
		\begin{align*}
		\Prob&\left(Q^r \geq \frac{\lambda r}{\mu} + x \sqrt{\frac{\lambda r}{\mu}} \bigg| Q^r \geq F^r \right) = \frac{\Prob\left(Q^r \geq \frac{\lambda r}{\mu} + x \sqrt{\frac{\lambda r}{\mu}}  \right)}{\Prob\left(Q^r \geq F^r \right)} = \frac{\sum_{k=\lambda r/\mu + x\sqrt{\lambda r/\mu}}^{B^r+r} \frac{1}{(r+B^r-k)!} \left(\frac{\lambda}{\mu F^r}\right)^k }{\sum_{k=F^r}^{B^r+r}\frac{1}{(r+B^r-k)!} \left(\frac{\lambda}{\mu F^r}\right)^k} \\
		&= \frac{\sum_{k=0}^{B^r+r-\lambda r/\mu - x\sqrt{\lambda r/\mu}} \frac{1}{k!} \left(\frac{\lambda}{\mu F^r}\right)^{B^r+r-k} }{\sum_{k=0}^{B^r+r-F^r}\frac{1}{k!} \left(\frac{\lambda}{\mu F^r}\right)^{B^r+r-k} }= \frac{\sum_{k=0}^{r + (\beta- x)\sqrt{\lambda r/\mu}} \frac{1}{k!} \left(\frac{\mu F^r}{\lambda}\right)^{k} e^{-\mu F^r/\lambda}}{\sum_{k=0}^{r + (\beta- \gamma)\sqrt{\lambda r/\mu}}\frac{1}{k!} \left(\frac{\mu F^r}{\lambda}\right)^{k} e^{-\mu F^r/\lambda}}.
		\end{align*}
		Since for every $z \in \mathbb{R}$,
		\begin{align*}
		\Prob\left(\textrm{Pois}\left(\mu F^r/\lambda \right) \leq r +z \sqrt{\lambda r /\mu} \right) &\rightarrow \Phi\left( z \sqrt{\frac{\lambda}{\mu}} -\gamma\sqrt{\frac{\mu}{\lambda}}\right) = \Phi\left( \frac{z-\gamma \frac{\mu}{\lambda}}{\sqrt{\mu/\lambda}}\right)
		\end{align*}
		as $r \rightarrow \infty$, the result follows immediately.
\hfill\end{proof}

		\begin{lemma}
			Under scaling rules~\eqref{eq:QEDscalingSingleBSSModel}, as $r \rightarrow \infty$,
			\begin{align*}
			\Prob\left( Q^r < F^r \right) &\rightarrow  \frac{r_1}{r_1+r_2+r_3}, \\
			\Prob\left( F^r \leq Q^r < B^r \right) &\rightarrow \frac{r_2}{r_1+r_2+r_3},\\
			\Prob\left( Q^r \geq B^r \right) &\rightarrow \frac{r_3}{r_1+r_2+r_3},
			\end{align*}
			where
			\begin{align*}
			r_1 &=1, \\
			r_2 &= \frac{\phi(\beta)}{\Phi(\beta)} \sqrt{\frac{\mu}{\lambda+\mu}}\phi\left( \sqrt{\frac{\mu}{\lambda+\mu}}\beta\right)^{-1} \left( \Phi\left( \frac{\gamma-\frac{\lambda}{\lambda+ \mu}\beta}{\sqrt{\mu / (\lambda+\mu)}}\right)-  \Phi\left( \sqrt{\frac{\mu}{\lambda+\mu}} \beta \right)   \right), \\
			r_3 &= \frac{\phi(\beta)}{\Phi(\beta)}  \phi\left( \frac{\gamma-\frac{\lambda}{\lambda+ \mu}\beta}{\sqrt{\mu/(\lambda+\mu)}} \right) \phi\left( \sqrt{\frac{\mu}{\lambda+\mu}}\beta\right)^{-1} \sqrt{\frac{\mu}{\lambda}} \phi\left( \frac{\gamma-(\beta-\frac{\mu}{\lambda} \gamma)}{\sqrt{\mu / \lambda}}\right)^{-1} \Phi\left(- \frac{\gamma-(\beta-\frac{\mu}{\lambda} \gamma)}{\sqrt{\mu / \lambda}}\right) .
			\end{align*}
			\label{lem:IntervalProbabilitiesB}
		\end{lemma}
\begin{proof}
		Applying Lemmas~\ref{lem:LimitedB1}-\ref{lem:LimitedB3} below yields the result.
\hfill\end{proof}

		\begin{lemma}
			Under scaling rules~\eqref{eq:QEDscalingSingleBSSModel}, as $r \rightarrow \infty$,
			\begin{align*}
			\sum_{k=0}^{B^r-1} \frac{\left( \frac{\lambda r}{\mu} \right)^k}{k!} e^{-\lambda r / \mu} \rightarrow \Phi(\beta).
			\end{align*}
			\label{lem:LimitedB1}
		\end{lemma}
		
\begin{proof}
		This follows from the CLT and the properties of the Poisson distribution.
\hfill\end{proof}
		
		\begin{lemma}
			Under scaling rules~\eqref{eq:QEDscalingSingleBSSModel}, as $r \rightarrow \infty$,
			\begin{align*}
			\sum_{k=B^r}^{F^r-1} \frac{r^{B^r} r!}{(r+B^r-k)!} \frac{\left(\lambda/\mu \right)^k }{k!} &\rightarrow \sqrt{\frac{\mu}{\lambda+\mu}} e^{-\frac{\beta^2}{2} \frac{\lambda}{\lambda+\mu}} \left( \Phi\left( \frac{\gamma-\frac{\lambda}{\lambda+ \mu}\beta}{\sqrt{\mu / (\lambda+\mu)}}\right)-  \Phi\left( \sqrt{\frac{\mu}{\lambda+\mu}} \beta \right)   \right)\\
			&= \phi(\beta)\sqrt{\frac{\mu}{\lambda+\mu}}\phi\left( \sqrt{\frac{\mu}{\lambda+\mu}}\beta\right)^{-1} \left( \Phi\left( \frac{\gamma-\frac{\lambda}{\lambda+ \mu}\beta}{\sqrt{\mu / (\lambda+\mu)}}\right)-  \Phi\left( \sqrt{\frac{\mu}{\lambda+\mu}} \beta \right)   \right).
			\end{align*}
			\label{lem:LimitedB2}
		\end{lemma}
		
\begin{proof}
		The proof is similar to that of Lemma~\ref{lem:NormalizationConstantAsymptoticsSingleServer}, using the properties of the binomial distribution and the CLT, and Stirling's approximation.
\hfill\end{proof}
		
		\begin{lemma}
			Under scaling rules~\eqref{eq:QEDscalingSingleBSSModel}, as $r \rightarrow \infty$,
			\begin{align*}
			\frac{r^{B^r} r! {F^r}^{F^r} }{F^r!} e^{-\frac{\lambda r}{\mu}} \sum_{k=F^r}^{B^r+r} \frac{1}{(r+B^r-k)!}\left(\frac{\lambda}{\mu F^r} \right)^k &\rightarrow \sqrt{\frac{\mu}{\lambda}} e^{\frac{\gamma^2}{2}-\beta \gamma +\frac{\mu}{\lambda} \frac{\gamma^2}{2}} \Phi\left( - \sqrt{\frac{\mu}{\lambda} \gamma} \right) \\
			&=\phi(\gamma) e^{-\gamma(\beta-\gamma)}\sqrt{\frac{\mu}{\lambda}} \phi\left( \sqrt{\frac{\mu}{\lambda}}\gamma \right)^{-1} \Phi\left( -\sqrt{\frac{\mu}{\lambda}}\gamma\right).
			\end{align*}
			\label{lem:LimitedB3}
		\end{lemma}
		
\begin{proof}
		We observe that
		\begin{align*}
		&\left(\frac{\lambda}{\mu F^r} \right)^{-(B^r+r)} e^{-\frac{\mu F^r}{\lambda}} \sum_{k=F^r}^{B^r+r} \frac{1}{(r+B^r-k)!}\left(\frac{\lambda}{\mu F^r} \right)^k =\Prob \left( \textrm{Pois}\left(\frac{\mu F^r}{\lambda} \right) \leq B^r+r-F^r \right) \\
		&= \Prob \left( \frac{\textrm{Pois}\left(\mu F^r/\lambda \right)-\mu F^r / \lambda}{\sqrt{\mu F^r / \lambda}} \leq \frac{r+B^r-F^r-\mu F^r / \lambda}{\sqrt{\mu F^r / \lambda}} \right) \rightarrow \Phi \left((\beta-\gamma)\sqrt{\frac{\lambda}{\mu}}-\gamma \sqrt{\frac{\mu}{\lambda}} \right).
		\end{align*}
		Moreover, applying Stirling's approximation twice,
		\begin{align*}
		\frac{r^{B^r} r! {F^r}^{F^r} }{F^r!} &e^{-\frac{\lambda r}{\mu}}   \left(\frac{\lambda}{\mu F^r} \right)^{B^r+r} e^{\frac{\mu F^r}{\lambda}} \rightarrow \sqrt{\frac{\mu}{\lambda}} \textrm{exp}\left\{\frac{\gamma^2}{2}-\beta \gamma +\frac{\mu}{\lambda} \frac{\gamma^2}{2}\right\}\\
		&=\phi(\beta)  \phi\left( \frac{\gamma-\frac{\lambda}{\lambda+ \mu}\beta}{\sqrt{\mu/(\lambda+\mu)}} \right) \phi\left( \sqrt{\frac{\mu}{\lambda+\mu}}\beta\right)^{-1} \sqrt{\frac{\mu}{\lambda}} \phi\left( \frac{\gamma-(\beta-\frac{\mu}{\lambda} \gamma)}{\sqrt{\mu / \lambda}}\right)^{-1} \Phi\left(- \frac{\gamma-(\beta-\frac{\mu}{\lambda} \gamma)}{\sqrt{\mu / \lambda}}\right).
		\end{align*}
		Multiplying the two expressions concludes the result.
\hfill\end{proof}

		\subsection{Performance measures}
		\begin{theorem}
			Under scaling rules~\eqref{eq:QEDscalingSingleBSSModel},
			\begin{align*}
			\Prob&(W>0) \rightarrow \Prob\left( \hat{Q}(\infty) \geq \beta \right) \\
			&= 1-\left(1+\sqrt{\frac{\mu}{\lambda+\mu}} \frac{\phi(\beta)}{\phi(\sqrt{\mu/(\lambda+\mu)}\beta)}  \left(  \Phi\left( \frac{\gamma-\frac{\lambda}{\lambda+ \mu}\beta}{\sqrt{\mu / (\lambda+\mu)}}\right)- \Phi\left( \sqrt{\frac{\mu}{\lambda+\mu}} \beta \right)  \right)\Phi(\beta)^{-1} \right. \\
			&\left. \hspace{1.5cm}+ \sqrt{\frac{\mu}{\lambda}} \frac{\phi(\beta)}{ \phi\left( \sqrt{\frac{\mu}{\lambda+\mu}}\beta\right)}  \phi\left( \frac{\frac{\lambda+\mu}{\lambda}\gamma-\beta}{\sqrt{\mu / \lambda}}\right) \phi\left( \frac{\gamma-\frac{\lambda+\mu}{\lambda}\beta}{\sqrt{\mu/(\lambda+\mu)}} \right)^{-1} \Phi\left(-  \frac{\gamma-\lambda/(\lambda+ \mu)\beta}{\sqrt{\mu/(\lambda+\mu)}}\right) \Phi(\beta)^{-1} \right)^{-1}.
			\end{align*}
			The expected waiting time behaves as
			\begin{align*}
			\frac{\E(W)}{\sqrt{r}} \rightarrow \frac{\alpha_2}{\sqrt{\lambda \mu}} \left( \sqrt{\frac{\mu}{\lambda+\mu}} \frac{ \phi\left( \frac{\gamma-\frac{\lambda}{\lambda+ \mu}\beta}{\sqrt{\mu / (\lambda+\mu)}}\right)-  \phi\left( \sqrt{\frac{\mu}{\lambda+\mu}}\beta\right)  }{ \Phi\left( \frac{\gamma-\frac{\lambda}{\lambda+ \mu}\beta}{\sqrt{\mu / (\lambda+\mu)}}\right)-  \Phi\left( \sqrt{\frac{\mu}{\lambda+\mu}}\beta\right)   } -\frac{\mu}{\lambda+\mu}\beta \right)+\frac{\alpha_3}{\sqrt{\lambda \mu}} \left(\sqrt{\frac{\mu}{\lambda}} \frac{\phi\left( \frac{\gamma-\beta+\frac{\mu}{\lambda} \gamma}{\sqrt{\mu / \lambda}}\right) }{\Phi\left(- \frac{\gamma-\beta+\frac{\mu}{\lambda} \gamma}{\sqrt{\mu / \lambda}}\right)} - \frac{\mu}{\lambda} \gamma \right),
			\end{align*}
			with $\alpha_i$ as in Theorem~\ref{thm:DiffusionDenstityQEDLimitedGUnlimited}. Finally, the utilization levels behave as
			\begin{align*}
			\rho_{F^r} \rightarrow 1 , \hspace{1cm} \rho_{B^r} \rightarrow 1 \quad \text{as} \quad r \rightarrow \infty..
			\end{align*}
		\end{theorem}
		
\begin{proof}
		Since the PASTA property holds asymptotically, we find that the probability an arriving car has to wait for a fully-charged battery is asymptotically equal to the probability that the system is in a state where more than $B^r$ batteries are charging at the same time, which gives the first asymptotic relation. Thus,
		\begin{align*}
		\Prob\left( \hat{Q}(\infty) \geq \beta \right) \rightarrow \left\{\begin{array}{ll}
		r_3/(r_1+r_2+r_3) & \textrm{if } \gamma<\beta,\\
		1- r_1/(r_1+r_2+r_3) & \textrm{if } \beta \leq \gamma,
		\end{array}\right.
		\end{align*}
		where $r_i, i=1,2,3$ are as in Theorem~\ref{thm:DiffusionDenstityQEDLimitedGUnlimited}, from which the result follows directly.
		
		For the expected waiting time, note that
		\begin{align*}
		\frac{\E(Q^W)}{\sqrt{\lambda r/\mu}} &\rightarrow \int_{0}^\infty y \hat{f}(y+\beta) \, dy = \alpha_3  \Phi\left(- \frac{\gamma-(\beta-\frac{\mu}{\lambda} \gamma)}{\sqrt{\mu / \lambda}}\right)^{-1} \int_{\gamma-\beta}^\infty \sqrt{\frac{\lambda}{\mu}} \phi\left( \frac{y+\frac{\mu}{\lambda} \gamma}{\sqrt{\mu / \lambda}}\right) \, dy\\
		&\hspace{2cm} + \alpha_2 \left( \Phi\left( \frac{\gamma-\frac{\lambda}{\lambda+ \mu}\beta}{\sqrt{\mu / (\lambda+\mu)}}\right)-  \Phi\left( \sqrt{\frac{\mu}{\lambda+\mu}}\beta\right)   \right)^{-1}  \int_{0}^{\beta-\gamma} \sqrt{\frac{\lambda+\mu}{\mu}}\phi\left( \frac{y+\frac{\mu}{\lambda+ \mu}\beta}{\sqrt{\mu/(\lambda+\mu)}}\right) \, dy \\
		&= \alpha_2 \left( \sqrt{\frac{\mu}{\lambda+\mu}} \frac{ \phi\left( \frac{\gamma-\frac{\lambda}{\lambda+ \mu}\beta}{\sqrt{\mu / (\lambda+\mu)}}\right)-  \phi\left( \sqrt{\frac{\mu}{\lambda+\mu}}\beta\right)  }{ \Phi\left( \frac{\gamma-\frac{\lambda}{\lambda+ \mu}\beta}{\sqrt{\mu / (\lambda+\mu)}}\right)-  \Phi\left( \sqrt{\frac{\mu}{\lambda+\mu}}\beta\right)   } -\frac{\mu}{\lambda+\mu}\beta \right)+\alpha_3 \left(\sqrt{\frac{\mu}{\lambda}} \frac{\phi\left( \frac{\gamma-\beta+\frac{\mu}{\lambda} \gamma}{\sqrt{\mu / \lambda}}\right) }{\Phi\left(- \frac{\gamma-\beta+\frac{\mu}{\lambda} \gamma}{\sqrt{\mu / \lambda}}\right)} - \frac{\mu}{\lambda} \gamma \right).
		\end{align*}
		Due to Little's law, we have the identity
		\begin{align*}
		E(W) = \frac{\E(Q^W)}{\lambda \left( r - \E(Q^W)\right)},
		\end{align*}
		which yield the results for the expected waiting time.
		
		Finally, for the utilization levels, the scaling of $\hat{Q}(\infty)$ implies that
		\begin{align*}
		\E(\textrm{\# Idle Charging Points}) &= \Theta(\sqrt{r}), \\
		\E(\textrm{\# Fully-Charged Batteries}) = \Theta(\sqrt{r}), \hspace{1cm}& \E(Q^r) = \Theta(r).
		\end{align*}
		Therefore, the utilization levels satisfy~\eqref{eq:UtilizationSingleBSS} and the result follows.
\hfill\end{proof}
		
		\section{Fluid limit proofs}\label{app:FluidLimit}
		The proof of Theorem~\ref{thm:FluitLimitResult} is similar to the proof of~\cite{DaiTezcan2011}, but adapted appropriately to our system.
		
\begin{proof}[Proof of Theorem~\ref{thm:FluitLimitResult}]
		First, we show that the fluid limits exist, where all components are Lipschitz continuous. For this purpose, we show that for all $\omega \in \mathcal{A}$ the sequence $\{\bar{\mathbb{X}}(\cdot,\omega)\}$ has a convergent subsequence, where every component in the limiting process is Lipschitz continuous.
		
		Fix $\omega \in \mathcal{A}$. We observe that for every $0 \leq t_1 <t_2$,
		\begin{align*}
		\left| \frac{T_j^r(t_2,\omega)}{r} - \frac{T_j^r(t_1,\omega)}{r}\right| \leq \frac{F_j^r (t_2-t_1)}{r} < \left(\frac{\lambda}{\mu}+ |\gamma| \sqrt{\frac{\lambda}{\mu}} \right)(t_2-t_1).
		\end{align*}
		Therefore, there exists a subsequence $\{r_l\}$ such that $\bar{T}_j^{r_l}(\cdot,\omega)$ converges u.o.c. to some $\bar{T}_j(\cdot, \omega)$ as $l \rightarrow \infty$ for every $j=1,\ldots,S$, which is Lipschitz continuous.
		
		Using Lemma~11 in~\cite{AtaKumar2005}, Equation~\eqref{eq:IdentityServiceCompletions} and the fact that $\omega \in \mathcal{A}$, it follows that $\bar{D}^r_j(\cdot, \omega)$ also converges u.o.c.~to $D_j(\cdot, \omega)$ for every $j=1,\ldots,S$. In fact, it follows that $D_j(\cdot, \omega) = \mu \bar{T}_j(\cdot, \omega)$, and is therefore also Lipschitz continuous.
		
		Next, we consider the arrival processes. First, we observe that $L^r(t) \leq r $ for every $t \geq 0$. Therefore,
		\begin{align*}
		\left| \frac{Y^r(t_2,\omega)}{r} - \frac{Y^r(t_1,\omega)}{r}\right| \leq t_2-t_1,
		\end{align*}
		for all $0 \leq t_1 <t_2$, and hence there exists a subsequence $\{r_l\}$ such that $\bar{Y}^{r_l}(\cdot,\omega)$ converges u.o.c. to some $\bar{Y}(\cdot, \omega)$ as $l \rightarrow \infty$ for every $j=1,\ldots,S$, which is again Lipschitz continuous.
		
		Moreover, it follows from~\eqref{eq:IdentityArrivals2} for all $0 \leq t_1 < t_2$,
		\begin{align*}
		\bar{A}_{ij}^r(t_2,\omega)-\bar{A}_{ij}^r(t_1,\omega) \leq \frac{\Lambda_{ij}\left(  r t_2 \right)-\Lambda_{ij}\left( r t_1 \right)}{r}.
		\end{align*}
		As $\omega \in \mathcal{A}$ and the FSLLN applies, it follows from Theorem~12.3 in~\cite{Billingsley1999} that there is some subsequence $\{r_l\}$ such that $\bar{A}_{ij}^r(\cdot,\omega)$ converges u.o.c. as $l \rightarrow \infty$ to some process $\bar{A}_{ij}(\cdot,\omega)$. In particular, it holds for all $0 \leq t_1 < t_2$ that
		\begin{align*}
		\bar{A}_{ij}(t_2,\omega)- \bar{A}_{ij}(t_1,\omega) \leq p_{ij} \lambda (t_2-t_1),
		\end{align*}
		and $\bar{A}_{ij}$ is hence Lipschitz continuous. Similarly, we can show the same convergence result for the processes $\bar{A}_{ij,i}^r(\cdot,\omega)$ to $\bar{A}_{ij,j}(\cdot,\omega)$.
		
		By~\eqref{eq:IdentityQueueLength}, it follows also that $\{Q_j^{r_l}(\cdot, \omega)\}$ is precompact, which in turn implies that $\{Z_j^{r_l}(\cdot, \omega)\}$ is precompact due to~\eqref{eq:IdentityBusyServers}. Moreover, $\{L^{r_l}(\cdot, \omega)\}$ is precompact by~\eqref{eq:IdentityDrivingCars}. In conclusion, the fluid limit exists with each component being Lipschitz continuous.
		
		Fluid equations~\eqref{eq:IdentityArrivals1Fluid}-\eqref{eq:IdentityDrivingCarsFluid} follow from the FSLLN results and applying Lemma~11 of~\cite{AtaKumar2005}. Equation~\eqref{eq:IdentityArrivalDerivativeFluid} requires additional arguments. Suppose $\bar{\mathbb{X}}$ to be a fluid limit with corresponding $\omega \in \mathcal{A}$ and subsequence $\{r_l\}_{l \in \mathbb{N}}$. If for some $t >0$ we have that $\bar{Q}_j(t)/p_j > \bar{Q}_i(t)/p_i$, then it follows by the continuity of the fluid limit that there exists a $\delta>0$ such that
		\begin{align*}
		\frac{\bar{Q}_j(s)}{p_j} > \frac{\bar{Q}_i(s)}{p_i}
		\end{align*}
		for all $s \in [t-\delta,t+\delta]$. By definition of the fluid limit, it holds for large enough $r_l$ that
		\begin{align*}
		\frac{\bar{Q}^{r_l}_j(s, \omega)}{p_j} > \frac{\bar{Q}^{r_l}_i(s, \omega)}{p_i}
		\end{align*}
		for all $s \in [t-\delta,t+\delta]$. In this case, the routing policy states that all arrivals of type $\{i,j\}$ move to station station $i$. Therefore, $A^{r_l}_{ij,j}$ remains constant on $[t-\delta,t+\delta]$, and hence $\bar{A}'_{ij,j}(t)=0$. Moreover, station~$i$ receives all arrivals and by the FSLLN and~\eqref{eq:IdentityArrivals2Fluid},
		\begin{align*}
		\bar{A}_{ij,i}(t_2, \omega)-\bar{A}_{ij,i}(t_1, \omega) = p_{ij}\lambda \bar{L}(t, \omega)(t_2-t_1).
		\end{align*}
		for all $t_1< t_2$ with $t_1,t_2 \in [t-\delta,t+\delta]$. It follows that $\bar{A}'_{ij,i}(t,\omega)= p_{ij} \lambda \bar{L}(t, \omega)$ by~\eqref{eq:IdentityArrivals1}.
		
		Finally, we show that there is a unique invariant state given by $q = (p_1 \lambda/\mu,\ldots,p_S \lambda/\mu)$. Introduce the function
		\begin{align*}
		h(t)= \max_{1\leq j \leq S} \left\{\frac{\bar{Q}_j(t)}{p_j} \right\} - \min_{1\leq j \leq S} \left\{\frac{\bar{Q}_j(t)}{p_j} \right\},
		\end{align*}
		and write
		\begin{align*}
		\bar{S}_{\max}(t) = \argmax_{1\leq j\leq S} \left\{\frac{\bar{Q}_j(t)}{p_j} \right\} , \hspace{0.5cm} \bar{S}_{\min}(t) = \argmin_{1\leq j\leq S} \left\{\frac{\bar{Q}_j(t)}{p_j} \right\}.
		\end{align*}
		Trivially, $h(t) \geq 0$. Since $\bar{Q}(\cdot)$ is Lipschitz continuous, so is $h(\cdot)$, and hence it is differentiable almost everywhere. To show that $h(0)=0$ implies $h(t) =0$ for all $t\geq0$, it therefore suffices to show that if $h(t)>0$ then $h'(t)<0$ for every regular point $t$ of $\bar{\mathbb{X}}$. By~\eqref{eq:IdentityQueueLengthFluid}, we observe that for every $j=1,\ldots,S$ and regular point $t\geq 0$,
		\begin{align*}
		\bar{Q}'_j(t)= \sum_{i : \{i,j\} \in E } \bar{A}'_{ij,j}(t)-\bar{D}'_j(t).
		\end{align*}
		Due to~\eqref{eq:IdentityServiceCompletionsFluid}-\eqref{eq:IdentityBusyServersFluid}, $\bar{D}'_j(t) = \min\{\mu \bar{Q}_j(t), p_j \lambda\}$. In particular, we observe that $\bar{D}'_i(t)/p_i = \bar{D}'_j(t)/p_j$ for all $i,j \in \bar{S}_{\max}(t)$, as well as for all $i,j \in \bar{S}_{\min}(t)$. Due to Lemma~2.8.6 from~\cite{Dai1999}, as $t$ is a regular point, it follows that
		\begin{align*}
		\frac{\sum_{i : \{i,j\} \in E } \bar{A}'_{ij,j}(t)}{p_j} = \frac{\sum_{k : \{k,l\} \in E } \bar{A}'_{kl,l}(t)}{p_l}
		\end{align*}
		for every $j,l \in \bar{S}_{\max}(t)$, as well as for every $j,l \in \bar{S}_{\min}(t)$. Due to~\eqref{eq:IdentityArrivals1Fluid},~\eqref{eq:IdentityArrivals2Fluid} and~\eqref{eq:IdentityArrivalDerivativeFluid}, we conclude that for every $j \in S_{\max}(t)$,
		\begin{align*}
		\frac{\sum_{i : \{i,j\} \in E } \bar{A}'_{ij,j}(t)}{p_j} = \frac{1}{p_j} \left( \frac{p_j}{\sum_{i \in \bar{S}_{\max}(t)} p_i} \sum_{\substack{\{i,j\} \in E,\\ i,j \in S_{\max}(t)}} p_{ij} \lambda \bar{L}(t) \right) < \lambda \bar{L}(t).
		\end{align*}
		On the other hand, for every $j \in S_{\min}(t)$,
		\begin{align*}
		\frac{\sum_{i : \{i,j\} \in E } \bar{A}'_{ij,j}(t)}{p_j} = \frac{1}{p_j} \left( \frac{p_j}{\sum_{i \in \bar{S}_{\min}(t)} p_i} \sum_{\substack{\{i,j\} \in E,\\ i\in \bar{S}_{\min}(t) \cup j\in \bar{S}_{\min}(t)}} p_{ij} \lambda \bar{L}(t) \right) > \lambda \bar{L}(t).
		\end{align*}
		Observing that $\bar{D}'_j(t)/p_j > \bar{D}'_i(t)/p_i$ for every $j \in S_{\max}(t)$ and $i\in S_{\min}(t)$, we therefore conclude that if $h(t) >0$ with $t$ a regular point, then
		\begin{align*}
		h'(t) < \lambda \bar{L}(t) - \lambda \bar{L}(t) = 0.
		\end{align*}
		In other words, for every invariant state of the fluid limit, it must hold that $\bar{Q}_i/p_i(t)= \bar{Q}_j/p_j(t)$ for every $i \neq j$. We observe, in view of~\eqref{eq:IdentityQueueLengthFluid}, that
		\begin{align*}
		\frac{\bar{Q}'_j(t)}{p_j} = \frac{\sum_{i : \{i,j\} \in E } \bar{A}'_{ij,j}(t)-\bar{D}'_j(t)}{p_j},
		\end{align*}
		where for every $1 \leq j \leq S$,
		\begin{align*}
		\frac{\bar{D}'_j(t)}{p_j} = \mu \min\{\bar{Q}_j(t) /p_j , \lambda/\mu\},
		\end{align*}
		and hence in view of~\eqref{eq:IdentityArrivals1Fluid} and~\eqref{eq:IdentityArrivals2Fluid},
		\begin{align*}
		\sum_{i : \{i,j\} \in E } \bar{A}'_{ij,j}(t) = p_j \lambda \bar{L}(t).
		\end{align*}
		That is, $\bar{Q}'_j(t) =0$ if and only if
		\begin{align*}
		\mu \min\{\bar{Q}_j(t) /p_j , \lambda/\mu\} = p_j \lambda \bar{L}(t).
		\end{align*}
		In view of~\eqref{eq:IdentityDrivingCarsFluid}, this occurs if and only if $\bar{Q}_j(t) = p_j \lambda /\mu$ for every $j=1,\ldots,S$.
\hfill\end{proof}
		
		\section{State space collapse proofs}\label{app:SSCproofs}
		To prove Theorem~\ref{thm:StrongSSC}, we use a framework similar to that of~\cite{Bramson1998}, and~\cite{DaiTezcan2011}. The construction consists of several steps, which we lay out next.
		
		\begin{enumerate}
			\item Divide interval $[0,T]$ into $T \sqrt{r}$ intervals of length $1/\sqrt{r}$, indexed by $m$. In each interval, consider the \textit{hydrodynamically}-scaled process of $\mathbb{X}$. For each of these interval, we
			\begin{enumerate}
				\item show the scaled process is "almost" Lipschitz continuous;
				\item show convergence to some \textit{hydrodynamic} limiting process for a sufficiently large part of the state space;
				\item derive the hydrodynamic limit equations.
			\end{enumerate}
			\item Relate the hydrodynamic scaling to the diffusion scaling, using a SSC function to deal with complications regarding the range of the time variable. Transferring the results appropriately, we show \textit{multiplicative} SSC with respect to the SSC function.
			\item Using a compact containment condition, we show that this implies \textit{strong} SSC.
		\end{enumerate}
		
		\subsection{Hydrodynamic scaling and its limiting process}
		In order to introduce the hydrodynamic scaling, we use a diffusion scaling for the values of the process but we slow the process down in time in order to analyze what occurs initially (what would happen instantaneously on diffusive scale). That is, we divide the interval $[0,T]$ in $T \sqrt{r}$ intervals of length $1/\sqrt{r}$, indexed by $m$. We write $p=(p_1,\ldots,p_S)$, and
		\begin{align}
		x_{r,m} = \max\left\{ \left\lvert Q^r\left( \frac{m}{\sqrt{r}}\right)  - p \frac{\lambda r}{\mu} \right\rvert^2 , \left\lvert L^r\left( \frac{m}{\sqrt{r}}\right)  - r \right\rvert^2, r \right\}.
		\label{eq:xStepVariableDefinition}
		\end{align}
		For the processes in $\mathbb{X}$, we introduce the following hydronimically-scaled variants. For $Q^r$, $Z^r$ and $L^r$, let
		\begin{align}
		Q^{r,m}(t) &= \frac{1}{\sqrt{x_{r,m}}} \left(Q^r\left(\frac{m}{\sqrt{r}} + \frac{\sqrt{x_{r,m}} t}{r} \right) - p \frac{\lambda r}{\mu} \right),\\
		Z^{r,m}(t) &= \frac{1}{\sqrt{x_{r,m}}} \left(Z^r\left(\frac{m}{\sqrt{r}} + \frac{\sqrt{x_{r,m}} t}{r} \right) - p \frac{\lambda r}{\mu} \right),\\
		L^{r,m}(t) &= \frac{1}{\sqrt{x_{r,m}}} \left(L^r\left(\frac{m}{\sqrt{r}} + \frac{\sqrt{x_{r,m}} t}{r} \right) - r \right),
		\end{align}
		the deviations of these processes with respect to their fluid limits. For the processes $A^r$, $A_d$, $Y^r$, $T^r$ and $D_r$, we introduce
		\begin{align}
		A^{r,m}(t) &= \frac{1}{\sqrt{x_{r,m}}} \left(A^r\left(\frac{m}{\sqrt{r}} + \frac{\sqrt{x_{r,m}} t}{r} \right) - A^r\left(\frac{m}{\sqrt{r}}  \right) \right), \\
		A^{r,m}_d(t) &= \frac{1}{\sqrt{x_{r,m}}} \left(A^r_d \left(\frac{m}{\sqrt{r}} + \frac{\sqrt{x_{r,m}} t}{r} \right) - A^r_d \left(\frac{m}{\sqrt{r}}  \right) \right),\\
		Y^{r,m}(t) &= \frac{1}{\sqrt{x_{r,m}}} \left(Y^r\left(\frac{m}{\sqrt{r}} + \frac{\sqrt{x_{r,m}} t}{r} \right) - Y^r\left(\frac{m}{\sqrt{r}}  \right) \right), \\
		T^{r,m}(t) &= \frac{1}{\sqrt{x_{r,m}}} \left(T^r\left(\frac{m}{\sqrt{r}} + \frac{\sqrt{x_{r,m}} t}{r} \right) - T^r\left(\frac{m}{\sqrt{r}}  \right) \right), \\
		D^{r,m}(t) &= \frac{1}{\sqrt{x_{r,m}}} \left(D^r\left(\frac{m}{\sqrt{r}} + \frac{\sqrt{x_{r,m}} t}{r} \right) - D^r\left(\frac{m}{\sqrt{r}}  \right) \right).
		\end{align}
		In other words, we track the increase of these processes during the interval  $[m/\sqrt{r},m/\sqrt{r}+\sqrt{x_{r,m}}t/r]$. By definition of $x_{r,m}$, we note that
		\begin{align*}
		\lvert \mathbb{X}^{r,m}(0) \rvert \leq 1,
		\end{align*}
		which will be a required compactness property when we prove convergence to a hydrodynamic limit. Moreover, due to our fluid limit results, we can show that $\sqrt{x_{r,m}}/r$ is very small for all $\omega \in \mathcal{A}$.
		
		\begin{lemma}
			Suppose $\hat{Q}^r(0) \rightarrow \hat{Q}(0)$ for some random vector $\hat{Q}(0)$, and let $M>0$ fixed. For every $\epsilon >0$ and $\omega \in \mathcal{A}$,
			\begin{align*}
			\max_{m < \sqrt{r} T} \left\{ \frac{\sqrt{x_{r,m}}}{r} \lVert Q^{r,m}(t) \rVert_M , \frac{\sqrt{x_{r,m}}}{r} \lVert L^{r,m}(t) \rVert_M \right\} \leq \epsilon
			\end{align*}
			for $r$ large enough.
			\label{lem:BoundedXrm}
		\end{lemma}
		
\begin{proof}
		Due to our fluid limit result in Theorem~\ref{thm:FluitLimitResult} and the definition of $x_{r,m}$ in~\eqref{eq:xStepVariableDefinition}, we observe that
		\begin{align*}
		\frac{\sqrt{x_{r,m}}}{r} \leq \max\{\lVert Q^r(t) - p \lambda r / \mu \rVert_T / r , \lVert L^r(t) -r \rVert_T / r , 1/\sqrt{r}  \} \leq \epsilon
		\end{align*}
		for $r$ large enough. Moreover, for $r$ large enough,
		\begin{align*}
		\max\left\{\left\lVert Q^r(t) - p \lambda r / \mu \right\rVert_{T+M\epsilon} / r , \lVert L^r(t) - r \rVert_{T+M\epsilon} / r   \right\} \leq \epsilon.
		\end{align*}
		We conclude that for every $m \leq \sqrt{r} T$,
		\begin{align*}
		\frac{\sqrt{x_{r,m}}}{r} \lVert Q^{r,m}(t) \rVert_M = \frac{\lVert Q^{r}(m/\sqrt{r}+\sqrt{x_{r,m}} / r t) - p \lambda r / \mu \rVert_M}{r} \leq \epsilon,
		\end{align*}
		and
		\begin{align*}
		\frac{\sqrt{x_{r,m}}}{r} \lVert L^{r,m}(t) - r \rVert_M \leq \epsilon.
		\end{align*}
\hfill\end{proof}
		
		For the hydrodynamically scaled process, the system identities translate to
		\begin{align}
		A_{ij}^{r,m}(t) &= A_{ij,i}^{r,m}(t) + A_{ij,j}^{r,m}(t) \hspace{0.5cm} \forall \{i,j\} \in E, \label{eq:HydroArrivals1}\\
		A_{ij}^{r,m}(t) &= \frac{1}{\sqrt{x_{r,m}}}\left( \Lambda_{ij}\left( Y^{r}(m/\sqrt{r}) +  \sqrt{x_{r,m}} Y^{r,m}(t) \right) -  \Lambda_{ij}\left( Y^{r}(m/\sqrt{r}) \right)  \right), \hspace{0.5cm} \forall \{i,j\} \in E, \label{eq:HydroArrivals2}\\
		Q_j^{r,m}(t) &= Q_j^{r,m}(0) + \sum_{i : \{i,j\} \in E } A_{ij,j}^{r,m}(t)-D_j^{r,m}(t), \hspace{0.5cm} \forall j=1,\ldots,S, \label{eq:HydroQueueLength}\\
		D_j^{r,m}(t) &= \frac{1}{\sqrt{x_{r,m}}} \left( S_j\left( T_j^{r}(m/\sqrt{r}) +  \sqrt{x_{r,m}} T_j^{r,m}(t) \right) -  S_j\left( T_j^{r}(m/\sqrt{r}) \right)  \right), \hspace{0.5cm} \forall j=1,\ldots,S, \label{eq:HydroServiceCompletions}\\
		Y^{r,m}(t) &= t + \frac{\sqrt{x_{r,m}}}{r} \int_0^t L^{r,m}(s) \, ds, \hspace{0.5cm} \forall j=1,\ldots,S, \label{eq:HydroAggregatedArrivalTime}\\
		T_j^{r,m}(t) &= p_j \frac{\lambda}{\mu}t + \frac{\sqrt{x_{r,m}}}{r} \int_0^t Z_j^{r,m}(s) \, ds, \hspace{0.5cm} \forall j=1,\ldots,S, \label{eq:HydroBusyTime}\\
		Z_j^{r,m}(t) &= \min\left\{Q_j^{r,m}(t), \frac{1}{\sqrt{x_{r,m}}} \gamma p \sqrt{\frac{\lambda r}{\mu}} \right\}, \hspace{0.5cm} \forall j=1,\ldots,S, \label{eq:HydroBusyServers}\\
		L^{r,m}(t) &= - \sum_{j=1}^S \left( Q_j^{r,m}(t) - \frac{1}{\sqrt{x_{r,m}}} \beta p_j \sqrt{\frac{\lambda r}{\mu}} \right)^+ \label{eq:HydroDrivingCars}\\
		A^{r,m}_{ij,i}(t) & \textrm{ can only increase when } \frac{Q_i^{r,m}(t)}{p_i} \leq \frac{Q_j^{r,m}(t)}{p_j} \hspace{0.5cm} \forall \{i,j\} \in E. \label{eq:HydroDerivativeIncreasePossibility}
		\end{align}
		
		In order to show that $\mathbb{X}^{r,m}$ is almost (with the exception of certain events) Lipschitz continuous, we would like to exclude these certain events, i.e.\ show that such events are unlikely to occur.
		
		\begin{lemma}
			Fix $\epsilon >0$, $M > 0$ and $T>0$. For $r$ large enough, there exists a constant $N >0$ (only depending on $\lambda$, $\mu$, and $\{p_{ij}; \{i,j\} \in E\}$) such that
			\begin{align}
			\Prob\left( \max_{m < \sqrt{r}T} \sup_{0 \leq t_1 \leq t_2 \leq M} \left\{ \left\lvert  A^{r,m}(t_2) - A^{r,m}(t_1) \right\rvert - N (t_2-t_1) \right\} \geq  \epsilon \right) \leq \epsilon,\label{eq:AlmostLipCont1}\\
			\Prob\left( \max_{m < \sqrt{r}T} \sup_{0 \leq t_1 \leq t_2 \leq M} \left\{ \left\lvert  D^{r,m}(t_2) - D^{r,m}(t_1) \right\rvert - N(t_2-t_1) \right\} \geq \epsilon \right) \leq \epsilon.\label{eq:AlmostLipCont2}
			\end{align}
			Moreover,
			\begin{align}
			\Prob\left( \max_{m < \sqrt{r}T} \left\lVert Y^{r,m}(t) - \frac{1}{p_{ij}\lambda} A_{ij}^{r,m}(t) \right\rVert_M \geq \epsilon \right) \leq \epsilon,  \hspace{0.5cm} \{i,j\} \in E \label{eq:YCloseToExpectation}\\
			\Prob\left( \max_{m < \sqrt{r}T} \left\lVert  T_j^{r,m}(t) - \frac{1}{\mu} D_j^{r,m}(t) \right\rVert_M \geq  \epsilon \right) \leq \epsilon, \hspace{0.5cm} j=1,\ldots,S,\label{eq:TCloseToExpectation}
			\end{align}
			\label{lem:UnlikelyBadEvents}
		\end{lemma}
		
\begin{proof}
		First, we note that due to the memoryless property which both the arrival and service completion processes satisfy, the choice of $m$ is irrelevant and thus can be made arbitrarily. To prove~\eqref{eq:AlmostLipCont1}, we start by showing that for every $\{i,j\} \in E$,
		\begin{align*}
		\Prob\left( \sup_{0 \leq t_1 \leq t_2 \leq M} \left\{   A_{ij}^{r,m}(t_2) - A_{ij}^{r,m}(t_1)  - \frac{1}{p_{ij} \lambda} (t_2-t_1) \right\} \geq  \epsilon \right) \leq \epsilon.
		\end{align*}
		From~\eqref{eq:HydroAggregatedArrivalTime} and~\eqref{eq:HydroDrivingCars}, we observe that $Y_{r,m}(t) \leq t$ and non-decreasing. Due to the properties of Poisson processes,
		\begin{align*}
		A_{ij}^{r,m}(t_2) - A_{ij}^{r,m}(t_1) \overset{d}{=} \frac{\Lambda_{ij}\left(\sqrt{x_{r,m}} Y^{r,m}(t_2)\right) -\Lambda_{ij}\left(\sqrt{x_{r,m}} Y^{r,m}(t_1)\right)}{\sqrt{x_{r,m}}} \overset{d}{\leq}  \frac{\Lambda_{ij}\left(\sqrt{x_{r,m}} t_2\right) -\Lambda_{ij}\left(\sqrt{x_{r,m}} t_1 \right)}{\sqrt{x_{r,m}}}.
		\end{align*}
		Therefore,
		\begin{align*}
		\Prob&\left( \sup_{0 \leq t_1 \leq t_2 \leq M} \left\{   A_{ij}^{r,m}(t_2) - A_{ij}^{r,m}(t_1)  - \frac{t_2-t_1}{p_{ij} \lambda}  \right\} \geq  \epsilon \right) \\
		& \hspace{3cm}\leq \Prob\left( \sup_{0 \leq t_1 \leq t_2 \leq M} \left\{  \frac{\Lambda_{ij}\left(\sqrt{x_{r,m}} t_2\right) -\Lambda_{ij}\left(\sqrt{x_{r,m}} t_1 \right)}{\sqrt{x_{r,m}}}  - \frac{t_2-t_1}{p_{ij} \lambda}  \right\} \geq  \epsilon \right) \\
		&\hspace{3cm} \leq \Prob\left( \left\lVert \frac{\Lambda_{ij}\left(\sqrt{x_{r,m}} t\right)}{\sqrt{x_{r,m}}}  - \frac{t}{p_{ij} \lambda} \right\rVert_M \geq  \epsilon/2 \right)  \leq \frac{\epsilon}{2 M^2 \sqrt{x_{r,m}}} \leq \frac{\epsilon}{2 M^2 \sqrt{r}} ,
		\end{align*}
		where the final to last inequality follows from Proposition~4.3 in~\cite{Bramson1998}. Choosing $N = 1/(\lambda \min_{\{i,j\} \in E} p_{ij})$  and applying the union bound twice, we obtain
		\begin{align*}
		\Prob\left( \max_{m < \sqrt{r}T} \sup_{0 \leq t_1 \leq t_2 \leq M} \left\{ \left\lvert  A^{r,m}(t_2) - A^{r,m}(t_1) \right\rvert - N (t_2-t_1) \right\} \geq  \epsilon \right) \leq \frac{\epsilon T |E|}{2 M^2},
		\end{align*}
		which yields~\eqref{eq:AlmostLipCont1}.
		
		The proof for~\eqref{eq:AlmostLipCont2} is completely analogous, but with minor adaptions as one uses $T_j^{r,m}(t) \leq p_j \lambda / \mu t$ instead of $Y^{r,m}(t) \leq t$. We conclude that~\eqref{eq:AlmostLipCont1} and~\eqref{eq:AlmostLipCont2} show that the hydrodynamically scaled arrival process and service completion process are almost Lipschitz continuous.
		
		In order to prove~\eqref{eq:YCloseToExpectation} and~\eqref{eq:TCloseToExpectation}, we introduce the following processes. Let $\{u_{ij}(l), l\geq 1\}$ be independent exponentially distributed random variables with rate $p_{ij} \lambda$, representing the time that a car has before it needs recharging at either station~$i$ or~$j$. Let $\{v_j(l), l\geq 1\}$ be independent exponentially distributed random variables with rate $\mu$, representing the service requirement (recharging time) of a battery at station $j$. Define
		\begin{align*}
		U_{ij}(n) &= \sum_{l=1}^n u_{ij}(l), \hspace{0.5cm} \{i,j\} \in E \\
		V_{j}(n) &= \sum_{l=1}^n v_{j}(l), \hspace{0.5cm} j=1,\ldots,S,
		\end{align*}
		the aggregated interarrival time of $n$ cars that will choose between stations $i$ and $j$, and the total service requirement of $n$ batteries at station $j$, respectively. We observe the identities
		\begin{align*}
		\Lambda_{ij}(t) = \max\{n : U_{ij}(n) \leq t\},  \hspace{0.5cm} S_j(t) = \max\{n : V_j(n) \leq t\}, \hspace{0.5cm} t\geq 0.
		\end{align*}
		Moreover, due to~\eqref{eq:IdentityArrivals2} and~\eqref{eq:IdentityServiceCompletions}, we observe
		\begin{align}
		U_{ij}(A_{ij}^r(t)) \leq &Y^r(t) \leq  U_{ij}(A_{ij}^r(t)+1), \hspace{0.75cm} \{i,j\} \in E \label{eq:arrivalBoundsBusy} \\
		V_j(D_j^r(t)) \leq &T_j^r(t) \leq V_j(D_j(t)+1), \hspace{1cm} j=1,\ldots,S. \label{eq:serverBoundsBusy}
		\end{align}
		As in~\cite{DaiTezcan2011}, we define for notational convention, for $b=(b_1,b_2) \in \mathbb{N}$,
		\begin{align}
		U^{r,m}_{ij} \left(A_{ij}^{r,m}(t),b\right) &= \frac{1}{\sqrt{x_{r,m}}} \left( U^r_{ij}\left( A_{ij}^{r} \left( \frac{m}{\sqrt{r}} + \frac{\sqrt{x_{r,m}} t}{r} \right) + b_1 \right) - U^r_{ij}\left( A_{ij}^{r} \left( \frac{m}{\sqrt{r}}  \right) + b_2 \right) \right), \\
		V^{r,m}_j \left(D_j^{r,m}(t),b\right) &= \frac{1}{\sqrt{x_{r,m}}} \left( V^r_j\left( D_j^{r} \left( \frac{m}{\sqrt{r}} + \frac{\sqrt{x_{r,m}} t}{r} \right) + b_1 \right) - V^r_j\left( D_j^{r} \left( \frac{m}{\sqrt{r}}  \right) + b_2 \right) \right).
		\end{align}
		In view of~\eqref{eq:arrivalBoundsBusy} and~\eqref{eq:serverBoundsBusy}, this yields the inequalities
		\begin{align}
		U_{ij}^{r,m}(A_{ij}^r(t),(0,1)) \leq &Y^{r,m}(t) \leq  U_{ij}^{r,m}(A_{ij}^r(t),(1,0)), \hspace{0.8cm} \{i,j\} \in E \label{eq:hydroArrivalBoundsBusy} \\
		V_j^{r,m}(D_j(t),(0,1)) \leq &T_j^{r,m}(t) \leq V_j^{r,m}(D_j(t),(1,0)), \hspace{1cm} j=1,\ldots,S. \label{eq:hydroServerBoundsBusy}
		\end{align}
		Using these processes, we first prove the following bounds:
		\begin{align}
		\Prob\left( \max_{m < \sqrt{r}T} \left\lVert  U_{ij}(A_{ij}^{r,m}(t),b) - \frac{1}{p_{ij}\lambda} A_{ij}^{r,m}(t) \right\rVert_M \geq \epsilon \right) \leq \epsilon, \hspace{1cm} \forall \{i,j\} \in E, \label{eq:ArrivalsCloseToExpectation}\\
		\Prob\left( \max_{m < \sqrt{r}T} \left\lVert  V_{j}(D^{r,m}(t),b) - \frac{1}{\mu} D_j^{r,m}(t) \right\rVert_M \geq  \epsilon \right) \leq \epsilon, \hspace{1cm} j=1,\ldots,S,\label{eq:ServicesCloseToExpectation}
		\end{align}
		for $b=(1,0)$ and $b=(0,1)$. The proof is similar to that of (78) in~\cite{Tezcan2008}. We observe that the proof of~\eqref{eq:AlmostLipCont1} implies that in particular
		\begin{align*}
		\Prob&\left( A_{ij}^{r}\left(\frac{\sqrt{x_{r,m}}M}{r} \right)  \geq \frac{2M}{p_{ij}\lambda} \sqrt{x_{r,m}} \right) \leq \frac{\epsilon}{M^2 \sqrt{r}},
		\end{align*}
		and hence also
		\begin{align*}
		\Prob&\left( A_{ij}^{r}\left(\frac{\sqrt{x_{r,m}}M}{r} \right) +1  \geq \frac{3M}{p_{ij}\lambda} \sqrt{x_{r,m}} \right) \leq \frac{\epsilon}{M^2 \sqrt{r}}
		\end{align*}
		for $r$ large enough. Proposition~4.2 of~\cite{Bramson1998} states
		\begin{align*}
		\Prob \left( \left\lVert U_{ij}(l) - \frac{l}{p_{ij}\lambda} \right\rVert_n  \geq \epsilon n \right) \leq \frac{\epsilon}{n}.
		\end{align*}
		Therefore, it follows that
		\begin{align*}
		\Prob \left( \left\lVert U_{ij}(A_{ij}^r(t)) - \frac{A_{ij}^r(t)}{p_{ij}\lambda} \right\rVert _{\sqrt{x_{r,m}}M/r} \geq \epsilon \frac{2M\sqrt{x_{r,m}}}{p_{ij}\lambda} \right) \leq \frac{\epsilon}{\sqrt{r}} \left(\frac{1}{M^2} + \frac{p_{ij}\lambda}{2 M} \right),
		\end{align*}
		and
		\begin{align*}
		\Prob \left( \left\lVert U_{ij}(A_{ij}^r(t)+1) - \frac{A_{ij}^r(t)}{p_{ij}\lambda} \right\rVert _{\sqrt{x_{r,m}}M/r} \geq \epsilon \frac{3M\sqrt{x_{r,m}}}{p_{ij}\lambda} \right) \leq \frac{\epsilon}{\sqrt{r}} \left(\frac{1}{M^2} + \frac{p_{ij}\lambda}{3 M} \right).
		\end{align*}
		Increasing $\epsilon$ appropriately, we obtain
		\begin{align*}
		\Prob \left( \left\lVert U_{ij}^{r,m}(A_{ij}^{r,m}(t),b) - \frac{A_{ij}^{r,m}(t)}{p_{ij}\lambda} \right\rVert _{M} \geq \epsilon \right)  \leq \frac{\epsilon}{T\sqrt{r}}
		\end{align*}
		for both $b=(0,0)$ and $b=(1,0)$.
		Using the union bound yields
		\begin{align*}
		\Prob \left( \max_{m < \sqrt{r} T} \left\lVert U_{ij}^{r,m}(A_{ij}^{r,m}(t),b) - \frac{A_{ij}^{r,m}(t)}{p_{ij}\lambda} \right\rVert _{M} \geq \epsilon \right)  \leq \epsilon.
		\end{align*}
		for both $b=(0,0)$ and $b=(1,0)$. To conclude the proof for $b=(0,1)$ as well, we observe
		\begin{align*}
		\Prob & \left( \max_{m < \sqrt{r} T}  \left\lVert U_{ij}^{r,m}(A_{ij}^{r,m}(t),b) - U_{ij}^{r,m}(A_{ij}^{r,m}(t),(0,0))\right\rVert _{M} \geq \epsilon \right) \\
		&\leq \Prob \left( \max_{m < \sqrt{r} T}  \left\lvert  U_{ij}\left(A_{ij}^{r} \left(\frac{m}{\sqrt{r}} \right) +1 \right) -  U_{ij}\left(A_{ij}^{r} \left(\frac{m}{\sqrt{r}}\right) \right) \right\rvert \geq \epsilon \sqrt{x_{r,m}} \right)  \leq \Prob\left(u_{ij}^{r,T,\max} \geq \sqrt{x_{r,m}} \epsilon \right) \leq \epsilon,
		\end{align*}
		where the final inequality follows from Lemma~5.1 in~\cite{Bramson1998} with
		\begin{align*}
		u_{ij}^{r,T,\max}= \max\{u_{ij}(l) : U_i(l-1) \leq r T\}.
		\end{align*}
		
		The proof of~\eqref{eq:ServicesCloseToExpectation} is analogous to~\eqref{eq:ArrivalsCloseToExpectation}, replacing the arrival processes by the service processes. Equations~\eqref{eq:YCloseToExpectation} and~\eqref{eq:TCloseToExpectation} are then a direct consequence of~\eqref{eq:hydroArrivalBoundsBusy} and~\eqref{eq:hydroServerBoundsBusy}.
\hfill\end{proof}
		
		Using the previous result, we can show that $\mathbb{X}$ is almost Lipschitz continuous.
		\begin{proposition}
			Fix $\epsilon >0$, $M > 0$ and $T>0$. For $r$ large enough,
			\begin{align*}
			\Prob\left( \max_{m < \sqrt{r}T} \sup_{0 \leq t_1 \leq t_2 \leq M} \left\{ \left\lvert  \mathbb{X}^{r,m}(t_2) - \mathbb{X}^{r,m}(t_1) \right\rvert - N (t_2-t_1) \right\} \geq \epsilon \right) \leq \epsilon,
			\end{align*}
			where $N < \infty$ is constant (depending only on $\lambda$, $\mu$ and $\{p_{ij}; \{i,j\} \in E\}$).
			\label{prop:AlmostLipschitzContinuousHydroScaling}
		\end{proposition}
		
\begin{proof}
		This follows in a straightforward way from Lemma~\ref{lem:UnlikelyBadEvents} and the hydrondynamically-scaled system equations. That is, let $\mathcal{V}^r$ denote the intersection of the complements of the events given in equations~\eqref{eq:AlmostLipCont1}-\eqref{eq:TCloseToExpectation}, so $\Prob(\mathcal{V}^r) \leq 1-N_0 \epsilon$ with $N_0$ the number of equations in Lemma~\ref{lem:UnlikelyBadEvents}. We note that in order to prove the proposition, it suffices to show that for every $\omega \in \mathcal{V}^r$, and for every $t_1,t_2 \in [0,T]$ and $m < \sqrt{r}T$,
		\begin{align}
		\left\lvert  \mathbb{X}^{r,m}(t_2) - \mathbb{X}^{r,m}(t_1) \right\rvert \leq N_1 (t_2-t_1)  + N_2 \epsilon, \label{eq:toShowAlternativeLipschitzContinuous}
		\end{align}
		where $N_1$ and $N_2$ are only dependent on the system parameters (i.e.~$\lambda$, $\mu$, $p$). Let $t_1,t_2 \in [0,T]$ with $t_1 \leq t_2$. By definition of $\mathcal{V}^r$,
		\begin{align*}
		\left\lvert  A^{r,m}(t_2) - A^{r,m}(t_1) \right\rvert \leq N(t_2-t_1) + \epsilon,
		\end{align*}
		and
		\begin{align*}
		\left\lvert  D^{r,m}(t_2) - D^{r,m}(t_1) \right\rvert \leq N(t_2-t_1) + \epsilon,
		\end{align*}
		for $N$ as in Lemma~\ref{lem:UnlikelyBadEvents}. Due to~\eqref{eq:HydroArrivals1},
		\begin{align*}
		\left\lvert  A^{r,m}_d(t_2) - A^{r,m}_d(t_1) \right\rvert \leq \left\lvert  A^{r,m}(t_2) - A^{r,m}(t_1) \right\rvert \leq N(t_2-t_1) + \epsilon.
		\end{align*}
		In view of~\eqref{eq:HydroQueueLength} and~\eqref{eq:HydroArrivals1}, we observe
		\begin{align*}
		\left\lvert  Q^{r,m}(t_2) - Q^{r,m}(t_1) \right\rvert \leq |E| \left\lvert  A^{r,m}(t_2) - A^{r,m}(t_1) \right\rvert + S \left\lvert  D^{r,m}(t_2) - D^{r,m}(t_1) \right\rvert \leq (|E|+S) N(t_2-t_1) + 2 \epsilon.
		\end{align*}
		Due to~\eqref{eq:YCloseToExpectation},
		\begin{align*}
		\left\lvert  Y^{r,m}(t_2) - Y^{r,m}(t_1) \right\rvert \leq \sum_{\{i,j\}\in E} \frac{ \left\lvert  A^{r,m}(t_2) - A^{r,m}(t_1) \right\rvert}{p_{ij}\lambda} + 2\epsilon \leq  \sum_{\{i,j\}\in E} \frac{N}{p_{ij}\lambda}(t_2-t_1) + \left( \sum_{\{i,j\}\in E} \frac{1}{p_{ij}\lambda}+ 2\right) \epsilon.
		\end{align*}
		and similarly, due to~\eqref{eq:TCloseToExpectation},
		\begin{align*}
		\left\lvert  T^{r,m}(t_2) - T^{r,m}(t_1) \right\rvert \leq  \frac{1}{\mu} \left\lvert  D^{r,m}(t_2) - D^{r,m}(t_1) \right\rvert + 2\epsilon \leq \frac{N}{\mu}(t_2-t_1) + \left(\frac{1}{\mu}+2\right) \epsilon.
		\end{align*}
		In view of~\eqref{eq:HydroBusyServers},
		\begin{align*}
		\left\lvert  Z^{r,m}(t_2) - Z^{r,m}(t_1) \right\rvert \leq \left\lvert  Q^{r,m}(t_2) - Q^{r,m}(t_1) \right\rvert \leq (|E|+S) N(t_2-t_1) + 2 \epsilon.
		\end{align*}
		Finally, due to~\eqref{eq:HydroDrivingCars},
		\begin{align*}
		\left\lvert  L^{r,m}(t_2) - L^{r,m}(t_1) \right\rvert \leq S \left\lvert  Q^{r,m}(t_2) - Q^{r,m}(t_1) \right\rvert  \leq S(|E|+S) N(t_2-t_1) + 2 S \epsilon.
		\end{align*}
		We conclude that~\eqref{eq:toShowAlternativeLipschitzContinuous} is satisfied, as each process in~$\mathbb{X}^{r,m}$ satisfies this property.
\hfill\end{proof}
		
		As is done in~\cite{Bramson1998,DaiTezcan2011}, one can take $\epsilon$ appropriately small for every system. That is, for fixed $M>0$, $N>0$ and $T>0$, let
		\begin{align*}
		\mathcal{K}_0^r = \left\{ \max_{m < \sqrt{r}T} \sup_{0 \leq t_1 \leq t_2 \leq M} \left\lvert  \mathbb{X}^{r,m}(t_2) - \mathbb{X}^{r,m}(t_1) \right\rvert \geq N (t_2-t_1) + \epsilon(r)  \right\},
		\end{align*}
		where $\epsilon(r)\rightarrow 0$ as $r \rightarrow \infty$ is a sequence of positive real numbers. Moreover, in view of Lemma~\ref{lem:BoundedXrm}, let for that same sequence $\{\epsilon(r)\}_{r \in \mathbb{R}}$,
		\begin{align*}
		\mathcal{H}^r = \left\{ \max_{m < \sqrt{r} T} \left\{ \frac{\sqrt{x_{r,m}}}{r} \lVert Q^{r,m}(t) \rVert_M , \frac{\sqrt{x_{r,m}}}{r} \lVert L^{r,m}(t) \rVert_M \right\} \leq \epsilon(r) \right\}.
		\end{align*}
		Let $\mathcal{K}^r$ denote the intersection of $\mathcal{K}_0^r$, $\mathcal{H}^r$, and the complements of the events in Lemma~\ref{lem:UnlikelyBadEvents}. We note that Lemma~\ref{lem:BoundedXrm}, Lemma~\ref{lem:UnlikelyBadEvents} and Proposition~\ref{prop:AlmostLipschitzContinuousHydroScaling} continue to hold for the sequence $\epsilon(r)$ if $\epsilon(r) \rightarrow 0$ sufficiently slow. We conclude that $\Prob(\mathcal{K}^r) \rightarrow 1$ as $r \rightarrow \infty$.
		
		\begin{corollary}
			Fix $M>0$ and choose $N>0$ and $\epsilon(r)$ as above. Then,
			\begin{align*}
			\lim_{r \rightarrow \infty} \Prob(\mathcal{K}^r) =1.
			\end{align*}
		\end{corollary}
		
		Following the framework of~\cite{Bramson1998}, we can use these results to state that the hydrodynamically scaled system convergences to a hydrodynamic limit. Fix $M >0$ and let $\tilde{E}$ be the set of right-continuous functions $x: [0,M] \rightarrow \mathbb{R}^d$ with left limits. Let
		\begin{align*}
		E'= \{x \in \tilde{E} : |x(0)| \leq 1, |x(t_2)-x(t_1)| \leq N |t_2-t_1| \;\;\; \forall t_1,t_2 \in [0,M] \}.
		\end{align*}
		Moreover, we set
		\begin{align*}
		E^r= \{\mathbb{X}^{r,m},  m < \sqrt{r} T, \omega \in \mathcal{K}^r \},
		\end{align*}
		and
		\begin{align*}
		\mathcal{E} = \{E^r, r \in \mathbb{N}\}.
		\end{align*}
		We remark that these definitions are not related to $E$, the set of all possible pairs of stations where cars can move to.
		
		\begin{definition}
			A hydrodynamic limit of $\mathcal{E}$ is a point $x \in \tilde{E}$ such that for all $\epsilon >0$  there exists a  $r_0 \in \mathbb{N}$ so that for every $r \geq r_0$ there is some $y \in E^r$ such that $\lvert x(\cdot) - y(\cdot) \rvert_M < \epsilon$.
		\end{definition}
		
		Since $|\mathbb{X}^{r,m}(0)|\leq1$, the following result is a consequence of Proposition~4.1 in~\cite{Bramson1998}.
		\begin{corollary}
			Let $\tilde{E}, E^r, \mathcal{E}$ be as above. Fix $\epsilon>0$, $M>0$, $T\geq 0$, and choose $r$ large enough. Then, for $\omega \in \mathcal{K}^r$ and any $m < \sqrt{r}T$,
			\begin{align*}
			\lVert \mathbb{X}^{r,m}(\cdot) - \mathbb{X}(\cdot) \rVert_M \leq \epsilon
			\end{align*}
			for some hydrodynamic limit $\mathbb{X}(\cdot) \in E'$ of $\mathcal{E}$.
			\label{cor:ConvergenceToHydrodynamicLimit}
		\end{corollary}
		
		Finally, to conclude this section, we derive the equations that are satisfied by any hydrodynamic limit.
		\begin{proposition}
			Let $M>0$ be fixed, and let $\tilde{\mathbb{X}}$ be a hydrodynamic limit of $\mathcal{E}$ over $[0,M]$. Then $\tilde{\mathbb{X}}$ satisfies the following equations:
			\begin{align}
			\tilde{A}_{ij}(t) &= \tilde{A}_{ij,i}(t) + \tilde{A}_{ij,j}(t) \hspace{0.5cm} \forall \{i,j\} \in E, \label{eq:HydroLimitArrivals1Fluid}\\
			\tilde{A}_{ij}(t) &=  p_{ij} \lambda \tilde{Y}(t) = p_{ij} \lambda t \hspace{0.5cm} \forall \{i,j\} \in E,  \label{eq:HydroLimitArrivals2Fluid}\\
			\tilde{Q}_j(t) &= \tilde{Q}_j(0) + \sum_{i : \{i,j\} \in E } \tilde{A}_{ij,j}(t)-\tilde{D}_j(t), \hspace{0.5cm} \forall j=1,\ldots,S, \label{eq:HydroLimitQueueLengthFluid} \\
			\tilde{D}_j(t) &= \mu \tilde{T}_j(t) = p_j \lambda t , \hspace{0.5cm} \forall j=1,\ldots,S,\label{eq:HydroLimitServiceCompletionsFluid}\\
			\tilde{Y}(t) &= t, \hspace{0.5cm} \forall j=1,\ldots,S, \label{eq:HydroLimitAggregatedArrivalTimeFluid}\\
			\tilde{T}_j(t) &= p_j \lambda/\mu t, \hspace{0.5cm} \forall j=1,\ldots,S, \label{eq:HydroLimitBusyTimeFluid}\\
			\tilde{A}_{ij,i}'(t) &= \left\{ \begin{array}{ll}
			p_{ij} \lambda & \textrm{if }  \frac{\tilde{Q}_i(t)}{p_i} < \frac{\tilde{Q}_j(t)}{p_j}\\
			0 & \textrm{if }  \frac{\tilde{Q}_j(t)}{p_j} > \frac{\tilde{Q}_i(t)}{p_i}
			\end{array} \right.  \hspace{0.5cm} \forall \{i,j\} \in E. \label{eq:HydroLoadBalancingArrivalsDerivative}
			\end{align}
			\label{prop:HydroLimitingEquations}
		\end{proposition}
		
		\begin{remark}\normalfont
			We cannot provide such general equations for $\tilde{Z}(\cdot)$ or $\tilde{L}(\cdot)$, since these limits depend on $x_{r,m}$. That is, the processes $\tilde{Z}^{r,m}(\cdot)$ and $\tilde{L}^{r,m}(\cdot)$ converge to a limit, but the limiting process may differ for different $m$. In the proof, we specify the limiting equations of these processes as well.
		\end{remark}
		
\begin{proof}[Proof of Proposition~\ref{prop:HydroLimitingEquations}]
		Let $\tilde{\mathbb{X}}$ be a hydrodynamic limit of~$\mathcal{E}$. For a given $\delta >0$, choose $(r,m)$ such that $\epsilon(r) \leq \delta$, and
		\begin{align*}
		\lVert \tilde{\mathbb{X}}(t) - \mathbb{X}^{r,m}(t, \omega) \rVert_M \leq \delta.
		\end{align*}
		Due to~\eqref{eq:HydroAggregatedArrivalTime}, \eqref{eq:HydroBusyTime} and $\omega \in \mathcal{H}^r$, we derive
		\begin{align*}
		\lVert \tilde{Y}(t) - t \rVert_M \leq (1+M) \delta, \hspace{1cm} \lVert \tilde{T}_j(t) - p_j \lambda / \mu t \rVert_M \leq (1+M) \delta, \hspace{0.25cm} j=1,\ldots,S.
		\end{align*}
		From~\eqref{eq:YCloseToExpectation} and~\eqref{eq:TCloseToExpectation}, we obtain
		\begin{align*}
		\lVert \tilde{A}_{ij}(t) - p_{ij} \lambda t \rVert_M \leq (2+M) \delta, \hspace{1cm} \{i,j\} \in E,
		\end{align*}
		and
		\begin{align*}
		\lVert \tilde{D}_{j}(t) - p_{j} \lambda t \rVert_M \leq (2+M) \delta, \hspace{1cm} j=1,\ldots,S.
		\end{align*}
		Equation~\eqref{eq:HydroLimitArrivals1Fluid} is a clear consequence of~\eqref{eq:HydroArrivals1}. Combining the above equations, we observe
		\begin{align*}
		\left\lVert \tilde{Q}_j(t) - \tilde{Q}_j(0) - \sum_{i : \{i,j\} \in E } \tilde{A}_{ij,j}(t) +\tilde{D}_j(t) \right\rVert_M \leq 2(|E|+S)(2+M)\delta
		\end{align*}
		These bounds imply that any hydrodynamic limit satisfies equations~\eqref{eq:HydroLimitArrivals1Fluid}-\eqref{eq:HydroLimitBusyTimeFluid}. Finally, we still have to show~\eqref{eq:HydroLoadBalancingArrivalsDerivative}. If for some $t \in [0,M]$,
		\begin{align*}
		\frac{\tilde{Q}_j(t)}{p_j}  > \frac{\tilde{Q}_i(t)}{p_i} ,
		\end{align*}
		then by continuity of $\tilde{\mathbb{X}}$, there exists a $\eta >0$ such that this holds for all $s \in [t-\eta, t+\eta)$, and also
		\begin{align*}
		\frac{\tilde{Q}^{r,m}_j(t)}{p_j}  > \frac{\tilde{Q}^{r,m}_i(t)}{p_i} .
		\end{align*}
		Due to~\eqref{eq:HydroDerivativeIncreasePossibility}, this implies that $A_{ij,i}^{r,m}(s)$ is constant on $s \in [t-\eta, t+\eta])$. Therefore, its limit is also constant on $[t-\eta, t+\eta]$, and hence the derivative is zero. On the other hand, if
		\begin{align*}
		\frac{\tilde{Q}_j(t)}{p_j}  < \frac{\tilde{Q}_i(t)}{p_i} ,
		\end{align*}
		then by continuity of $\tilde{\mathbb{X}}$, there exists a $\eta >0$ such that this holds for all $s \in [t-\eta, t+\eta])$, and
		\begin{align*}
		\frac{\tilde{Q}^{r,m}_j(t)}{p_j}  < \frac{\tilde{Q}^{r,m}_i(t)}{p_i} .
		\end{align*}
		Since $\tilde{A}'_{ij,j}=0$, and due to~\eqref{eq:HydroArrivals1} with limiting process~\eqref{eq:HydroLimitArrivals2Fluid},
		\begin{align*}
		\tilde{A}'_{ij,i} = \lim_{\eta \downarrow 0} \frac{\tilde{A}'_{ij}(t+ \eta)-\tilde{A}'_{ij}(t)}{\eta} = p_{ij} \lambda.
		\end{align*}\hfill\end{proof}
		
		\subsection{The SSC function}
		In this section, we introduce the state space collapse (SSC) function under which we show multiplicative state space collapse. The SSC function we use in our paper is $g: \mathbb{R}^S \rightarrow \mathbb{R}$, defined as
		\begin{align}
		g(q) = \max_{1 \leq j \leq S} \frac{q_j}{p_j} -  \min_{1 \leq j \leq S} \frac{q_j}{p_j}
		\label{eq:SSCfunctionDefinition}
		\end{align}
		where $q=(q_1,\ldots,q_S)$. We note that $g(\cdot)$ is a non-negative continuous function and satisfies
		\begin{align*}
		g(\alpha q) = \alpha g(q)
		\end{align*}
		for every $\alpha >0$. 
		
		\begin{lemma}
			Suppose $g: \mathbb{R}^S \rightarrow \mathbb{R}$ is defined as in~\eqref{eq:SSCfunctionDefinition}. Then,
			\begin{align*}
			g(\tilde{Q}(t)) \leq H(t) \hspace{0.5cm} \forall t\geq 0,
			\end{align*}
			for every hydrodynamic model solution $\tilde{\mathbb{X}}$ satisfying $|\tilde{\mathbb{X}}(0)|\leq 1$, where
			\begin{align}
			H(t) = \left( \frac{2}{\min_{1 \leq j \leq S} p_j} -h t \right)^+
			\label{eq:HfunctionDefinition}
			\end{align}
			with $h >0$ some constant that depends only on $\lambda$ and $\{p_{ij}, \{i,j\} \in E\}$. Moreover, if $g(\tilde{Q}(0))=0$ and $|\tilde{\mathbb{X}}(0)|\leq 1$, then $g(\tilde{Q} (t))=0$ for all $t \geq 0$.
			\label{lem:BoundedSSCfunction}
		\end{lemma}
		
\begin{proof}
		The proof relies heavily on the ideas used in the proof for the fluid limit. Let
		\begin{align*}
		h = \min\left\{ \frac{\lambda }{\sum_{i \in \mathcal{I}} p_i} \sum_{\substack{\{i,j\} \in E ,\\ i \in \mathcal{I} \cup j \in \mathcal{I}}} p_{ij} - \frac{\lambda }{\sum_{j \in \mathcal{J}} p_j} \sum_{\substack{\{i,j\} \in E ,\\ i \in \mathcal{J} \cap j \in \mathcal{J}}} p_{ij} : \mathcal{I}, \mathcal{J} \subset \{1,\ldots,S\}, \mathcal{I} \cap \mathcal{J} = \emptyset, \mathcal{I} \neq \emptyset, \mathcal{J} \neq \emptyset \right\}.
		\end{align*}
		Since
		\begin{align*}
		\frac{1 }{\sum_{i \in \mathcal{I}} p_i} \sum_{\substack{\{i,j\} \in E ,\\ i \in \mathcal{I} \cup j \in \mathcal{I}}} p_{ij} >1 , \hspace{1cm} \frac{1 }{\sum_{j \in \mathcal{J}} p_j} \sum_{\substack{\{i,j\} \in E ,\\ i \in \mathcal{J} \cap j \in \mathcal{J}}} p_{ij} <1,
		\end{align*}
		for any non-empty $\mathcal{I}, \mathcal{J} \subset \{1,\ldots,S\}$ with $\mathcal{I} \cap \mathcal{J} = \emptyset$, we observe that $h <0$. For a hydrodynamic limiting process $\tilde{\mathbb{X}}$, let $H_{\tilde{\mathbb{X}}}(\cdot)$ be given by
		\begin{align*}
		H_{\tilde{\mathbb{X}}}(t) = \left(g(\tilde{Q}(0)) - h t \right)^+.
		\end{align*}
		We note that this function is non-negative, and satisfies $H_{\tilde{\mathbb{X}}}(t) = 0$ for all $t \geq  g(\tilde{Q}(0))/h$.
		
		To show that $g(\tilde{Q}(t))$ is bounded by this function, we note that it suffices to show that whenever $g(\tilde{Q}(t)) >0$ with $t\geq 0$ being a regular point of $\tilde{\mathbb{X}}$,
		\begin{align*}
		g'(\tilde{Q}(t)) \leq -h,
		\end{align*}
		For this purpose, let
		\begin{align*}
		S_{\max}(t) = \left\{ i \in \{1,\ldots,S\} : \tilde{Q}_i(t)/p_i = \max_{1 \leq j \leq S}  \tilde{Q}_j(t)/p_j \right\},
		\end{align*}
		and
		\begin{align*}
		S_{\min}(t) = \left\{ i \in \{1,\ldots,S\} : \tilde{Q}_i(t)/p_i = \min_{1 \leq j \leq S}  \tilde{Q}_j(t)/p_j \right\}.
		\end{align*}
		Due to Lemma~2.8.6 in~\cite{Dai1999}, it holds for all $i,j \in S_{\max}(t)$ that $\tilde{Q}'_i(t)/p_i= \tilde{Q}'_j(t)/p_j$, and similarly, for all $i,j \in S_{\min}(t)$ it holds that $\tilde{Q}'_i(t)/p_i= \tilde{Q}'_j(t)/p_j$. Therefore, due to hydrodynamic limit equations~\eqref{eq:HydroLimitArrivals1Fluid}-\eqref{eq:HydroLoadBalancingArrivalsDerivative} and the observation that there is at least one station $j \not\in S_{\max}(t)$ such that $\{i,j\} \in E$,  it follows that
		\begin{align*}
		\frac{\tilde{Q}'_i(t)}{p_i} < \frac{\lambda }{ \sum_{i \in S_{\max}(t)} p_i} \sum_{\substack{\{i,j\} \in E ,\\ i ,j \in S_{\max}(t)}} p_{ij} - \lambda
		\end{align*}
		for all $i \in S_{\max}(t)$. Similarly, for all $i \in S_{\min}$,
		\begin{align*}
		\frac{\tilde{Q}'_i(t)}{p_i} > \frac{\lambda }{\sum_{i \in S_{\min}(t)} p_i} \sum_{\substack{\{i,j\} \in E ,\\ i \in S_{\min}(t) \cup j \in S_{\min}(t)}} p_{ij} - \lambda .
		\end{align*}
		We conclude that $g'(\tilde{Q}(t)) \leq -h$, and hence $g(\tilde{Q}(t)) \leq H_{\tilde{\mathbb{X}}}(t)$ for all $t\geq 0$. In particular, if $g(\tilde{Q}(0))=0$ and $|\tilde{\mathbb{X}}(0)|\leq 1$, it follows from the definition of $H_{\tilde{\mathbb{X}}}(\cdot)$ that $g(\tilde{Q} (t))=0$ for all $t \geq 0$.
		
		The first statement of the lemma follows since for every hydrodynamic model solution $\tilde{\mathbb{X}}$ satisfying $|\tilde{\mathbb{X}}(0)|\leq 1$, it holds that $g(\tilde{Q}(0)) \leq 2/\min_{1\leq j \leq S} p_j$. Hence, $H_{\tilde{\mathbb{X}}}(\cdot) \leq H(\cdot)$ for every hydrodynamic model solution $\tilde{\mathbb{X}}$ satisfying $|\tilde{\mathbb{X}}(0)|\leq 1$.
		
\hfill\end{proof}
		
		This result implies that the hydrodynamically scaled queue length almost satisfies this property as well.
		The next result is an immediate consequence of Corollary~\ref{cor:ConvergenceToHydrodynamicLimit}.

		\begin{corollary}
			Fix $\epsilon >0$, $M>0$ and $T>0$. Then, for every $\omega \in \mathcal{K}^r$,
			\begin{align*}
			g(Q^{r,m}(t)) \leq H(t) + \epsilon
			\end{align*}
			for all $t \in [0,M]$ and $m < \sqrt{r}T$, where $H(\cdot)$ is as in Lemma~\ref{lem:BoundedSSCfunction}. Moreover, if $g(\hat{Q}(0)) \rightarrow 0$ in probability as $r \rightarrow \infty$, then for all $\omega \in \mathcal{L}^r$ with
			\begin{align*}
			\mathcal{L}^r = \mathcal{K}^r \cap \left\{ \left\lvert g(Q^{r,0}(0)) \right\rvert \leq \epsilon \right\}
			\end{align*}
			it holds that
			\begin{align*}
			\lVert g(Q^{r,0}(t)) \rVert_M \leq \epsilon,
			\end{align*}
			and
			\begin{align*}
			\lim_{r \rightarrow \infty} \Prob(\mathcal{L}^r) = 1.
			\end{align*}
			\label{cor:BoundedSSCfunctionHydro}
		\end{corollary}

		\subsection{Multiplicative state space collapse}
		The goal of this section is to show multiplicative state space collapse for the SSC function defined in~\eqref{eq:SSCfunctionDefinition}. To do so, we first need to relate the hydrodynamic and diffusion scaling. That is, we observe that
		\begin{align*}
		Q_j^{r,m}(t) = p_j \sqrt{\frac{\lambda r/\mu }{x_{r,m}}} \hat{Q}_j^r\left( \frac{m}{\sqrt{r}}+\frac{\sqrt{x_{r,m}} t}{r} \right) = \frac{p_j \sqrt{\lambda / \mu}}{y_{r,m}} \hat{Q}_j^r\left( \frac{1}{\sqrt{r}} (m+y_{r,m} t) \right) ,
		\end{align*}
		where
		\begin{align*}
		y_{r,m} = \sqrt{\frac{x_{r,m}}{r}} = \max \left\{ \left\lvert p \sqrt{\frac{\lambda}{\mu}} \hat{Q}^r\left( \frac{m}{\sqrt{r}} \right) \right\rvert, \left\lvert \hat{L}^r\left( \frac{m}{\sqrt{r}} \right) \right\rvert, 1 \right\}
		\end{align*}
		with
		\begin{align*}
		\hat{L}^r(t) = \frac{L^r(t)-r}{\sqrt{r}}.
		\end{align*}
		Corollary~\ref{cor:BoundedSSCfunctionHydro} can be translated to the diffusion scaled process.
		Consider the SSC function $\hat{g}: \mathbb{R}^S \rightarrow \mathbb{R}$ defined as
		\begin{align*}
		\hat{g}(q) = \max_{1 \leq j \leq S} q_j -  \min_{1 \leq j \leq S} q_j
		\end{align*}
		with $q=(q_1,\ldots,q_S)$.
		
		\begin{corollary}
			Fix $\epsilon >0$, $M>0$ and $T>0$. Then for $r$ large enough, and $\omega \in \mathcal{K}^r$,
			\begin{align*}
			\hat{g}(\hat{Q}^r(t)) \leq \frac{y_{r,m}}{\sqrt{\lambda/\mu}} H \left( \frac{1}{y_{r,m}} (\sqrt{r}t-m) \right) + \epsilon \frac{y_{r,m}}{\sqrt{\lambda/\mu}}
			\end{align*}
			for all $t\in [0,T]$ with $m \in \mathbb{N}$ such that
			\begin{align*}
			\frac{m}{\sqrt{r}} \leq t \leq \frac{m+y_{r,m}M}{\sqrt{r}}.
			\end{align*}
			Also, for all $\omega \in \mathcal{L}^r$,
			\begin{align*}
			\lVert \hat{g}(\hat{Q}^r(t) \rVert_{M y_{r,0}/\sqrt{r}} \leq \epsilon \frac{y_{r,0}}{\sqrt{\lambda/\mu}}.
			\end{align*}
			\label{cor:BoundedSSCfunctionDiffusion}
		\end{corollary}
		
		Since $H(\cdot)$ is given as in~\eqref{eq:HfunctionDefinition}, we observe that $H(t)=0$ for all $t \geq 2/(h \min_{1 \leq j \leq S} p_j)$. We would like to show that $(\sqrt{r}t-m)/y_{r,m}$ can be chosen large enough to obtain a very small upper bound, and use that property to show multiplicative state space collapse.
		
		\begin{lemma}
			Suppose $M \geq 2(N+2)$ is fixed, and let
			\begin{align*}
			m_{r}(t) = \min\left\{m \in \mathbb{N} : \frac{m}{\sqrt{r}} \leq t \leq \frac{m+y_{r,m}M}{\sqrt{r}} \right\}
			\end{align*}
			Then for $r$ large enough,
			\begin{align*}
			\frac{\sqrt{r}t-m_r(t)}{y_{r,m_r(t)}}  \geq \frac{M}{2(N+2)}
			\end{align*}
			for every $\omega \in \mathcal{K}^r$ and $t \in ( M y_{r,0} /\sqrt{r} , T]$.
			\label{lem:makeDomainHLargeEnough}
		\end{lemma}
		
\begin{proof}
		For every $\omega \in \mathcal{K}^r$, by definition of the set,
		\begin{align*}
		\lvert \mathbb{X}^{r,m}(t_2) - \mathbb{X}^{r,m}(t_2) \rvert \leq N |t_2-t_1| +\epsilon,
		\end{align*}
		for $t_1,t_2 \in [0,M]$ and $m < \sqrt{r} T$. In particular, for $t_2=1/y_{r,m}$, $t_1=0$ and $\epsilon \leq 1$,
		\begin{align*}
		\max\left\{ \left\lvert Q^r\left( \frac{m+1}{\sqrt{r}} \right) - Q^r\left( \frac{m}{\sqrt{r}} \right) \right\rvert ,  \left\lvert L^r\left( \frac{m+1}{\sqrt{r}} \right) - L^r\left( \frac{m}{\sqrt{r}} \right) \right\rvert \right\} \leq \sqrt{x_{r,m}} \frac{N}{y_{r,m}} + \sqrt{x_{r,m}}.
		\end{align*}
		Applying the reverse triangle inequality, we observe
		\begin{align*}
		\left\lvert p \sqrt{\frac{\lambda}{\mu}} \hat{Q}^r\left( \frac{m+1}{\sqrt{r}} \right) \right\rvert &- \left\lvert p \sqrt{\frac{\lambda}{\mu}}  \hat{Q}^r\left( \frac{m}{\sqrt{r}} \right) \right\rvert \leq \left\lvert p \sqrt{\frac{\lambda}{\mu}}  \hat{Q}^r\left( \frac{m+1}{\sqrt{r}} \right) - p \sqrt{\frac{\lambda}{\mu}}  \hat{Q}^r\left( \frac{m}{\sqrt{r}} \right) \right\rvert \\
		&\leq N + y_{r,m} \leq (N+1) y_{r,m},
		\end{align*}
		and similarly,
		\begin{align*}
		\left\lvert \hat{L}^r\left( \frac{m+1}{\sqrt{r}} \right) \right\rvert &- \left\lvert  \hat{L}^r\left( \frac{m}{\sqrt{r}} \right) \right\rvert \leq  (N+1) y_{r,m}.
		\end{align*}
		Therefore, it always holds that
		\begin{align*}
		y_{r,m+1} \leq y_{r,m} + (N+1) y_{r,m}  = (N+2) y_{r,m}.
		\end{align*}
		For every $t \in (M y_{r,0}/\sqrt{r},T]$, it follows by definition of $m_r(t)$ that
		\begin{align*}
		\sqrt{r} t \geq m_r(t) -1 + y_{r,m_r(t)-1} M.
		\end{align*}
		In particular,
		\begin{align*}
		\frac{\sqrt{r}t-m_r(t)}{y_{r,m_r(t)}} \geq \frac{ y_{r,m_r(t)-1} M -1 }{y_{r,m_r(t)}} \geq \frac{M}{N+2} - \frac{ 1 }{y_{r,m_r(t)}} \geq \frac{M}{2(N+2)}
		\end{align*}
		where the last inequality follows since $M \geq 2(N+2)$.
\hfill\end{proof}
		
		Next, we show the main result of this section.
		\begin{theorem}
			Suppose $\hat{Q}^r(0) \rightarrow \hat{Q}(0)$ for some random vector $\hat{Q}(0)$. For every $T>0$, $\epsilon >0$ and $M<\infty$ with
			\begin{align*}
			M \geq \max\left\{ \frac{4(N+2)}{h \min_{1 \leq j \leq S} p_j} , 2(N+2), 1  \right\},
			\end{align*}
			it holds that
			\begin{align}
			\Prob\left( \frac{\sup_{M y_{r,0}/\sqrt{r} \leq t \leq T } \hat{g}(\hat{Q}^r(t)) }{ \max\{ \lVert \hat{Q}^r(t) \rVert_T , 1 \} } > \epsilon \right) \rightarrow 0 ,
			\label{eq:MultiplicativeSSCInitallyDifferent}
			\end{align}
			as $r \rightarrow \infty$. If, in addition,  $\hat{g}(\hat{Q}^r(0)) \rightarrow 0$ in probability as $r \rightarrow \infty$, then for every $T>0$,
			\begin{align}
			\frac{\Vert \hat{g}(\hat{Q}^r(t))\rVert_T }{ \max\{ \lVert\hat{Q}^r(t)\rVert_T ,  1 \} }  \overset{\Prob}{\rightarrow} 0 ,
			\label{eq:MultiplicativeSSCInitallyTheSame}
			\end{align}
			as $r \rightarrow \infty$.
			\label{thm:MultiplicativeSSC}
		\end{theorem}
		
\begin{proof}
		Fix $\eta >0$ and note that by construction there exists a $r_0$ such that for all $r > r_0$,
		\begin{align*}
		\Prob(\mathcal{K}^r) >1 -\eta.
		\end{align*}
		For every $\omega \in \mathcal{K}^r$, we have derived bounds that only require that $M$ is bounded. We note that for $M$ as in the statement of the theorem allows for every $t \in [0,T]$, $m/\sqrt{r} \leq t \leq (m+y_{r,m} M )/\sqrt{r}$ for some $m < \sqrt{r} M$. Moreover, it follows from Lemma~\ref{lem:makeDomainHLargeEnough} and~\eqref{eq:HfunctionDefinition}, that
		\begin{align*}
		H \left( \frac{1}{y_{r,m_r(t)}} (\sqrt{r}t-m_r(t)) \right) = 0
		\end{align*}
		for all $t \in ( M y_{r,0} /\sqrt{r} , T]$. In view of Corollary~\ref{cor:BoundedSSCfunctionDiffusion}, we obtain for every $\epsilon >0$,
		\begin{align*}
		\hat{g}(\hat{Q}^r(t)) \leq \epsilon \frac{y_{r,m_r(t)}}{\sqrt{\lambda/\mu}}
		\end{align*}
		for all $t \in ( M y_{r,0} /\sqrt{r} , T]$. Since for all $t \in [0,T]$,
		\begin{align*}
		y_{r,m_r(t)} = \max \left\{ \left\lvert p \sqrt{\frac{\lambda}{\mu}} \hat{Q}^r\left( \frac{m_r(t)}{\sqrt{r}} \right) \right\rvert, \left\lvert \hat{L}^r\left( \frac{m_r(t)}{\sqrt{r}} \right) \right\rvert, 1 \right\} \leq  \max \left\{  \left\lVert p \sqrt{\frac{\lambda}{\mu}} \hat{Q}^r\left( t\right) \right\rVert_T , \left\lVert \hat{L}^r\left(t \right) \right\rVert_T, 1 \right\} ,
		\end{align*}
		we obtain for every $\omega \in \mathcal{K}^r$,
		\begin{align*}
		\frac{\sup_{ \frac{M y_{r,0}}{\sqrt{r}} \leq t \leq T} \hat{g}(\hat{Q}^r(t)) }{ \max \left\{  \left\lVert p \sqrt{\frac{\lambda}{\mu}} \hat{Q}^r\left( t\right) \right\rVert_T, \left\lVert \hat{L}^r\left(t \right) \right\rVert_T, 1 \right\}} \leq \frac{\epsilon}{\sqrt{\lambda/\mu}}.
		\end{align*}
		Note that for every $t \geq 0$,
		\begin{align}
		\lvert\hat{L}^r(t)\rvert  = \left\lvert \sum_{j=1}^S \left( p_j \sqrt{\frac{\lambda}{\mu}} \hat{Q}_j^r(t) - \beta \sqrt{\frac{\lambda}{\mu}}\right)^+ \right\rvert \leq \sqrt{\frac{\lambda}{\mu}} \left( \lvert\hat{Q}^r(t)\rvert + S|\beta| \right).
		\label{eq:NoLNeeded}
		\end{align}
		Moreover, since $\epsilon>0$ is arbitrary, we can conclude that~\eqref{eq:MultiplicativeSSCInitallyDifferent} holds.

		If $|\hat{Q}^r(0)| \overset{\Prob}{\rightarrow} 0$, it follows from Corollary~\ref{cor:BoundedSSCfunctionDiffusion}, that for all $\omega \in \mathcal{L}^r$ and $t \in [0,M y_{r,0}/\sqrt{r}]$
		\begin{align*}
		\hat{g}(\hat{Q}^r(t)) \leq \frac{\epsilon}{\sqrt{\lambda/\mu}}y_{r,0} \leq \frac{\epsilon}{\sqrt{\lambda/\mu}} \max \left\{ \left\lVert p \sqrt{\frac{\lambda}{\mu}} \hat{Q}^r\left(t\right) \right\rVert_T, \left\lVert \hat{L}^r\left(t \right) \right\rVert_T, 1 \right\}.
		\end{align*}
		Since $\epsilon>0$ is arbitrary, together with~\eqref{eq:MultiplicativeSSCInitallyDifferent} and~\eqref{eq:NoLNeeded}, we obtain~\eqref{eq:MultiplicativeSSCInitallyTheSame}.\hfill\end{proof}
		
		\begin{remark}\normalfont
			Note that the bounds in Theorem~\ref{thm:MultiplicativeSSC} are obtained for every fixed $T>0$. Yet, from the proof it is clear that one has the following slightly more general result. Suppose $\hat{Q}^r(0) \rightarrow \hat{Q}(0)$ for some random vector $\hat{Q}(0)$. For every $\epsilon >0$, $M<\infty$ as in Theorem~\ref{thm:MultiplicativeSSC}, and $t_r \in (M y_{r,0}/\sqrt{r},\infty)$,
			\begin{align}
			\Prob\left( \frac{\sup_{M y_{r,0}/\sqrt{r} \leq s \leq t_r } \hat{g}(\hat{Q}^r(s)) }{ \max\{ \lVert \hat{Q}^r(s) \rVert_{t_r} , 1 \} } > \epsilon \right) \rightarrow 0 ,
			\label{eq:MultiplicativeSSCInitallyDifferentSlightlyMoreGeneral}
			\end{align}
			as $r \rightarrow \infty$. If, in addition,  $\hat{g}(\hat{Q}^r(0)) \rightarrow 0$ in probability as $r \rightarrow \infty$, then for every $t_r \in (M y_{r,0}/\sqrt{r},\infty)$,
			\begin{align}
			\frac{\lVert \hat{g}(\hat{Q}^r(s))\rVert_{t_r} }{ \max\{ \lVert\hat{Q}^r(t)\rVert_{t_r} ,  1 \} }  \overset{\Prob}{\rightarrow} 0 ,
			\label{eq:MultiplicativeSSCInitallyTheSameMoreGeneral}
			\end{align}
			as $r \rightarrow \infty$. In other words, the interval over which the state space collapse is considered can also be chosen as a sequence of intervals indexed by $r$.
		\end{remark}
		
		\subsection{Strong state space collapse}
		Although Theorem~\ref{thm:MultiplicativeSSC} shows multiplicative state space collapse, our interest lies in the strong state space collapse as is stated in Theorem~\ref{thm:StrongSSC}. To do so, it suffices to show that the denominators in Theorem~\ref{thm:MultiplicativeSSC} are bounded in a probabilistic sense. More specifically, $\lVert\hat{Q}^r(t)\rVert_T$ should satisfy the compact containment property. Before doing so, we prove a result that shows that even if the diffusion-scaled queue lengths are initially not necessarily close to one another, these queue lengths do not explode for a sufficiently short period of time.
		
		\begin{lemma}
			Suppose $\hat{Q}^r(0) \rightarrow \hat{Q}(0)$ for some random vector $\hat{Q}(0)$, and $M \in [0,\infty)$. Then,
			\begin{align*}
			\lim_{K \rightarrow \infty} \lim_{r\rightarrow \infty} \Prob\left(\lVert\hat{Q}^r(t)\rVert_{M y_{r,0} /\sqrt{r}} > K \right) =0.
			\end{align*}
			\label{lem:QueueLengthInitiallyBounded}
		\end{lemma}
		
		\noindent
\begin{proof}
		Fix $\epsilon \in (0,1)$ small. First, note that
		\begin{align}
		\begin{split}
		&\Prob\left(\lVert\hat{Q}^r(t)\rVert_{M y_{r,0} /\sqrt{r}} > K \right)\\
		&\hspace{1cm} \leq \Prob\left(\lVert\hat{Q}^r(t)\rVert_{M y_{r,0} /\sqrt{r}} > K ; \max\left\{ \lvert\hat{Q}^r(0)\rvert, y_{r,0} \right\} \leq \epsilon K \right) +  \Prob\left( \max\left\{ \lvert\hat{Q}^r(0)\rvert, y_{r,0} \right\} > \epsilon K \right).
		\end{split}
		\label{eq:SplittingTermsForSmallLemma}
		\end{align}
		By definition,
		\begin{align*}
		y_{r,0} = \max \left\{ \left\lvert p \sqrt{\frac{\lambda}{\mu}} \hat{Q}^r\left( 0\right) \right\rvert, \left\lvert \hat{L}^r\left(0 \right) \right\rvert, 1 \right\} .
		\end{align*}
		Since $\hat{Q}^r(0) \rightarrow \hat{Q}(0)$ for some random vector $\hat{Q}(0)$,
		\begin{align*}
		\lim_{K \rightarrow \infty} \lim_{r\rightarrow \infty} \Prob\left( \lvert\hat{Q}^r(0)\rvert > \epsilon K \right) =0,
		\end{align*}
		and due to the definition of $y_{r,0}$ and~\eqref{eq:NoLNeeded} for $t=0$, this implies
		\begin{align*}
		\lim_{K \rightarrow \infty} \lim_{r\rightarrow \infty} \Prob\left( \max\left\{ \lvert\hat{Q}^r(0)\rvert, y_{r,0} \right\} > \epsilon K \right) =0.
		\end{align*}
		
		To bound the first term in~\eqref{eq:SplittingTermsForSmallLemma} as well, we observe that the queue length at some time is trivially bounded by
		\begin{align*}
		|Q^r(t)| \leq |Q^r(0)|  + \max\left\{ \sum_{\{i,j\} \in E} \Lambda_{ij}(r t) ,  \max_{1\leq j \leq S} \{S_j(F_j^r t) \} \right\}.
		\end{align*}
		We observe that $F_j^r \leq (1+\epsilon) \lambda r/\mu$ for $r$ large enough. Moreover, if $\{\Lambda(t),t\geq 0\}$ denotes a Poisson process with rate $\lambda$, then due to the properties of the Poisson process it holds that $\sum_{\{i,j\} \in E} \Lambda_{ij}(\cdot) \overset{d}{=} \Lambda(\cdot)$. In terms of the diffusion scaling, the above bound yields for all $t \geq 0$,
		\begin{align*}
		|\hat{Q}^r(t)| \leq |\hat{Q}^r(0)|  + \frac{\max\left\{  \Lambda(r t) ,  \max_{1\leq j \leq S} \{S_j((1+\epsilon)\lambda r/\mu t) \} \right\}}{\min_{1 \leq j \leq S} p_j \sqrt{\lambda r / \mu}}.
		\end{align*}
		Therefore, using this bound for $t = M y_{r,0} / \sqrt{r} \leq \epsilon M K /\sqrt{r}$ and noting that Poisson processes are (non-decreasing) counting processes,
		\begin{align*}
		&\Prob\left(\lVert\hat{Q}^r(t)\rVert_{M y_{r,0} /\sqrt{r}} > K ; \max\left\{ \lvert\hat{Q}^r(0)\rvert, y_{r,0} \right\} \leq \epsilon K \right) \\
		&\hspace{1cm}\leq \Prob\left( \frac{\max\left\{  \Lambda(\epsilon M K \sqrt{r}) ,  \max_{1\leq j \leq S} \left\{S_j\left(\epsilon(1+\epsilon)\lambda/\mu M K \sqrt{r} \right) \right\} \right\}}{\min_{1 \leq j \leq S} p_j \sqrt{\lambda r / \mu}} > (1- \epsilon )K \right).
		\end{align*}
		Due to the LLN, we observe that
		\begin{align*}
		\lim_{K \rightarrow \infty} \lim_{r\rightarrow \infty} \Prob\left( \frac{  \Lambda(\epsilon M K \sqrt{r}) }{\min_{1 \leq j \leq S} p_j \sqrt{\lambda / \mu} K \sqrt{ r }} > (1- \epsilon ) \right)= 0
		\end{align*}
		for $\epsilon >0$ small enough (e.g. for $\epsilon < 1-M/(M+\sqrt{\lambda/\mu}\min_{1 \leq j \leq S} p_j)$). Similarly, due to the LLN,
		\begin{align*}
		\lim_{K \rightarrow \infty} \lim_{r\rightarrow \infty} \sum_{j=1}^S \Prob\left( \frac{  S_j\left(\epsilon(1+\epsilon)\lambda/\mu M K \sqrt{r} \right) }{\min_{1 \leq j \leq S} p_j \sqrt{\lambda / \mu} K \sqrt{ r }} > (1- \epsilon ) \right) =0
		\end{align*}
		for $\epsilon >0$ small enough (e.g. for $\epsilon < \min_{1 \leq j \leq S} p_j/(M\sqrt{\lambda/\mu}+\min_{1 \leq j \leq S} p_j)$). We conclude that the first term in~\eqref{eq:SplittingTermsForSmallLemma} converges to zero, i.e.
		\begin{align*}
		\lim_{K \rightarrow \infty} \lim_{r\rightarrow \infty} \Prob\left(\lVert\hat{Q}^r(t)\rVert_{M y_{r,0} /\sqrt{r}} > K ; \max\left\{ \lvert\hat{Q}^r(0)\rvert, y_{r,0} \right\} \leq \epsilon K \right) = 0,
		\end{align*}
		and hence the result follows.
\hfill\end{proof}
		
		Next, we show that the process $\hat{Q}^r(\cdot)$ satisfies the compact containment property.
		\begin{proposition}
			Suppose $|\hat{Q}^r(0)| \rightarrow \hat{Q}(0)$ for some random vector $\hat{Q}(0)$. Then, for every $T>0$ and $\epsilon >0$,
			\begin{align}
			\lim_{K \rightarrow \infty} \lim_{r \rightarrow \infty} \Prob\left(  \lVert\hat{Q}^r(t)\rVert_T > K \right)  =0.\label{eq:CompactContainmentProperty}
			\end{align}
			\label{prop:CompactContainmentProperty}
		\end{proposition}
		
\begin{proof}
		Fix $\epsilon \in (0,1/3)$ small, and  let $K > \max\{2|\beta|+2,2|\gamma|+2\}$. Introduce the sequence of stopping times
		\begin{align*}
		\hat{\tau}_K^r = \inf\left\{ t \geq 0 : \max_{1\leq j \leq S} \bar{Q}^r_j(t) > K \right\}, \hspace{1cm} \hat{T}_K^r = \sup\left\{ 0 \leq t \leq \hat{\tau}_k : \min_{1\leq j \leq S} \bar{Q}^r_j(t) \leq K/2 \right\},
		\end{align*}
		and similarly,
		\begin{align*}
		\breve{\tau}_K^r = \inf\left\{ t \geq 0 : \min_{1\leq j \leq S} \bar{Q}^r_j(t) < -K \right\}, \hspace{1cm} \breve{T}_K^r = \sup\left\{ 0 \leq t \leq \breve{\tau_k} : \max_{1\leq j \leq S} \bar{Q}^r_j(t) \geq -K/2 \right\}.
		\end{align*}
		Clearly,
		\begin{align}
		\Prob&\left( \lVert\hat{Q}^r(t)\rVert_T > K \right) \leq \Prob\left( \hat{\tau}_K^r \leq \breve{\tau}_K^r \leq T \right) + \Prob\left( \breve{\tau}_K^r \leq \hat{\tau}_K^r \leq T \right).
		\label{eq:SeperateTwoWaysOfCrossingBoundary}
		\end{align}
		
		In order to improve the readability of the proof, we first present a proof in the case when $\hat{g}(\hat{Q}^r(0)) \rightarrow 0$ in probability as $r \rightarrow \infty$. We then comment on the changes needed to adapt the proof to the case when this condition does not necessarily hold.
		\vspace{0.25cm}
		
		\noindent
		\textit{Case I: $\hat{g}(\hat{Q}^r(0)) \rightarrow 0$ in probability as $r \rightarrow \infty$.}
		
		Since $|\hat{Q}^r(0)| \rightarrow \hat{Q}(0)$ for some random vector $\hat{Q}(0)$, it holds that
		\begin{align*}
		\lim_{K \rightarrow \infty} \lim_{r \rightarrow \infty} \Prob\left( \lvert\hat{Q}^r(0)\rvert > K/2 \right) = 0,
		\end{align*}
		and hence we can assume that both $\hat{\tau}_K^r > \hat{T}_K^r >0$ and $\breve{\tau}_K^r > \breve{T}_K^r >0$ (for $K$ large enough).  Moreover, for every $t \leq \min\{\breve{\tau}_K^r,\hat{\tau}_K^r\}$ ,
		\begin{align}
		\frac{\Vert \hat{g}(\hat{Q}^r(t))\rVert_{\min\{\breve{\tau}_K^r,\hat{\tau}_K^r\}} }{ \max\{ \lVert\hat{Q}^r(t)\rVert_{\min\{\breve{\tau}_K^r,\hat{\tau}_K^r\}} , 1 \} } \leq \epsilon \hspace{0.5cm} \Rightarrow \hspace{0.5cm} \max_{1 \leq i \leq S} \hat{Q}_i^r(t) - \min_{1 \leq i \leq S} \hat{Q}_i^r(t) \leq \epsilon K.
		\label{eq:ImplicationContionalEvent}
		\end{align}
		Next, we provide bounds for the two ways of crossing the boundary $K$ separately. First, we consider the first term in~\eqref{eq:SeperateTwoWaysOfCrossingBoundary}. We observe
		\begin{align*}
		\Prob\left(\hat{\tau}_K^r \leq \breve{\tau}_K^r \leq T \right) &\leq \Prob\left( \hat{\tau}_K^r \leq \breve{\tau}_K^r \leq T \; ; \frac{\Vert \hat{g}(\hat{Q}^r(t))\rVert_{ \hat{\tau}_K^r} }{ \max\{ \lVert\hat{Q}^r(t)\rVert_{ \hat{\tau}_K^r} , 1 \} } \leq \epsilon \right) + \Prob\left( \frac{\Vert \hat{g}(\hat{Q}^r(t))\rVert_{ \hat{\tau}_K^r} }{ \max\{ \lVert\hat{Q}^r(t)\rVert_{ \hat{\tau}_K^r} , 1 \} } > \epsilon \right).
		\end{align*}
		Due to Theorem~\ref{thm:MultiplicativeSSC} and~\eqref{eq:MultiplicativeSSCInitallyTheSameMoreGeneral},
		\begin{align*}
		\lim_{r \rightarrow \infty}  \Prob\left( \frac{\Vert \hat{g}(\hat{Q}^r(t))\rVert_{ \hat{\tau}_K^r} }{ \max\{ \lVert\hat{Q}^r(t)\rVert_{ \hat{\tau}_K^r} , 1 \} } > \epsilon \right) =0.
		\end{align*}
		To bound the first term, define the process $\{\hat{Q}^r_\Sigma(t), t\geq 0\}$ with
		\begin{align*}
		\hat{Q}^r_\Sigma(t) = \frac{\sum_{j=1}^S \left(  Q_j^r(t) - p_j \lambda r / \mu \right)}{\sqrt{\lambda r /\mu}} = \sum_{j=1}^S p_j \hat{Q}^r_j(t).
		\end{align*}
		We observe that
		\begin{align}
		\min_{1 \leq j \leq S} \hat{Q}_j^r(t) \leq \hat{Q}^r_\Sigma(t) \leq \max_{1 \leq j \leq S} \hat{Q}_j^r(t).
		\label{eq:BoundsQSigma}
		\end{align}
		Due to the system identities, we observe that for every $t \in [\hat{T}_K^r, \hat{\tau}_K^r]$,
		\begin{align*}
		\hat{Q}^r_\Sigma(t) = \hat{Q}^r_\Sigma(\hat{T}_K^r) + \frac{\sum_{\{i,j\} \in E} A_{ij}^r(t) -A_{ij}^r(\hat{T}_K^r)}{\sqrt{\lambda r/\mu}} - \frac{\sum_{j=1}^S D_j^r(t) - D_j^r(\hat{T}_K^r)}{\sqrt{\lambda r/\mu}}.
		\end{align*}
		We note that due to the properties of the Poisson process,
		\begin{align*}
		\sum_{\{i,j\} \in E} A_{ij}^r(t) -A_{ij}^r(\hat{T}_K^r) \leq_{\textrm{ST}} \sum_{\{i,j\} \in E} \Lambda_{ij}^r(r t) -\Lambda_{ij}^r(r \hat{T}_K^r) \overset{d}{=} \Lambda(r t) - \Lambda(r \hat{T}_K^r)
		\end{align*}
		with $\{\Lambda(t),t\geq 0\}$ an (independent) Poisson process with rate $\lambda$. Moreover, since for all $t \in [\hat{T}_K^r, \hat{\tau}_K^r]$ it holds that $\hat{Q}^r_j(t) \geq \gamma$ for every $i \in \{1,\ldots,S\}$,
		\begin{align*}
		\sum_{j=1}^S D_j^r(t) - D_j^r(\hat{T}_K^r) = \sum_{j=1}^S S_j( F_j^r t) -  S_j( F_j^r \hat{T}_K^r)   .
		\end{align*}
		Using the FCLT, we observe that
		\begin{align*}
		\frac{\Lambda(r t) - \Lambda(r \hat{T}_K^r) - \sum_{j=1}^S S_j( F_j^r t) -  S_j( F_j^r \hat{T}_K^r) }{\sqrt{\lambda r/\mu}} \overset{d}{\rightarrow} \textrm{BM}(t) - \textrm{BM}(\hat{T}_K^r) -\gamma \mu (t -\hat{T}_K^r)
		\end{align*}
		as $r \rightarrow \infty$, where $\{\textrm{BM}(t), t\geq 0\}$ is a Brownian motion with zero mean and finite variance (independent of $K$). Finally, by the definitions of the stopping times, and in view of~\eqref{eq:ImplicationContionalEvent} and~\eqref{eq:BoundsQSigma}, for all $t \in [\hat{T}_K^r, \hat{\tau}_K^r]$,
		\begin{align*}
		\hat{Q}^r_\Sigma(\hat{\tau}_K^r) - \hat{Q}^r_\Sigma(\hat{T}_K^r) \geq (1-\epsilon)K - (1+\epsilon) K/2 = (1-3\epsilon) K/2.
		\end{align*}
		We conclude that
		\begin{align*}
		\lim_{r \rightarrow \infty} \Prob\left( \hat{\tau}_K^r \leq \breve{\tau}_K^r \leq T \; ; \frac{\Vert \hat{g}(\hat{Q}^r(t))\rVert_{ \hat{\tau}_K^r} }{ \max\{ \lVert\hat{Q}^r(t)\rVert_{ \hat{\tau}_K^r} , 1 \} } \leq \epsilon \right) \leq \Prob\left( \sup_{0 \leq s \leq t \leq T}  \textrm{BM}(t) - \textrm{BM}(s) -\gamma \mu (t -s) \geq \frac{1-3\epsilon}{2} K \right),
		\end{align*}
		which converges to zero as $K \rightarrow \infty$ since $\epsilon \in (0,1/3)$.
		
		The analysis of the second term in~\eqref{eq:SeperateTwoWaysOfCrossingBoundary} uses similar arguments as the first term. We observe
		\begin{align*}
		\Prob\left(\breve{\tau}_K^r \leq \hat{\tau}_K^r \leq T \right) &\leq \Prob\left( \breve{\tau}_K^r \leq \hat{\tau}_K^r \leq T \; ; \frac{\lVert \hat{g}(\hat{Q}^r(t))\rVert_{ \breve{\tau}_K^r} }{ \max\{ \lVert\hat{Q}^r(t)\rVert_{ \breve{\tau}_K^r} , 1 \} } \leq \epsilon \right) + \Prob\left( \frac{\Vert \hat{g}(\hat{Q}^r(t))\rVert_{ \breve{\tau}_K^r} }{ \max\{ \lVert\hat{Q}^r(t)\rVert_{ \breve{\tau}_K^r} , 1 \} } > \epsilon \right).
		\end{align*}
		Again, due to Theorem~\ref{thm:MultiplicativeSSC} and~\eqref{eq:MultiplicativeSSCInitallyTheSameMoreGeneral},
		\begin{align*}
		\lim_{r \rightarrow \infty}  \Prob\left( \frac{\Vert \breve{g}(\hat{Q}^r(t))\rVert_{ \breve{\tau}_K^r} }{ \max\{ \lVert\breve{Q}^r(t)\rVert_{ \breve{\tau}_K^r} , 1 \} } > \epsilon \right) =0.
		\end{align*}
		Due to the system identities, we observe that for every $t \in [\breve{T}_K^r, \breve{\tau}_K^r]$,
		\begin{align*}
		\hat{Q}^r_\Sigma(t) = \hat{Q}^r_\Sigma(\breve{T}_K^r) + \frac{\sum_{\{i,j\} \in E} A_{ij}^r(t) -A_{ij}^r(\breve{T}_K^r)}{\sqrt{\lambda r/\mu}} - \frac{\sum_{j=1}^S D_j^r(t) - D_j^r(\breve{T}_K^r)}{\sqrt{\lambda r/\mu}}.
		\end{align*}
		Due to the definitions of the stopping times, we observe that for all $t \in [\breve{T}_K^r, \breve{\tau}_K^r]$, it holds that $\hat{Q}_j^r(t) \leq \beta$ for every $j \in \{1,\ldots,S\}$, and hence $L^r(t)=r$. Therefore, due to the properties of the Poisson process,
		\begin{align*}
		\sum_{\{i,j\} \in E} A_{ij}^r(t) -A_{ij}^r(\breve{T}_K^r) = \sum_{\{i,j\} \in E} \Lambda_{ij}^r(r t) -\Lambda_{ij}^r(r \breve{T}_K^r) \overset{d}{=} \Lambda(r t) - \Lambda(r \breve{T}_K^r)
		\end{align*}
		with $\{\Lambda(t),t\geq 0\}$ a Poisson process with rate $\lambda$. Moreover,
		\begin{align*}
		\sum_{j=1}^S D_j^r(t) - D_j^r(\breve{T}_K^r) \leq_{\textrm{ST}} \sum_{j=1}^S S_j( F_j^r t) -  S_j( F_j^r \breve{T}_K^r)   .
		\end{align*}
		Using the FCLT, we observe again that
		\begin{align*}
		\frac{\Lambda(r t) - \Lambda(r \breve{T}_K^r) - \sum_{j=1}^S S_j( F_j^r t) -  S_j( F_j^r \breve{T}_K^r) }{\sqrt{\lambda r/\mu}} \overset{d}{\rightarrow} \textrm{BM}(t) - \textrm{BM}(\breve{T}_K^r) -\gamma \mu (t -\breve{T}_K^r)
		\end{align*}
		as $r \rightarrow \infty$, where we recall that $\{\textrm{BM}(t), t\geq 0\}$ is a Brownian motion with zero mean and finite variance (independent of $K$). Finally, by definition of the stopping times, and in view of~\eqref{eq:ImplicationContionalEvent} and~\eqref{eq:BoundsQSigma}, it holds for all $t \in [\breve{T}_K^r, \breve{\tau}_K^r]$,
		\begin{align*}
		\hat{Q}^r_\Sigma(\breve{\tau}_K^r) - \hat{Q}^r_\Sigma(\breve{T}_K^r) \leq -(1-\epsilon)K - (- (1+\epsilon) K/2) = - \frac{1-3\epsilon}{2} K.
		\end{align*}
		We conclude that
		\begin{align*}
		\lim_{r \rightarrow \infty}  \Prob\left( \breve{\tau}_K^r \leq \hat{\tau}_K^r \leq T \; ; \frac{\Vert \hat{g}(\hat{Q}^r(t))\rVert_{ \breve{\tau}_K^r} }{ \max\{ \lVert\hat{Q}^r(t)\rVert_{ \breve{\tau}_K^r} , 1 \} } \leq \epsilon \right)\leq \Prob\left( \sup_{0 \leq s \leq t \leq T}  \textrm{BM}(t) - \textrm{BM}(s) -\gamma \mu (t -s) \leq -\frac{1-3\epsilon}{2} K \right),
		\end{align*}
		which also converges to zero as $K \rightarrow \infty$ since $\epsilon \in (0,1/3)$. Since this holds for both of the two summed probabilities in~\eqref{eq:SeperateTwoWaysOfCrossingBoundary}, we conclude that~\eqref{eq:CompactContainmentProperty} holds.
		\vspace{0.25cm}
		
		\noindent
		\textit{Case II: general case, i.e.\ when we do not assume that~$\hat{g}(\hat{Q}^r(0)) \rightarrow 0$ in probability as $r \rightarrow \infty$.}
		
		Let $M \in [1,\infty)$ be fixed and satisfy the property as in~\ref{thm:MultiplicativeSSC}. Since $|\hat{Q}^r(0)| \rightarrow \hat{Q}(0)$ for some random vector $\hat{Q}(0)$ and due to Lemma~\ref{lem:QueueLengthInitiallyBounded}, it holds that
		\begin{align*}
		\lim_{K \rightarrow \infty} \lim_{r \rightarrow \infty} \Prob\left( \lVert\hat{Q}^r(t)\rVert_{M y_{r,0}/\sqrt{r}} > K/2 \right) = 0.
		\end{align*}
		Therefore, we can assume that both $\hat{\tau}_K^r > \hat{T}_K^r >  M y_{r,0}/\sqrt{r}$ and $\breve{\tau}_K^r > \breve{T}_K^r > M y_{r,0}/\sqrt{r} $ (for $K$ large enough). The proof in this general case is then completely analogous to that in the previous one: $\lVert \hat{g}(\hat{Q}^r(t))\rVert_{ \breve{\tau}_K^r}$ needs to be replaced with $\sup_{t \in ( M y_{r,0}/\sqrt{r}, \breve{\tau}_K^r]} \lvert \hat{g}(\hat{Q}^r(t))\rvert$, and $\lVert \hat{g}(\hat{Q}^r(t))\rVert_{ \hat{\tau}_K^r}$ -- with $\sup_{t \in ( M y_{r,0}/\sqrt{r}, \hat{\tau}_K^r]} \lvert \hat{g}(\hat{Q}^r(t))\rvert$.
\hfill\end{proof}
		
		\noindent
		Next, we prove our main result stated as in Theorem~\ref{thm:StrongSSC}.
		
\begin{proof}[Proof of Theorem~\ref{thm:StrongSSC}]
		Equation~\eqref{eq:StrongSSCInititiallyEqual} is a consequence of Theorem~\ref{thm:MultiplicativeSSC} and Proposition~\ref{prop:CompactContainmentProperty}. To prove~\eqref{eq:StrongSSCInititiallyNonequal}, note that for every $\epsilon >0$ and any sequence $\{K^r,r \in \mathbb{N}\}$,
		\begin{align*}
		\Prob\left( \sup_{K^r/\sqrt{r} \leq t \leq T } \hat{g}(\hat{Q}^r(t))  > \epsilon \right) &= \Prob\left( \sup_{K^r/\sqrt{r} \leq t \leq T } \hat{g}(\hat{Q}^r(t))  > \epsilon \; ;  K^r > M y_{r,0}  \right) + \Prob\left( K^r  \leq M y_{r,0} \right) \\
		&\leq \Prob\left( \sup_{M y_{r,0}/\sqrt{r} \leq t \leq T } \hat{g}(\hat{Q}^r(t))  > \epsilon   \right) + \Prob\left( K^r  \leq M y_{r,0} \right).
		\end{align*}
		Theorem~\ref{thm:MultiplicativeSSC} and Proposition~\ref{prop:CompactContainmentProperty} imply that for every $\epsilon >0$,
		\begin{align*}
		\lim_{r \rightarrow \infty} \Prob\left( \sup_{M y_{r,0}/\sqrt{r} \leq t \leq T } \hat{g}(\hat{Q}^r(t))  > \epsilon \right) = 0 ,
		\end{align*}
		Moreover, for any sequence $\{K^r,r \in \mathbb{N}\}$ for which $K^r=o(\sqrt{r})$ with $K^r \rightarrow \infty$ as $r \rightarrow \infty$, it holds that
		\begin{align*}
		\lim_{r \rightarrow \infty} \Prob\left( K^r  \leq M y_{r,0} \right) = 0,
		\end{align*}
		by definition of $y_{r,0}$,~\eqref{eq:NoLNeeded} and since $\hat{Q}^r(0) \rightarrow \hat{Q}(0)$ for some random vector $\hat{Q}(0)$. We conclude that~\eqref{eq:StrongSSCInititiallyNonequal} holds as well.
\hfill\end{proof}

\end{document}